\documentclass[11pt]{amsart}
\usepackage{amssymb,amsmath,amsthm,mathtools,mathrsfs}
\usepackage[T1]{fontenc}
\usepackage[utf8]{inputenc}
\usepackage{fourier}
\usepackage[english]{babel}
\usepackage{dsfont}
\usepackage{enumitem}
\usepackage{graphicx}
\usepackage{placeins}
\usepackage{float}
\usepackage[font=small]{caption}
\usepackage{tikz}
\usetikzlibrary{decorations.pathreplacing}
\usepackage[top=1.3in,bottom=1.3in,left=1.2in,right=1.2in]{geometry}
\usepackage{hyperref}
\hypersetup{colorlinks,linkcolor={red!65!black},citecolor={green!40!black},urlcolor={red!65!black},pdftitle={Poisson contingency tables, Gibbs permutations, and entropic optimal transport},pdfauthor={Danny Duan, Hanbaek Lyu, William Powell}}
\mathtoolsset{showonlyrefs}
\allowdisplaybreaks
\newif\ifshowchanges
\showchangestrue   
\newif\ifshownew
\shownewtrue       
\definecolor{revcolor}{rgb}{0.0,0.22,0.85}
\definecolor{newcolor}{rgb}{0.60,0.0,0.60}

\theoremstyle{plain}
\newtheorem{theorem}{Theorem}[section]
\newtheorem{lemma}[theorem]{Lemma}
\newtheorem{prop}[theorem]{Proposition}
\newtheorem{corollary}[theorem]{Corollary}
\theoremstyle{definition}
\newtheorem{definition}[theorem]{Definition}
\newtheorem{example}[theorem]{Example}
\newtheorem{remark}[theorem]{Remark}
\newtheorem{assumption}{Assumption}

\numberwithin{equation}{section}

\newcommand{\R}{\mathbb{R}}
\newcommand{\N}{\mathbb{N}}
\newcommand{\Z}{\mathbb{Z}}
\newcommand{\E}{\mathbb{E}}
\renewcommand{\P}{\mathbb{P}}
\newcommand{\1}{\mathds{1}}
\let\accentring\r
\let\accentcedilla\c
\renewcommand{\r}{\TextOrMath{\accentring}{\mathbf{r}}}
\renewcommand{\c}{\TextOrMath{\accentcedilla}{\mathbf{c}}}
\newcommand{\x}{\mathbf{x}}
\newcommand{\y}{\mathbf{y}}
\newcommand{\bLambda}{\boldsymbol{\Lambda}}
\newcommand{\balpha}{\boldsymbol{\alpha}}
\newcommand{\bbeta}{\boldsymbol{\beta}}
\newcommand{\bxi}{\boldsymbol{\xi}}
\newcommand{\eps}{\varepsilon}
\newcommand{\cA}{\mathcal{A}}
\newcommand{\cM}{\mathcal{M}}
\newcommand{\cS}{\mathcal{S}}
\newcommand{\cX}{\mathcal{X}}
\newcommand{\cY}{\mathcal{Y}}
\newcommand{\cZ}{\mathcal{Z}}
\newcommand{\cD}{\mathcal{D}}
\newcommand{\cH}{\mathcal{H}}
\newcommand{\cR}{\mathcal{R}}
\newcommand{\cL}{\mathcal{L}}
\newcommand{\cG}{\mathcal{G}}
\newcommand{\cK}{\mathcal{K}}
\newcommand{\cQ}{\mathcal{Q}}

\newcommand{\cC}{\mathcal{C}}
\newcommand{\cE}{\mathcal{E}}
\newcommand{\cU}{\mathcal{U}}
\newcommand{\cW}{\mathcal{W}}
\newcommand{\Tpoly}{\mathcal{T}}
\newcommand{\CT}{\mathrm{CT}}
\newcommand{\KL}{D_{\mathrm{KL}}}
\newcommand{\Ek}{\mathsf{E}}
\newcommand{\Gibbs}{\mathsf{P}}
\newcommand{\bD}{\mathbf{D}}
\newcommand{\bJ}{\mathbf{J}}
\newcommand{\FF}{\mathbb{F}}
\newcommand{\QQ}{\mathbb{Q}}
\newcommand{\Leb}{\mathrm{Leb}}

\DeclareMathOperator{\per}{per}
\DeclareMathOperator{\tr}{tr}
\DeclareMathOperator{\Var}{Var}
\DeclareMathOperator{\Cov}{Cov}
\DeclareMathOperator{\osc}{osc}

\DeclareMathOperator{\diag}{diag}
\DeclareMathOperator{\Lip}{Lip}
\DeclareMathOperator{\Ent}{Ent}
\DeclareMathOperator{\Poisson}{Poisson}
\DeclareMathOperator*{\argmin}{arg\,min}

\newcommand{\detF}{\operatorname{det}_{\mathrm{F}}}
\DeclarePairedDelimiter\norm{\lVert}{\rVert}

\title[Poisson contingency tables, Gibbs permutations, and entropic optimal transport]{Poisson contingency tables, Gibbs permutations, \\ and entropic optimal transport}

\author{Danny Duan}
\address{Danny Duan, Department of Mathematics, University of Wisconsin--Madison, Madison, WI 53706, USA}
\email{bduan5@wisc.edu}

\author{Hanbaek Lyu}
\address{Hanbaek Lyu, Department of Mathematics, University of Wisconsin--Madison, Madison, WI 53706, USA}
\email{hlyu@math.wisc.edu}

\author{William Powell}
\address{William Powell, Department of Mathematics, University of Wisconsin--Madison, Madison, WI 53706, USA}
\email{wgpowell@wisc.edu}


\begin{document}
	
	\begin{abstract}
		We study random contingency tables obtained by conditioning a matrix of independent Poisson entries on its row and column sums. We identify such a table with the block counts of a random permutation drawn from a Gibbs measure of Mallows type, an exact coupling whose normalizing constants recover Barvinok's permanent identity. We use this correspondence to study the table at three scales. For finite dimensions, a transference principle compares the table with a matrix of independent Poisson entries whose means form the static Schr\"odinger bridge, the solution of a discrete entropic optimal transport problem. As the dimensions grow and the margins and prior means converge, the normalized histogram of the table converges to a limiting Schr\"odinger bridge at a quantitative rate. The fluctuations of the table around the bridge are characterized in two complementary ways. Linear statistics satisfy a central limit theorem for bounded test functions, with limiting variance equal to the squared $L^2$-distance, under the bridge, of the test function to the subspace of additively separable functions; as a process they converge to white noise with intensity given by the bridge, conditioned to have vanishing marginals. Spectrally, the limiting spectral distribution of the centered table is governed by a Dyson equation whose variance profile is the density of the bridge. The key ingredient is an asymptotic formula for the partition function of the Gibbs measure in terms of a Fredholm determinant given by the canonical correlations of the bridge. For Mallows-type random permutations, our results prove Pal's conjecture on the limiting partition function and Mukherjee's conjecture on the central limit theorem for linear statistics, both in greater generality. 
	\end{abstract}
	
	\maketitle

	\section{Introduction}
	\label{sec: introduction}
	
	Let $m$ and $n$ be positive integers and let $(\r, \c) \in \N^{m} \times \N^{n}$ be positive integer vectors with total sum $N:=\sum_{i = 1}^{m} \r(i) = \sum_{j= 1}^{n} \c(j) $. The set
	\begin{align}
		\label{eq: def_CT}
		\Tpoly(\r, \c) := \left\{ \x \in \R_{\ge 0}^{m \times n} : \sum_{j= 1}^{n} \x_{ij} = \r(i), \quad \sum_{i = 1}^{m} \x_{ij} = \c(j)\right\}
	\end{align}
	of nonnegative matrices whose row and column sums are $\r$ and $\c$ is the \emph{transportation polytope}. Its integer points form the set $\CT(\r, \c) := \Tpoly(\r,\c) \cap \Z^{m\times n}$ of \emph{contingency tables} with margins $(\r,\c)$.
	
	\subsection{Poisson contingency tables and matrix scaling}
	\label{sec: intro_tables}
	
	In this paper, we consider the following model of a random contingency table. Let $\bLambda \in \R_{>0}^{m\times n}$,  which will be referred to as the \emph{prior mean} or \emph{reference}, and let $\Poisson(\bLambda)$ denote the distribution of a random matrix $\mathbf{X}^{\circ}$ with independent entries $\mathbf{X}^{\circ}_{ij} \sim \Poisson(\bLambda_{ij})$. Let $\lambda_{\r,\c;\bLambda}$ denote the law of the random contingency table $\mathbf{X}$ distributed as $\mathbf{X}^{\circ}$ conditioned to lie in $\CT(\r, \c)$, that is,
	\begin{align}
		\label{eq: conditioned_model}
		\lambda_{\r, \c; \bLambda}(\x) := \P\big(\mathbf{X}^{\circ} = \x \,\big|\, \mathbf{X}^{\circ} \in \CT(\r, \c)\big) = \frac{1}{\cZ_{\r,\c;\bLambda}} \prod_{i,j} \frac{\bLambda_{ij}^{\x_{ij}}}{\x_{ij}!}, \qquad \x \in \CT(\r,\c),
	\end{align}
	where $\cZ_{\r,\c;\bLambda} := \sum_{\y \in \CT(\r,\c)} \prod_{i,j} \bLambda_{ij}^{\y_{ij}}/\y_{ij}!$. If $\bLambda$ has rank one, then $\lambda_{\r, \c; \bLambda}$ is the multivariate hypergeometric (Fisher--Yates) distribution \cite{diaconis1995rectangular}
	\begin{align}
		\label{eq: fisher_yates_measure}
		\lambda_{\r, \c; \bLambda}(\x) = \frac{\prod_{i = 1}^{m} \r(i)! \prod_{j = 1}^{n} \c(j)!}{N! \prod_{i,j} \x_{ij}!},
	\end{align}
	the null distribution of the classical conditional tests of independence in two-way tables \cite{diaconis1985testing}. More generally, if $\log \bLambda_{ij} = \theta + a_i + b_j + \boldsymbol{\gamma}_{ij}$ is a Poisson log-linear model, then conditioning on the margins removes the main effects $(\theta, a, b)$, so that $\lambda_{\r,\c;\bLambda} = \lambda_{\r,\c;e^{\boldsymbol{\gamma}}}$ depends only on the interaction $\boldsymbol{\gamma}$ between rows and columns.

	For general $\bLambda$, we ask a fundamental question about this model, adding to an extensive line of research on random contingency tables \cite{barvinok2010does, barvinok2012matrices, dittmer2020phase, lyu2022number, wu2023asymptotic, lyu2024large}: what does a random contingency table $\mathbf{X} \sim \lambda_{\r, \c; \bLambda}$ look like? The answer is the solution of an entropic optimal transport (EOT) problem, the static Schr\"odinger bridge. This is perhaps unexpected. Entropic projections have long been used to \emph{fit} contingency tables to prescribed margins, through iterative proportional fitting \cite{deming1940least}, which is Sinkhorn's algorithm, and minimum discrimination information estimation \cite{ireland1968contingency}. However, the shape, fluctuations, and partition function of a \emph{random} table with prescribed margins have not been systematically connected to EOT. We show that they are governed by the static Schr\"odinger bridge, its potentials, and its canonical correlations.

	We will show that $\mathbf{X}$ is close to the solution $\mathbf{Z} = \mathbf{Z}^{\r,\c}$ of the following \emph{matrix scaling} problem:
	\begin{align}
		\label{eq: def_matrix_scaling}
		\mathbf{Z}^{\r,\c} := \argmin_{\x \in \Tpoly(\r, \c)} \sum_{i,j} \x_{ij}\log \frac{\x_{ij}}{\bLambda_{ij}}.
	\end{align}
	Matrix scaling is a classical problem \cite{deming1940least, sinkhorn1964relationship, sinkhorn1967concerning} with applications ranging from statistics to computational optimal transport; see \cite{idel2016review, cuturi2013sinkhorn, peyre2019computational}.

	Since the entries of $\bLambda$ are positive, \eqref{eq: def_matrix_scaling} has a unique solution, of the form
	\begin{align}
		\label{eq: typical_table_tilting}
		\mathbf{Z}^{\r,\c}_{ij} = \bLambda_{ij}e^{\balpha(i) + \bbeta(j)},
	\end{align}
	where $(\balpha, \bbeta) \in \R^{m} \times \R^{n}$ are known as \emph{Sinkhorn} or \emph{Schr\"odinger potentials}. They are unique up to transformations of the form $(\balpha, \bbeta) \mapsto (\balpha + t, \bbeta - t)$, $t \in \R$, and are the solutions of the dual problem
	\begin{align}
		\label{eq: kantorovich_dual}
		\sup_{\balpha \in \R^m, \bbeta\in \R^n} \Big( \langle \balpha,\r \rangle + \langle \bbeta, \c \rangle - \langle \exp(\balpha \oplus \bbeta), \bLambda \rangle \Big),
	\end{align}
	where $(\balpha\oplus\bbeta)_{ij} := \balpha(i) + \bbeta(j)$ and the exponential is taken entrywise. We call $\mathbf{Z}$ the \emph{static Schr\"odinger bridge} with margins $(\r,\c)$ and reference $\bLambda$ (see Section~\ref{sec: intro_EOT}). It is also the \emph{typical table} of the Poisson model in \cite{lyu2024large}. Statistically, $(\balpha,\bbeta)$ is the maximum likelihood estimator of the rank-one tilting model $\mathbf{Y}_{ij} \sim \Poisson(\bLambda_{ij}e^{\balpha(i)+\bbeta(j)})$ given the margins $(\r,\c)$ (Section~\ref{sec: bg_poisson}). The potentials can be computed efficiently by \emph{Sinkhorn's algorithm}, which alternately updates $\bbeta$ so that the column sums of $(\bLambda_{ij}e^{\balpha(i)+\bbeta(j)})_{i,j}$ equal $\c$, and $\balpha$ so that its row sums equal $\r$.
	
	One reason to expect $\mathbf{X}$ to be close to $\mathbf{Z}^{\r,\c}$ is that the latter is the mean of the distribution solving
	\begin{align}\label{eq:KL_min}
		\min_{\cH} \KL(\cH \Vert \cR) \quad \text{subject to } \E_{\cH}[\mathbf{Y}] \in \Tpoly(\r, \c),
	\end{align}
	where $\KL$ is the Kullback--Leibler divergence and the reference measure $\cR = \Poisson(\bLambda)$ is the law of the unconstrained prior. The minimizing measure $\cH$ is the law of an $m \times n$ matrix $\mathbf{Y}$ with independent entries $\mathbf{Y}_{ij} \sim \Poisson(\mathbf{Z}^{\r,\c}_{ij})$ (see Section~\ref{sec: bg_poisson}). Moreover, $\cH$ is proportional to $\Poisson(\bLambda)$ on $\CT(\r, \c)$, in the sense that
	\begin{align}
		\label{eq: same_conditional_law}
		\P\big(\mathbf{Y} = \x \,\big|\, \mathbf{Y} \in \CT(\r, \c)\big) = \P\big(\mathbf{X}^{\circ} = \x \,\big|\, \mathbf{X}^{\circ} \in \CT(\r, \c)\big), \qquad \x \in \CT(\r,\c).
	\end{align}
	Since $\E[\mathbf{Y}] \in \Tpoly(\r, \c)$, one expects $\P(\mathbf{Y} \in \CT(\r, \c))$ to be relatively large. Thus, $\mathbf{X} \sim \lambda_{\r, \c; \bLambda}$ should behave similarly to $\mathbf{Y} \sim \Poisson(\mathbf{Z}^{\r,\c})$.
	
	\subsection{Gibbs random permutations}
	\label{sec: intro_gibbs}
	
	A key observation of this work is that the Poisson contingency table is exactly a coarse-graining of a random permutation drawn from a Gibbs measure on the symmetric group $S_N$. To define this measure, for $k \in [N]$, let
	\begin{align}
		\label{eq: index_quantile}
		Q_{\r}(k) := \min \Big\{ i \in [m] :  \sum_{\tau = 1}^{i} \r(\tau) \geq k \Big\}
	\end{align}
	be the quantile function of the (unnormalized) measure on $[m]$ induced by $\r$, and define $Q_{\c}$ analogously. Let $\widetilde{\bLambda}$ be the $N \times N$ matrix with entries $\widetilde{\bLambda}_{k \ell} := \bLambda_{Q_{\r}(k), Q_{\c}(\ell)}$. It consists of $mn$ blocks: block $(i, j)$ has size $\r(i) \times \c(j)$, and all of its entries equal $\bLambda_{ij}$ (see Figure~\ref{fig: tilde_matrix}).
	
	\begin{figure}[t]
		\centering
		\begin{tikzpicture}[x=1.07cm, y=1.07cm, font=\small,
			blk/.style={draw=black!80, line width=0.9pt},
			grid/.style={draw=black!25, line width=0.3pt, step=1},
			pick/.style={draw=red!70!black, line width=1pt, fill=red!12},
			lab/.style={font=\footnotesize, text=black!75},
			brc/.style={decorate, decoration={brace, amplitude=3.5pt, raise=1pt}, black!65, line width=0.6pt}]
			\definecolor{tintA}{RGB}{226,234,244}
			\definecolor{tintB}{RGB}{246,246,246}
			\begin{scope}[shift={(0,-0.5)}]
				\foreach \i in {1,2,3}{\foreach \j in {1,2,3}{
						\pgfmathparse{int(mod(\i+\j,2))}
						\ifnum\pgfmathresult=0 \def\tint{tintA}\else\def\tint{tintB}\fi
						\fill[\tint] (\j-1,-\i+1) rectangle (\j,-\i);
				}}
				\draw[grid] (0,0) grid (3,-3);
				\draw[blk] (0,0) rectangle (3,-3);
				\foreach \i in {1,2,3}{\foreach \j in {1,2,3}{\node at (\j-0.5,-\i+0.5) {$\bLambda_{\i\j}$};}}
				\node[lab, below] at (1.5,-3.55) {(a) $\bLambda$, $m\times n$};
			\end{scope}
			\begin{scope}[shift={(5.3,0)}]
				\foreach \i/\ya/\yb in {1/0/2, 2/2/3, 3/3/4}{
					\foreach \j/\xa/\xb in {1/0/1, 2/1/3, 3/3/4}{
						\pgfmathparse{int(mod(\i+\j,2))}
						\ifnum\pgfmathresult=0 \def\tint{tintA}\else\def\tint{tintB}\fi
						\fill[\tint] (\xa,-\ya) rectangle (\xb,-\yb);
				}}
				\draw[grid] (0,0) grid (4,-4);
				\draw[blk] (0,-2) -- (4,-2) (0,-3) -- (4,-3) (1,0) -- (1,-4) (3,0) -- (3,-4);
				\draw[blk] (0,0) rectangle (4,-4);
				\foreach \k/\l in {1/2, 2/3, 3/4, 4/1}{
					\draw[pick] (\l-0.5,-\k+0.5) circle (0.36);
				}
				\foreach \k/\i in {1/1,2/1,3/2,4/3}{
					\foreach \l/\j in {1/1,2/2,3/2,4/3}{
						\node at (\l-0.5,-\k+0.5) {$\bLambda_{\i\j}$};
				}}
				\draw[brc] (-0.1,-2) -- (-0.1,0) node[midway, left=5pt, lab]{$\r(1)=2$};
				\draw[brc] (-0.1,-3) -- (-0.1,-2) node[midway, left=5pt, lab]{$\r(2)=1$};
				\draw[brc] (-0.1,-4) -- (-0.1,-3) node[midway, left=5pt, lab]{$\r(3)=1$};
				\draw[brc] (0,0.1) -- (1,0.1) node[midway, above=5pt, lab]{$\c(1)$};
				\draw[brc] (1,0.1) -- (3,0.1) node[midway, above=5pt, lab]{$\c(2)$};
				\draw[brc] (3,0.1) -- (4,0.1) node[midway, above=5pt, lab]{$\c(3)$};
				\node[lab, below] at (2,-4.05) {(b) $\widetilde{\bLambda}$, $N\times N$, and $\sigma$};
			\end{scope}
			\begin{scope}[shift={(10.7,-0.5)}]
				\foreach \i/\row in {1/{0,2,0}, 2/{0,0,1}, 3/{1,0,0}}{
					\foreach \x [count=\j] in \row {
						\pgfmathparse{int(mod(\i+\j,2))}
						\ifnum\pgfmathresult=0 \def\tint{tintA}\else\def\tint{tintB}\fi
						\fill[\tint] (\j-1,-\i+1) rectangle (\j,-\i);
				}}
				\draw[grid] (0,0) grid (3,-3);
				\draw[blk] (0,0) rectangle (3,-3);
				\foreach \i/\row in {1/{0,2,0}, 2/{0,0,1}, 3/{1,0,0}}{\foreach \x [count=\j] in \row {\node at (\j-0.5,-\i+0.5) {$\x$};}}
				\foreach \i/\s in {1/2,2/1,3/1}{\node[lab] at (3.45,-\i+0.5) {$\s$};}
				\foreach \j/\s in {1/1,2/2,3/1}{\node[lab] at (\j-0.5,-3.4) {$\s$};}
				\node[lab, below] at (1.5,-3.55) {(c) $\mathbf{X}(\sigma)$};
			\end{scope}
		\end{tikzpicture}
		\caption{Barvinok's block matrix for $m = n = 3$, $\r = (2,1,1)$ and $\c = (1,2,1)$, so that $N = 4$. (a)~The prior mean $\bLambda$. (b)~The $N\times N$ matrix $\widetilde{\bLambda}$, with its blocks in heavy lines. The circled entries are $(k,\sigma(k))$ for the permutation $\sigma = (2,3,4,1)$, which hits the $(1,2)$ block twice; its weight in \eqref{eq: gibbs_intro} is proportional to $\bLambda_{12}^2\bLambda_{23}\bLambda_{31}$. (c)~The block counts $\mathbf{X}(\sigma)$, the numbers of circled entries in the blocks, with row and column sums $\r$ and $\c$ (gray).}
		\label{fig: tilde_matrix}
	\end{figure}
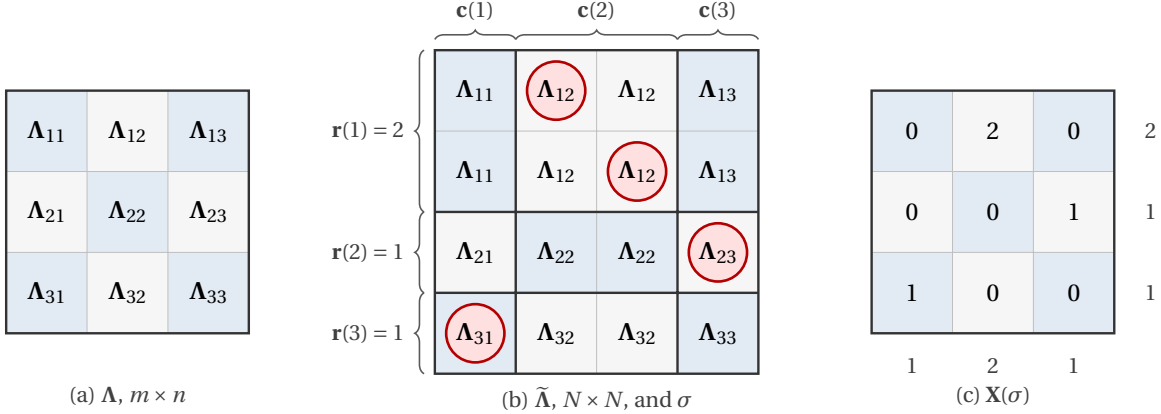
	
	The \emph{Gibbs measure} $\Gibbs_{\r,\c;\bLambda}$ associated with $(\r,\c,\bLambda)$ is the probability measure on $S_N$
	\begin{align}
		\label{eq: gibbs_intro}
		\Gibbs_{\r,\c;\bLambda}(\sigma) := \frac{1}{\per(\widetilde{\bLambda})}\prod_{k=1}^{N}\widetilde{\bLambda}_{k\sigma(k)}, \qquad \sigma \in S_N,
	\end{align}
	whose partition function $\per(\widetilde{\bLambda}) = \sum_{\sigma\in S_{N}} \prod_{k=1}^{N}\widetilde{\bLambda}_{k\sigma(k)}$ is the permanent of $\widetilde{\bLambda}$. For $\sigma \in S_N$, let $\mathbf{X}(\sigma)$ be the $m \times n$ matrix of block counts
	\begin{align}\label{eq:construction_Gibbs_X}
		\mathbf{X}_{ij}(\sigma) := \#\{k \in [N] : Q_{\r}(k) = i,\ Q_{\c}(\sigma(k)) = j\}.
	\end{align}
	Its row and column sums are $\r$ and $\c$, so $\mathbf{X}(\sigma) \in \CT(\r,\c)$. A direct counting argument gives an exact coupling of the Poisson contingency table and the Gibbs permutation (Proposition~\ref{prop: permutation_CT_coupling}):
	\begin{align}\label{eq:coupling_Gibbs}
		\mathbf{X}(\sigma) \sim \lambda_{\r,\c;\bLambda} \quad \textup{if} \quad \sigma \sim \Gibbs_{\r,\c;\bLambda}. 
	\end{align}
	Comparing normalizing constants recovers Barvinok's identity \eqref{eq: barvinok_per_identity} \cite{barvinok2008enumerating}.  
	
	For unit margins $\r = \c = \mathbf{1}$, the block-counting matrix $\mathbf{X}(\sigma)$ in \eqref{eq:construction_Gibbs_X} is exactly the permutation matrix of $\sigma$. If, in addition, $\bLambda_{ij} = e^{-f(i/N,j/N)}$ for a function $f$ on $[0,1]^2$, then $\Gibbs_{\mathbf{1},\mathbf{1};\bLambda}(\sigma) \propto \exp\big(-\sum_{i=1}^N f\big(\tfrac{i}{N},\tfrac{\sigma(i)}{N}\big)\big)$ is a Mallows-type random permutation \cite{mallows1957non, diaconis1988group, mukherjee2016estimation}. Two conjectures on these measures motivate this work. Pal's conjecture \cite{pal2024limitingconjecture} concerns the asymptotics of their partition function, the permanent of $(e^{-f(i/N,j/N)})_{i,j}$. Mukherjee's conjecture \cite{mukherjee2016estimation, mukherjee2023inference} asserts that, for every $\theta \in \R$, the statistic $\Upsilon_N(\sigma) := \sum_{i=1}^N f\big(\tfrac{i}{N},\tfrac{\sigma(i)}{N}\big)$ satisfies a central limit theorem when $\sigma$ is drawn from the Gibbs measure above with $f$ replaced by $-\theta f$, i.e., with probability proportional to $e^{\theta \Upsilon_N(\sigma)}$. We address both conjectures in a general form (Sections~\ref{sec: bg_pal} and \ref{sec: bg_mukherjee}).
	
	\subsection{Entropic optimal transport and Schr\"odinger bridges}
	\label{sec: intro_EOT}
	
	After normalization, the matrix scaling problem \eqref{eq: def_matrix_scaling} is a discrete instance of the \emph{static Schr\"odinger bridge} problem \cite{leonard2013survey, nutz2021introduction}. Fix probability measures $\mu$ and $\nu$ on spaces $\cX$ and $\cY$, respectively, and a probability measure $\cR$ on $\cX \times \cY$. The static Schr\"odinger bridge is the solution of
	\begin{align}
		\label{eq: def_SSB}
		\min_{\pi \in \Pi(\mu, \nu)} \KL(\pi \Vert \cR),
	\end{align}
	where $\Pi(\mu, \nu)$ is the set of couplings of $\mu$ and $\nu$. A sufficient condition for the existence of a solution to \eqref{eq: def_SSB} is that $\cR$ and $\mu \otimes \nu$ are mutually absolutely continuous and $\KL(\mu \otimes \nu \Vert \cR) < \infty$ (see \cite[Theorem~2.1]{nutz2021introduction}). In that case there is a function $\kappa: \cX \times \cY \to \R$ with $\frac{d\cR}{d(\mu \otimes \nu)} = e^{-\kappa}$, which we call a \emph{cost function}, since \eqref{eq: def_SSB} is then equivalent to the EOT problem
	\begin{align}
		\label{eq: EOT_intro}
		\min_{\pi \in \Pi(\mu, \nu)} \left[\int_{\cX \times \cY} \kappa \, d\pi + \KL(\pi \Vert \mu \otimes \nu)\right].
	\end{align}
	
	To compare contingency tables of growing dimensions, we embed a matrix $\mathbf{A} \in \R_{\ge 0}^{m\times n}$ as the probability measure $\pi_{\mathbf{A}}$ on $[0,1]^2$ whose density is proportional to $\mathbf{A}_{ij}$ on the cell $\big(\tfrac{i-1}{m}, \tfrac{i}{m}\big] \times \big(\tfrac{j-1}{n},\tfrac{j}{n}\big]$, and we embed the margins $\r$ and $\c$ as probability measures $\rho_{\r}$ and $\rho_{\c}$ on $[0,1]$ in the same way (Definition~\ref{def: histogram}). If $\mathbf{X} \in \CT(\r,\c)$, then $\pi_{\mathbf{X}}$ is a coupling of $\rho_{\r}$ and $\rho_{\c}$.
	
	Our coupling \eqref{eq:coupling_Gibbs} makes EOT appear naturally for Gibbs random permutations, such as the Mallows-type measures in Pal's and Mukherjee's conjectures. Indeed, the block counts of $\sigma \sim \Gibbs_{\r,\c;\bLambda}$ form the Poisson contingency table $\mathbf{X} \sim \lambda_{\r,\c;\bLambda}$, which is close to the Schr\"odinger bridge $\mathbf{Z}$. In turn, its histogram $\pi_{\mathbf{Z}}$ is the static Schr\"odinger bridge between $\rho_{\r}$ and $\rho_{\c}$ with reference measure $\pi_{\bLambda}$ (Lemma~\ref{lem: typical_table_bridge}), that is, the solution of the EOT problem \eqref{eq: EOT_intro} with these marginals and the corresponding cost induced by $\pi_{\bLambda}$. 
	
	\subsection{Overview of results}
	\label{sec: intro_overview}
	
	We describe the conditioned table $\mathbf{X}$, equivalently the Gibbs random permutation, at three scales. For finite dimensions, a transference principle (Lemma~\ref{thm: transference}) compares $\mathbf{X}$ with the matrix $\mathbf{Y} \sim \Poisson(\mathbf{Z})$ of independent entries. For sequences of problems with $m,n,N \to \infty$ (Section~\ref{sec: assumptions}), the histogram $\pi_{\mathbf{X}}$ converges weakly to the static Schr\"odinger bridge $\pi_\infty$ that solves the limiting EOT problem, at a quantitative rate (Theorems~\ref{thm: conditioned_model_weak_as_limit_no_rate} and \ref{thm: conditioned_model_weak_as_limit}). The fluctuations of $\mathbf{X}$ around $\mathbf{Z}$ are described in two complementary ways. Linear statistics $\sqrt N\int g\,d(\pi_{\mathbf{X}} - \pi_{\mathbf{Z}})$ of bounded test functions are asymptotically Gaussian with variance $\norm{(I-P)g}^2_{L^2(\pi_\infty)}$, where $P$ is the orthogonal projection onto additively separable functions, and as a process 
	$\sqrt{N}(\pi_{\mathbf{X}} - \pi_{\mathbf{Z}})$ converges to white noise with intensity $\pi_\infty$ conditioned to have vanishing marginals (Theorems~\ref{thm: CLT} and \ref{thm: functional_limits}). Spectrally, the limiting spectral distribution of the centered table is governed by a Dyson equation whose variance profile is the density of $\pi_\infty$ (Theorem~\ref{thm: ESD_limit}). The key ingredient for the fluctuations is the asymptotics of the partition function of the Gibbs measure: normalized by the EOT value, it converges to $\prod_{i}(1-s_i^2)^{-1/2}$, the inverse square root of a Fredholm determinant, where $s_1 \ge s_2 \ge \cdots$ are the canonical correlations of $\pi_\infty$ (Theorem~\ref{thm: permuton_partition_det_limit}). It enters through an exact identity linking the Laplace transform of the fluctuations to partition functions with perturbed costs (Lemma~\ref{lem: laplace_identity}), together with the second-order expansion of the EOT value in the cost (Proposition~\ref{thm: smooth_cost_perturbations}). For unit margins, these results prove Pal's conjecture on the limiting partition function under weaker assumptions and Mukherjee's conjecture on the central limit theorem for the sufficient statistic of an exponential family of Gibbs permutations (Corollary~\ref{cor: mallows}). We also obtain asymptotics for Poisson-weighted counts of contingency tables (Corollary~\ref{cor: enumeration}).
	
	\subsection{Organization}
	
	Section~\ref{sec: main_results} states our results precisely, introducing the necessary definitions along the way. Section~\ref{sec: background} discusses them in the context of the literature, including their relation to Pal's conjecture and Mukherjee's conjecture on Mallows-type random permutations. Section~\ref{sec: preliminaries} proves the permutation representation, Barvinok's identity and the transference principle. Section~\ref{sec: proof_LLN} recalls the relevant facts on static Schr\"odinger bridges and proves the law of large numbers, and Section~\ref{sec: proof_rates} proves the rates of convergence. Section~\ref{sec: partition_function} proves the partition function asymptotics, and Section~\ref{sec: proof_fluctuations} proves the fluctuation results, which build on 
	the results of Section~\ref{sec: partition_function} and on the second-order expansion of the EOT value in the cost; the latter is proved in Section~\ref{sec: smooth_cost}. The spectral results are proved in Section~\ref{sec: proof_ESD}.
	
	\subsection{Notation}
	
	We write $[m] := \{1, \dots, m\}$ and $(a \oplus b)(x,y) := a(x) + b(y)$ for functions $a$ and $b$. For a matrix $\mathbf{A}$ (and similarly for vectors), $\norm{\mathbf{A}}_1 := \sum_{i,j}|\mathbf{A}_{ij}|$ and $\norm{\mathbf{A}}_\infty := \max_{i,j}|\mathbf{A}_{ij}|$. For probability measures $\mu$ and $\nu$, $\KL(\mu\Vert\nu) := \int \log\frac{d\mu}{d\nu}\,d\mu$ if $\mu \ll \nu$ and $\KL(\mu \Vert \nu) := +\infty$ otherwise. We write $\Leb$ for the Lebesgue measure on $[0,1]$ and $u := \Leb \otimes \Leb$ for the Lebesgue (uniform) measure on $[0,1]^2$. We use the same symbol for an absolutely continuous measure on $[0,1]$ or $[0,1]^2$ and its density with respect to Lebesgue measure. For a bounded measurable function $f$, $\osc(f) := \operatorname{ess\,sup} f - \operatorname{ess\,inf} f$. We write $\Rightarrow$ for weak convergence of probability measures and for convergence in distribution. For a measurable map $T: \Omega \to \Omega'$ and a measure $\mu$ on $\Omega$, $T_{\#}\mu := \mu \circ T^{-1}$ is the push-forward of $\mu$ by $T$.  Asymptotic statements refer to a sequence of problems indexed by $k \to \infty$; we write $m = m_k$, $n = n_k$, $N = N_k$, and so on, and we often suppress the index $k$.
	
	\section{Statement of results}
	\label{sec: main_results}
	
	We state our results for a sequence of problems of growing size. 
	
	\subsection{Setting and assumptions}
	\label{sec: assumptions}
	\label{sec: main_histograms}
	
	We first make precise the embedding of margin vectors and matrices as probability measures described in Section~\ref{sec: intro_EOT}.
	
	\begin{definition}[Histograms]
		\label{def: histogram}
		Given a nonzero vector $\bxi \in \R_{\ge 0}^{n}$, its \emph{histogram} is the probability measure $\rho_{\bxi}$ on $[0,1]$ with density
		\begin{align}
			\rho_{\bxi}(x) := \frac{n}{\norm{\bxi}_1}\,\bxi(\lceil nx \rceil), \qquad x \in (0,1].
		\end{align}
		For a nonzero matrix $\mathbf{A} \in \R_{\ge 0}^{m \times n}$, its \emph{histogram} is the probability measure $\pi_{\mathbf{A}}$ on $[0, 1]^2$ with density
		\begin{align}
			\varphi_{\mathbf{A}}(x, y) :=  \frac{mn}{\norm{\mathbf{A}}_{1}} \mathbf{A}_{\lceil mx \rceil, \lceil ny \rceil}, \qquad (x,y) \in (0,1]^2.
		\end{align}
	\end{definition}
	
	We use two partitions of the unit square, the $m \times n$ grid of cells and the $N \times N$ grid of squares:
	\begin{align}
		\label{eq: def_partitions}
		A_{ij} := \Big( \tfrac{i-1}{m}, \tfrac{i}{m} \Big] \times \Big(\tfrac{j-1}{n}, \tfrac{j}{n} \Big], \qquad B_{k\ell} := \Big( \tfrac{k-1}{N}, \tfrac{k}{N} \Big] \times \Big(\tfrac{\ell-1}{N}, \tfrac{\ell}{N} \Big].
	\end{align}
	A function on $[0,1]^2$ is called \emph{grid-step} if it is constant on each cell $A_{ij}$. For instance, $\varphi_{\mathbf{A}}$ is grid-step for every $\mathbf{A} \in \R_{\ge 0}^{m\times n}$.
	
	For Borel probability measures $\mu$, $\nu$ on $[0,1]$ and a cost $c \in L^\infty(\mu \otimes \nu)$, we write
	\begin{align}
		\label{eq: EOT_Problem_General}
		\Gamma_{\mu,\nu}(c) := \min_{\pi \in \Pi(\mu, \nu)} \left[ \int_{[0,1]^2} c \, d\pi + \KL(\pi \Vert \mu \otimes \nu) \right]
	\end{align}
	for the value of the EOT problem \eqref{eq: EOT_intro} with cost $c$, and $\pi^{c}_{\mu,\nu}$ for its minimizer.
	
	We recall the following standard facts; see \cite[Section~4]{nutz2021introduction} and \cite{csiszar1975divergence}. The problem \eqref{eq: EOT_Problem_General} has a unique minimizer, and there exist \emph{Schr\"odinger potentials} $a \in L^\infty(\mu)$ and $b \in L^\infty(\nu)$, unique up to the transformation $(a,b)\mapsto(a+t,b-t)$ for $t \in \R$, such that
	\begin{align}
		\label{eq: schrodinger_form}
		\frac{d\pi^c_{\mu,\nu}}{d(\mu\otimes\nu)}(x,y) = e^{a(x) + b(y) - c(x,y)}.
	\end{align}
	The marginal constraints are equivalent to the \emph{Schr\"odinger system}
	\begin{align}
		\label{eq: schrodinger_system}
		\int e^{a(x) + b(y) - c(x,y)}\,\nu(dy) = 1 \quad \text{for } \mu\text{-a.e. } x, \qquad \int e^{a(x) + b(y) - c(x,y)}\,\mu(dx) = 1 \quad \text{for } \nu\text{-a.e. } y.
	\end{align}
	
	For margins $(\r,\c)$ and a prior mean $\bLambda$, the associated \emph{cost function} is the grid-step function
	\begin{align}
		\label{eq: def_kappa}
		\kappa(x, y) := - \log \left( \frac{\varphi_{\bLambda}(x, y)}{\rho_{\r}(x)\rho_{\c}(y)} \right) = \log\frac{\r(i)\c(j)\norm{\bLambda}_1}{N^2\bLambda_{ij}} \qquad \text{for } (x,y) \in A_{ij}.
	\end{align}
	Since $e^{-\kappa}\,d(\rho_{\r}\otimes\rho_{\c}) = d\pi_{\bLambda}$, for $\pi \in \Pi(\rho_{\r},\rho_{\c})$ we have
	\begin{align}
		\label{eq: EOT_is_SB}
		\int \kappa \, d\pi + \KL(\pi \Vert \rho_{\r}\otimes\rho_{\c}) = \KL(\pi \Vert \pi_{\bLambda}),
	\end{align}
	so the Schr\"odinger bridge of Section~\ref{sec: intro_EOT} is $\pi_{\mathbf{Z}} = \pi^{\kappa}_{\rho_{\r},\rho_{\c}}$ (Lemma~\ref{lem: typical_table_bridge}).
	
	We consider a sequence of problems indexed by $k \in \N$: dimensions $(m_k,n_k)$, positive integer margins $(\r_k, \c_k) \in \N^{m_k}\times\N^{n_k}$ with $\sum_i \r_k(i) = \sum_j \c_k(j) = N_k$, and prior means $\bLambda_k \in \R^{m_k\times n_k}_{>0}$. We write $\rho_{\r_k}$ and $\rho_{\c_k}$ for the histograms of the margins, $\kappa_k$ for the cost function \eqref{eq: def_kappa}, $\mathbf{Z}_k$ for the Schr\"odinger bridge, $\mathbf{X}_k \sim \lambda_{\r_k,\c_k;\bLambda_k}$, $\Gibbs_k := \Gibbs_{\r_k,\c_k;\bLambda_k}$ for the Gibbs measure \eqref{eq: gibbs_intro}, and
	\begin{align}\label{eq:EOT_value}
		\Gamma_k(f) := \Gamma_{\rho_{\r_k},\rho_{\c_k}}(f), \qquad \pi_k(f) := \pi^{f}_{\rho_{\r_k},\rho_{\c_k}}
	\end{align}
	for the value and the minimizer of the EOT problem with the discrete margins and a bounded cost $f$. We will use the following assumptions.
	
	\begin{assumption}
		\label{assumption: bounded_cost}
		The cost functions $\kappa_k$ converge uniformly on $[0,1]^2$ to a continuous function $\kappa_{\infty}$.
	\end{assumption}
	
	\begin{assumption}
		\label{assumption: margins_weak_convergence}
		We have $m_k, n_k \to \infty$ and $m_k/n_k \to \gamma \in (0, \infty)$. Moreover, there exist probability measures $\rho_{\r_{\infty}}, \rho_{\c_{\infty}}$ on $[0,1]$ such that $\rho_{\r_k} \to \rho_{\r_{\infty}}$ and $\rho_{\c_k} \to \rho_{\c_{\infty}}$ weakly.
	\end{assumption}
	
	\begin{assumption}
		\label{assumption: L^1_limits_margins}
		There are positive probability densities $\rho_{\r_{\infty}}$, $\rho_{\c_{\infty}}$ on $[0,1]$ and $\varphi_{\bLambda_{\infty}}$ on $[0,1]^2$ with
		\begin{align}
			\rho_{\r_k}\rightarrow \rho_{\r_{\infty}} \ \text{ in } L^{1}[0, 1], \qquad \rho_{\c_k}\rightarrow \rho_{\c_{\infty}} \ \text{ in } L^{1}[0, 1], \qquad \varphi_{\bLambda_{k}}\rightarrow \varphi_{\bLambda_{\infty}} \ \text{ in } L^{1}([0, 1]^2).
		\end{align}
	\end{assumption}

	Assumption~\ref{assumption: bounded_cost} says that the prior means, normalized by the margins as in \eqref{eq: def_kappa}, converge to a limiting cost; in particular, $\sup_k \norm{\kappa_k}_\infty < \infty$. Continuity of $\kappa_\infty$ is needed because the margins converge only weakly: if $\kappa_\infty$ jumped across a set charged by the limiting margins, the cost seen by that mass could depend on the approximating sequence. Under Assumption~\ref{assumption: margins_weak_convergence}, $N_k \ge \max(m_k,n_k) \to \infty$, since the margins are positive integers, and the mesh of the grid $(A_{ij})$ tends to zero. Only weak convergence of the margin profiles is required: row sums alternating between $N_k/(2m_k)$ and $3N_k/(2m_k)$ give $\rho_{\r_k} \to \Leb$ weakly, although $\norm{\rho_{\r_k} - \Leb}_{L^1} = 1/2$ for all $k$, and the limit may have atoms, as when a single row carries a fixed fraction of the total mass. Assumption~\ref{assumption: L^1_limits_margins} strengthens this to $L^1$ convergence of the margins and of the histograms of the prior means; it is used only for the rates of convergence (Theorem~\ref{thm: conditioned_model_weak_as_limit}). Given margins satisfying Assumption~\ref{assumption: margins_weak_convergence}, discretizing a continuous cost gives prior means satisfying Assumptions~\ref{assumption: bounded_cost} and, under a mild condition, \ref{assumption: L^1_limits_margins}:
	
	\begin{example}
		\label{ex: continuous_kappa}
		Let $(\r_k,\c_k)$ satisfy Assumption~\ref{assumption: margins_weak_convergence}, and let $\kappa$ be continuous on $[0,1]^2$, normalized so that $\int e^{-\kappa}\,d(\rho_{\r_\infty}\otimes\rho_{\c_\infty}) = 1$. Set $\bLambda_{k,ij} := e^{-\kappa(i/m_k,\,j/n_k)}\,\r_k(i)\c_k(j)$. By the uniform continuity of $e^{-\kappa}$ and the weak convergence of $\rho_{\r_k}\otimes\rho_{\c_k}$, $\norm{\bLambda_k}_1/N_k^2 = N_k^{-2}\sum_{i,j} e^{-\kappa(i/m_k,\, j/n_k)} \r_k(i)\c_k(j) \to 1$. By \eqref{eq: def_kappa}, $\kappa_k(x,y) = \kappa\big(\lceil m_kx\rceil/m_k, \lceil n_ky\rceil/n_k\big) - \log\big(N_k^2/\norm{\bLambda_k}_1\big)$, so $\kappa_k \to \kappa$ uniformly and Assumption~\ref{assumption: bounded_cost} holds with $\kappa_\infty = \kappa$. If, moreover, the margins discretize densities bounded away from zero and infinity, for instance $\r_k(i) \approx N_k\int_{(i-1)/m_k}^{i/m_k}\rho_{\r_\infty}$ up to rounding and similarly for $\c_k$, with $N_k/m_k, N_k/n_k \to \infty$, then Assumption~\ref{assumption: L^1_limits_margins} holds as well. If $\kappa(x,y) = a(x) + b(y)$, then $\bLambda_k$ has rank one (Corollary~\ref{cor: product_case}). \hfill $\blacktriangle$
	\end{example}

	The limiting objects are defined in terms of the limiting EOT problem
	\begin{align}
		\label{eq: continous_EOT_problem}
		\Gamma_\infty := \min_{\pi \in \Pi(\rho_{\r_{\infty}}, \rho_{\c_{\infty}})} \left[ \int_{[0, 1]^2} \kappa_{\infty} \, d\pi + \KL(\pi \Vert \rho_{\r_{\infty}} \otimes \rho_{\c_{\infty}}) \right].
	\end{align}
	We write $\pi_\infty := \pi^{\kappa_\infty}_{\rho_{\r_\infty},\rho_{\c_\infty}}$ for its minimizer.

	\subsection{Scaling limit}
	\label{sec: main_LLN}
	\label{sec: permutons_prelim}
	
	The first-order behavior is most naturally described in permuton coordinates for the Gibbs measure. We use the theory of permutation limits \cite{hoppen2013permuton}, inspired by graph limits \cite{lovasz2012large}. A \emph{permuton} is a Borel probability measure on $[0,1]^2$ with uniform marginals; the set $\cM$ of permutons is compact in the weak topology. For a permutation $\sigma \in S_N$, the associated permuton is the measure $\mu_{\sigma}$ with density
	\begin{align}
		\frac{d\mu_\sigma}{du}(x,y) := N \sum_{k=1}^{N} \1_{B_{k\sigma(k)}}(x,y) = N\,\1\big(\sigma(\lceil Nx \rceil) = \lceil Ny \rceil\big),
	\end{align}
	that is, $\mu_\sigma$ is uniform on the union of the squares $B_{k \sigma(k)}$.
	
	For a probability measure $\rho$ on $[0,1]$ with distribution function $F_\rho(t) := \rho([0,t])$, let
	\begin{align}
		\label{eq: histogram_quantiles}
		q_{\rho}(s) := \inf \big\{ t \in [0,1] : F_\rho(t) \ge s \big\}, \qquad s \in [0,1],
	\end{align}
	be its quantile function, so that $q_\rho(0) = 0$. Then $q_\rho(s) \le t \iff s \le F_\rho(t)$, and $F_\rho(q_\rho(s)) \ge s$, with equality if $\rho$ has no atoms \cite[Lemma~21.1]{vandervaart1998asymptotic} (the cases $s \in \{0,1\}$ being immediate). Hence, $q_\rho$ is nondecreasing, and in particular Borel, and $(q_\rho)_{\#}\Leb = \rho$. For probability measures $\rho_1, \rho_2$ on $[0,1]$ we set $T := (q_{\rho_1}, q_{\rho_2})$, so that $T_{\#}u = \rho_1 \otimes \rho_2$. When $\rho_1 = \rho_{\r}$ and $\rho_2 = \rho_{\c}$ are histograms of positive vectors, their distribution functions $F_{\r}, F_{\c}$ are continuous, strictly increasing and piecewise linear, and $T$ is a homeomorphism of $[0,1]^2$ with inverse $(F_{\r}, F_{\c})$; we write $q_{\r} := q_{\rho_{\r}}$ and $q_{\c} := q_{\rho_{\c}}$, so that $T = (q_{\r}, q_{\c})$.
	
	For a bounded measurable function $f$ on $[0,1]^2$, let
	\begin{align}
		\label{eq: def_partition_function}
		\cZ(f) := \sum_{\sigma \in S_N} \exp \left( -N \int_{[0,1]^2} f \circ T\, d\mu_{\sigma} \right)
	\end{align}
	be the partition function with cost $f$. In permuton coordinates, the Gibbs measure $\Gibbs := \Gibbs_{\r,\c;\bLambda}$ in \eqref{eq: gibbs_intro} takes the form
	\begin{align}
		\label{eq: permutation_measure_continuous}
		\Gibbs(\sigma) = \frac{1}{\cZ(\kappa)}\exp\left( -N \int_{[0,1]^2} \kappa \circ T\, d\mu_{\sigma}\right), \qquad \sigma \in S_N,
	\end{align}
	by Lemma~\ref{lem: gibbs_permuton_form}. In these coordinates, Proposition~\ref{prop: permutation_CT_coupling} compares the histogram $\pi_{\mathbf{X}(\sigma)}$ of the block counts \eqref{eq:construction_Gibbs_X} with the push-forward $T_{\#}\mu_\sigma$ of the permuton of $\sigma$. For the sequence of problems, we write $T_k := (q_{\r_k}, q_{\c_k})$, $T_\infty := (q_{\rho_{\r_\infty}}, q_{\rho_{\c_\infty}})$ and $\cZ_k$ for the partition function \eqref{eq: def_partition_function} of the $k$th problem.
	
	By \eqref{eq:coupling_Gibbs}, $\mathbf{X}_k \overset{d}{=} \mathbf{X}(\sigma)$ with $\sigma \sim \Gibbs_k$, and by Proposition~\ref{prop: permutation_CT_coupling}(iii), $\int g\,d\pi_{\mathbf{X}(\sigma)}$ is close to $\int g\circ T_k\,d\mu_\sigma$. The Gibbs weight of $\sigma$ is $\exp(-N_k\int\kappa_k\circ T_k\,d\mu_\sigma)$. On the other hand, by the large deviation principle for uniformly random permutations \cite{trashorras2008large, mukherjee2016estimation}, the number of permutations whose permuton is close to a given permuton $\mu$ is roughly $N_k!\exp(-N_k\KL(\mu\Vert u))$. Both effects are exponential in $N_k$, so $\Gibbs_k$ concentrates near the minimizer of $\mu \mapsto \int\kappa_\infty\circ T_\infty\,d\mu + \KL(\mu\Vert u)$ over permutons. This is the limiting problem \eqref{eq: continous_EOT_problem} written in permuton coordinates, and $T_\infty$ maps its minimizer to $\pi_\infty$. Varadhan's lemma turns this heuristic into the following law of large numbers, which identifies the static Schr\"odinger bridge $\pi_\infty$ as the scaling limit of the conditioned table.
	
	\begin{theorem}[Scaling limit]
		\label{thm: conditioned_model_weak_as_limit_no_rate}
		Assume \ref{assumption: bounded_cost} and \ref{assumption: margins_weak_convergence}.
		\begin{description}
			\item[(i)] The log partition function of the Gibbs measure \eqref{eq: permutation_measure_continuous} converges to the limiting EOT value:
			\begin{align}
				\label{eq: log_parition_limit}
				\lim_{k \to \infty} \frac{1}{N_k} \log \frac{\cZ_k(\kappa_k)}{N_k!}
				=  -\Gamma_{\infty}.
			\end{align}
			\item[(ii)] The random measure $\pi_{\mathbf{X}_k}$ converges weakly in probability to $\pi_{\infty}$, that is, $\int g\,d\pi_{\mathbf{X}_k} \to \int g \,d\pi_\infty$ in probability for every $g \in C([0,1]^2)$.
		\end{description}
	\end{theorem}
	
	Theorem~\ref{thm: conditioned_model_weak_as_limit_no_rate} gives no rate; under the stronger Assumption~\ref{assumption: L^1_limits_margins}, we obtain one for every bounded test function. Under Assumptions~\ref{assumption: bounded_cost} and \ref{assumption: L^1_limits_margins}, passing to an almost everywhere convergent subsequence shows that $\kappa_\infty = -\log\big(\varphi_{\bLambda_\infty}/(\rho_{\r_\infty}\otimes\rho_{\c_\infty})\big)$ almost everywhere, so that $\pi_\infty$ is the static Schr\"odinger bridge between $\rho_{\r_\infty}$ and $\rho_{\c_\infty}$ with reference measure $\varphi_{\bLambda_\infty}\,du$. The error $\int g\,d(\pi_{\mathbf{X}_k} - \pi_\infty)$ splits into a random part, the deviation of $\mathbf{X}_k$ from the Schr\"odinger bridge $\mathbf{Z}_k$, and a deterministic part, the distance between the discrete Schr\"odinger bridge $\pi_{\mathbf{Z}_k}$ and its limit $\pi_\infty$. To state the result, we set $M := \exp\big(\sup_k \norm{\kappa_k}_\infty\big) < \infty$. For probability densities $p$ and $q$, let $d_{H}(p,q) := \big( \int (\sqrt{p} - \sqrt{q})^2 \big)^{1/2}$ denote the Hellinger distance.
	
	\begin{theorem}[Rates of convergence]
		\label{thm: conditioned_model_weak_as_limit}
		Assume \ref{assumption: bounded_cost} and \ref{assumption: L^1_limits_margins}, that $m_k + n_k \to \infty$, and that there is $c \in (0,1]$ such that $\norm{\bLambda_k}_1 \ge c\, m_kn_k\norm{\bLambda_k}_\infty$ for all $k$. Fix $\eps > 0$ and $g \in L^\infty([0,1]^2)$, and suppose that $N_k \geq (m_k + n_k)^{1 + \eps}\log N_k$ for all large $k$. Then there is a constant $C>0$, depending only on $c$, $M$ and $\norm{g}_{\infty}$, such that for all sufficiently large $k$, with probability at least $1 - 2\exp \big( -C(m_k + n_k)^{1+\eps} \log N_k\big)$,
		\begin{multline}
			\label{eq: rate_bound}
			\left| \int_{[0, 1]^2} g \, d(\pi_{\mathbf{X}_k} - \pi_{\infty}) \right| \le \sqrt{\frac{(m_k+n_k)^{1+\eps}\log N_k}{N_k}} \\
			+ 2\norm{g}_\infty \Big( 40 \log M \big( \norm{\rho_{\r_k} - \rho_{\r_\infty}}_{L^1} + \norm{\rho_{\c_k} - \rho_{\c_\infty}}_{L^1} \big) + 11 M^5 d_{H}(\varphi_{\bLambda_k}, \varphi_{\bLambda_{\infty}}) \Big)^{1/2}.
		\end{multline}
	\end{theorem}
	
	The two terms in \eqref{eq: rate_bound} correspond to the random and the deterministic part of the error; the second tends to zero under Assumption~\ref{assumption: L^1_limits_margins}, and the first if $N_k/((m_k+n_k)^{1+\eps}\log N_k) \to \infty$. The first term is not sharp: the fluctuation results of Section~\ref{sec: main_fluctuations} show that $\int g\,d(\pi_{\mathbf{X}_k} - \pi_{\mathbf{Z}_k})$ is of order $N_k^{-1/2}$, and this holds uniformly over all axis-parallel rectangles (Corollary~\ref{cor: KS}).
	
	\subsection{Partition function asymptotics}
	\label{sec: main_partition}
	
	Theorem~\ref{thm: conditioned_model_weak_as_limit_no_rate}(i) identifies the exponential order of the partition function. Our next result identifies the subexponential correction. For a bounded grid-step function $f$ we write
	\begin{align}
		\label{eq: def_normalized_partition}
		\cD_k(f) := \frac{e^{N_k\Gamma_k(f)}}{N_k!}\,\cZ_k(f)
	\end{align}
	for the partition function \eqref{eq: def_partition_function} normalized by the corresponding discrete EOT value \eqref{eq:EOT_value}. 
	
	\begin{theorem}[Partition function asymptotics]
		\label{thm: permuton_partition_det_limit}
		Suppose that $N_k \to \infty$, $\rho_{\r_k} \to \rho_{\r_\infty}$ and $\rho_{\c_k}\to\rho_{\c_\infty}$ weakly. Let $f_k$ be grid-step functions with respect to the $m_k \times n_k$ grid that converge uniformly to a continuous function $f_\infty$ on $[0,1]^2$. Let $\xi_\infty := d\pi^{f_\infty}/d(\rho_{\r_\infty}\otimes\rho_{\c_\infty})$, where $\pi^{f_\infty} := \pi^{f_\infty}_{\rho_{\r_\infty},\rho_{\c_\infty}}$, and let $\cS : L^2(\rho_{\c_\infty}) \to L^2(\rho_{\r_\infty})$ be the integral operator
		\begin{align}
			\label{eq: def_S_operator}
			(\cS h)(x) := \int \big(\xi_\infty(x,y) - 1\big) h(y)\, \rho_{\c_\infty}(dy).
		\end{align}
		Then $\cS$ is a Hilbert--Schmidt operator with $\norm{\cS}_{\mathrm{HS}}^2 = \chi^2\big(\pi^{f_\infty} \,\Vert\, \rho_{\r_\infty}\otimes\rho_{\c_\infty}\big) := \int(\xi_\infty - 1)^2\,d(\rho_{\r_\infty}\otimes\rho_{\c_\infty})$ and $\norm{\cS}_{\mathrm{op}} \le 1 - e^{-2\osc(f_\infty)}$, and
		\begin{align}
			\label{eq: permuton_partition_limit_statement}
			\lim_{k \to \infty} \cD_k(f_k) = \detF(I - \cS^*\cS)^{-1/2} = \prod_{i \ge 1} (1 - s_i^2)^{-1/2},
		\end{align}
		where $s_1 \ge s_2 \ge \cdots \ge 0$ are the singular values of $\cS$, i.e., the canonical correlations of $\pi^{f_\infty}$, and $\detF$ is the Fredholm determinant.
	\end{theorem}

	The \emph{Poisson-weighted number of contingency tables} is the normalizing constant of \eqref{eq: conditioned_model},
	\begin{align}
		\label{eq: poisson_weighted_count}
		\cZ_{\r,\c;\bLambda} = \sum_{\x\in\CT(\r,\c)}\prod_{i,j}\frac{\bLambda_{ij}^{\x_{ij}}}{\x_{ij}!} = e^{\norm{\bLambda}_1}\,\P\big(\mathbf{X}^{\circ}\in\CT(\r,\c)\big),\qquad \mathbf{X}^{\circ}\sim\Poisson(\bLambda).
	\end{align}
	For $\bLambda \equiv 1$, each table $\x$ is counted with weight $\prod_{i,j}1/\x_{ij}!$. This quantity is hard to compute: it is a sum over all of $\CT(\r,\c)$, whose cardinality grows exponentially with the dimensions, and by Barvinok's identity (see \eqref{eq: barvinok_per_identity}) it is the permanent of the $N\times N$ matrix $\widetilde{\bLambda}$, up to an explicit factor. As an approximation of $\cZ_{\r,\c;\bLambda}$, consider the explicit quantity
	\begin{align}
		\label{eq: def_enumeration_main_term}
		\widehat{\cZ}_{\r,\c;\bLambda} := \Big(\frac{\norm{\bLambda}_1}{N}\Big)^{N}e^{-N\Gamma_{\rho_{\r},\rho_{\c}}(\kappa)}\;\frac{N!}{N^{N}}\prod_{i=1}^{m}\frac{\r(i)^{\r(i)}}{\r(i)!}\prod_{j=1}^{n}\frac{\c(j)^{\c(j)}}{\c(j)!},
	\end{align}
	whose only nontrivial ingredient is the discrete EOT value, equivalently the value $N\Gamma_{\rho_{\r},\rho_{\c}}(\kappa) - N\log(\norm{\bLambda}_1/N)$ of the matrix scaling problem \eqref{eq: def_matrix_scaling}, which Sinkhorn's algorithm computes.
	
	\begin{corollary}[Asymptotic enumeration]
		\label{cor: enumeration}
		Under Assumptions~\ref{assumption: bounded_cost} and \ref{assumption: margins_weak_convergence},
		\begin{align}
			\label{eq: enumeration_identity}
			\lim_{k\to\infty}\frac{\cZ_{\r_k,\c_k;\bLambda_k}}{\widehat{\cZ}_{\r_k,\c_k;\bLambda_k}} = \detF(I-\cS^*\cS)^{-1/2} =: D_\infty,
		\end{align}
		where $\cS$ is the operator of Theorem~\ref{thm: permuton_partition_det_limit} with $f_\infty = \kappa_\infty$.
	\end{corollary}
	
	Thus the main term $\widehat{\cZ}$ captures $\cZ$ up to a factor that depends only on the limiting Schr\"odinger bridge, through its canonical correlations. This parallels the asymptotic formula of Barvinok and Hartigan for the number of contingency tables \cite{barvinok2012asymptotic}, in which the main term is given by the typical table of the finite problem and the correction by a Gaussian determinant. Unlike their formula, Corollary~\ref{cor: enumeration} requires no tameness condition on the margins; see Section~\ref{sec: bg_tables}.

	\subsection{Fluctuations}
	\label{sec: main_fluctuations}
	
	We now describe the fluctuations of $\pi_{\mathbf{X}_k}$ around the static Schr\"odinger bridge $\pi_{\mathbf{Z}_k}$, at scale $N_k^{-1/2}$. We center at $\pi_{\mathbf{Z}_k}$ rather than $\pi_\infty$, because $\pi_{\mathbf{Z}_k} - \pi_\infty$ only reflects how fast the data converge and may be much larger than $N_k^{-1/2}$. We record the fluctuations as a linear functional $\Phi_k$ of test functions and as a two-parameter process $\FF_k$, the analogue of the two-dimensional empirical process. Both are expressed through cell averages.
	
	\begin{definition}[Cell averages]
		\label{def: cell_average}
		For $g \in L^1([0,1]^2)$, the \emph{cell average} of $g$ is the grid-step function
		\begin{align}
			\Ek g := \sum_{i,j} \mathbf{G}_{ij}\,\1_{A_{ij}}, \qquad \mathbf{G}_{ij} := mn\int_{A_{ij}} g(x,y)\,dx\,dy.
		\end{align}
	\end{definition}
	\noindent For nonzero $\mathbf{A} \in \R_{\ge 0}^{m\times n}$, integrating $g$ against $\pi_{\mathbf{A}}$ only involves the integrals of $g$ over the cells:
	\begin{align}
		\label{eq: cell_average_invariance}
		\int_{[0,1]^2} g \, d\pi_{\mathbf{A}} = \int_{[0,1]^2} \Ek g \, d\pi_{\mathbf{A}} = \frac{1}{\norm{\mathbf{A}}_1}\sum_{i,j} \mathbf{G}_{ij}\mathbf{A}_{ij} \qquad \text{for all } g \in L^1([0,1]^2).
	\end{align}
	In particular, linear statistics of $\pi_{\mathbf{X}}$ are linear combinations of the entries of $\mathbf{X}$.
	\begin{definition}[Fluctuation process]
		\label{def: fluctuation_fields}
		For $g \in L^\infty([0,1]^2)$ and $(u,v) \in [0,1]^2$ set
		\begin{align}
			\label{eq: def_fluctuation_field}
			\Phi_k(g) := \sqrt{N_k}\int_{[0,1]^2} g\,d(\pi_{\mathbf{X}_k} - \pi_{\mathbf{Z}_k}) = \frac{1}{\sqrt{N_k}} \sum_{i,j} \mathbf{G}_{ij}\,(\mathbf{X}_{k,ij} - \mathbf{Z}_{k,ij}), \qquad \FF_k(u,v) := \Phi_k\big(\1_{[0,u]\times[0,v]}\big),
		\end{align}
		where $\mathbf{G}_{ij} = m_kn_k\int_{A_{ij}} g(x,y)\,dx\,dy$ as in Definition~\ref{def: cell_average}.
	\end{definition}
	
	By \eqref{eq: cell_average_invariance}, $\Phi_k(g) = \Phi_k(\Ek_k g)$, where $\Ek_k$ is the cell-average operator of the $m_k \times n_k$ grid. We identify $\Phi_k$ with the random function $\sqrt{N_k}(\varphi_{\mathbf{X}_k} - \varphi_{\mathbf{Z}_k}) \in L^2([0,1]^2)$, so that $\Phi_k(g) = \langle \Phi_k, g\rangle_{L^2}$. For a rectangle $B = (u_1,u_2]\times(v_1,v_2]$ and a function $F$ on $[0,1]^2$ we write $F(B) := F(u_2,v_2) - F(u_1,v_2) - F(u_2,v_1) + F(u_1,v_1)$, so that $\FF_k(B) = \Phi_k(\1_B)$.
	
	Two observations suggest the limit. In the tilted model $\mathbf{Y}_k \sim \Poisson(\mathbf{Z}_k)$, whose entries are independent, the statistic $N_k^{-1/2}\sum_{i,j}\mathbf{G}_{ij}(\mathbf{Y}_{k,ij} - \mathbf{Z}_{k,ij})$ has variance $\int(\Ek_kg)^2\,d\pi_{\mathbf{Z}_k} \approx \int g^2\,d\pi_\infty$, as for white noise with intensity $\pi_\infty$. But $\mathbf{X}_k$ and $\mathbf{Z}_k$ have the same margins, so $\Phi_k(a\oplus b) = 0$ for all bounded $a$ and $b$. As the conditioned model is the tilted model conditioned on its margins, the limit of $\Phi_k$ should be white noise with intensity $\pi_\infty$ conditioned to have vanishing marginals, with variance the squared distance from $g$ to the additively separable functions.
	
	To define this distance, note that the space $\cA = \{a\oplus b : a \in L^2(\rho_{\r_\infty}),\ b \in L^2(\rho_{\c_\infty})\}$ of additively separable functions (Definition~\ref{def: sum_space}) is closed in $L^2(\pi_\infty)$ for arbitrary limiting margins, possibly with atoms, by Lemmas~\ref{lem: density_bounds} and \ref{lem: sum_space}. With $P := P_{\pi_\infty}$ the projection onto $\cA$ in $L^2(\pi_\infty)$, let
	\begin{align}
		\label{eq: def_limiting_variance}
		V(g) := \norm{(I - P)g}^2_{L^2(\pi_\infty)} = \min_{a \oplus b \in \cA} \int (g - a\oplus b)^2 \, d\pi_\infty, \qquad g \in L^2(\pi_\infty).
	\end{align}

	\begin{definition}[Schr\"odinger-bridge white noise and pillow]
		\label{def: SB_white_noise}
		The \emph{Schr\"odinger-bridge white noise} is the centered Gaussian process $\Xi = \{\Xi(g)\}_{g \in L^2(\pi_\infty)}$ with covariance
		\begin{align}\label{eq:covariance_mx_Xi}
			\Cov\big(\Xi(g), \Xi(h)\big) = \big\langle (I-P_{\pi_{\infty}})g, (I-P_{\pi_{\infty}})h\big\rangle_{L^2(\pi_\infty)}.
		\end{align}
		The \emph{Schr\"odinger-bridge pillow} is the two-parameter Gaussian process $\FF(u,v) := \Xi(\1_{[0,u]\times[0,v]})$, $(u,v) \in [0,1]^2$, the restriction of $\Xi$ to indicators of rectangles anchored at the origin.
	\end{definition}
	
	Our next result confirms this picture; its proof rests on Section~\ref{sec: main_partition}.
	
	\begin{theorem}[Central limit theorem]
		\label{thm: CLT}
		Assume \ref{assumption: bounded_cost} and \ref{assumption: margins_weak_convergence}, and suppose that either (i) $g, g_1, \ldots, g_p$ are bounded Borel functions on $[0,1]^2$ that are continuous at $(\rho_{\r_\infty}\otimes\rho_{\c_\infty})$-almost every point, or (ii) $g, g_1, \ldots, g_p \in L^\infty([0,1]^2)$ and the families $\{\rho_{\r_k}\}_k$ and $\{\rho_{\c_k}\}_k$ are uniformly integrable (for instance, uniformly bounded). Then $\Phi_k(g) \Rightarrow \mathcal{N}(0, V(g))$ and, more generally,
		\begin{align}
			\big(\Phi_k(g_1), \ldots, \Phi_k(g_p)\big) \Longrightarrow \big(\Xi(g_1), \ldots, \Xi(g_p)\big),
		\end{align}
		with convergence of all moments. The same holds if each $g_i$ is replaced by bounded Borel functions $g_{i,k}$ with $\norm{g_{i,k} - g_i}_\infty \to 0$.
	\end{theorem}
	
	In case~(ii), the limiting margins are absolutely continuous by Lemma~\ref{lem: transported_cell_averages}(b), so that $V$ and $\Xi$ are defined on $L^\infty([0,1]^2)$.

	The variance $V(g)$ concerns the component of $g$ complementary to the one in the central limit theorems for discrete Schr\"odinger bridges with random margins of Harchaoui, Liu and Pal \cite{pal2024chaos} and Gonz\'alez-Sanz, Loubes and Niles-Weed \cite{gonzalezsanz2024weaklimits}. There, the margins are empirical measures of $N$ i.i.d.\ samples. In \cite{pal2024chaos}, the plan is averaged over the couplings,  while \cite{gonzalezsanz2024weaklimits} solves the empirical Schr\"odinger bridge problem directly. In either case 
	the fluctuations come from the margins; at order $N^{-1/2}$, they are driven by the additively separable part $Pg$, while $(I-P)g$ contributes only at order $N^{-1}$. Here, the margins are fixed and the fluctuations come from the randomness of the table within the transportation polytope,  or equivalently of the permutation; they are driven by $(I-P)g$ alone. See Section~\ref{sec: bg_random_margins} for a precise comparison.

	Theorem~\ref{thm: CLT} concerns finitely many test functions at a time. For a functional limit theorem, we regard $\Phi_k$ as a random distribution, and $\Xi$ through its distributional version (Definition~\ref{def: Xi_hat} below), since, like white noise, $\Xi$ is too rough to be a random function. Let $\mathrm{c}_0 :\equiv 1$, $\mathrm{c}_p(x) := \sqrt{2}\cos(p\pi x)$ for $p \ge 1$, and $e_j(x,y) := \mathrm{c}_{j_1}(x)\mathrm{c}_{j_2}(y)$ for $j = (j_1,j_2)\in \Z_{\ge 0}^2$. The \emph{tensor-product cosine basis} $(e_j)_{j\in\Z_{\ge 0}^2}$, the eigenfunctions of the Neumann Laplacian on $[0,1]^2$, is an orthonormal basis of $L^2([0,1]^2)$ with $\norm{e_j}_\infty \le 2$; write $\hat f_j := \langle f, e_j\rangle_{L^2}$ for $f \in L^2([0,1]^2)$. Test functions range over the space $H^\infty$ of $g \in L^2([0,1]^2)$ whose coefficients decay faster than any power, i.e., $\sup_j(1+|j|)^q|\hat g_j| < \infty$ for every $q \ge 0$, where $|j| := (j_1^2+j_2^2)^{1/2}$, with the topology in which $g_m \to g$ if $\sup_j(1+|j|)^q|\hat g_{m,j} - \hat g_j| \to 0$ for every $q$. Since $\sum_j|\hat g_j| < \infty$, every $g\in H^\infty$ agrees almost everywhere with the continuous function $\sum_{j} \hat{g}_{j}e_{j}$, which we use throughout, as $\pi_\infty$ may charge Lebesgue-null sets. 
	
	Let $H^{-\infty}$ denote the space of continuous linear functionals on $H^{\infty}$. It consists of formal series $f = \sum_j \hat f_je_j$ with $|\hat f_j| \le C(1+|j|)^q$ for some $C$ and $q$, acting by
	\begin{align}
		\label{eq: sobolev_pairing}
		\langle f, g\rangle := \sum_{j} \hat f_j\,\hat g_j,\qquad f\in H^{-\infty},\ g\in H^\infty. 
	\end{align}
	By Parseval's identity, \eqref{eq: sobolev_pairing} is $\int fg$ for $f\in L^2([0,1]^2)$, so $L^2([0,1]^2) \subset H^{-\infty}$; in particular, $\Phi_k \in H^{-\infty}$. We equip $H^{-\infty}$ with the strong dual topology: $f_m \to f$ if $\sup_{g\in G}|\langle f_m - f, g\rangle| \to 0$ for every bounded set $G \subset H^\infty$, that is, every $G$ with $\sup_{g\in G}\sup_j(1+|j|)^q|\hat g_j| < \infty$ for each $q$.
	
	\begin{definition}[Distributional version of $\Xi$]
		\label{def: Xi_hat}
		A \emph{distributional version} of $\Xi$ is a random element $\hat\Xi$ of $H^{-\infty}$ such that $\langle\hat\Xi, g\rangle = \Xi(g)$ almost surely for every $g \in H^\infty$.
	\end{definition}

	The next theorem shows that $\Phi_k$ converges as a random distribution: the convergence in Theorem~\ref{thm: CLT} holds jointly for all test functions in $H^\infty$ and uniformly over bounded families of them, so that continuous functionals of the whole fluctuation process, not only finitely many linear statistics, converge in law. Part~(ii) is the analogue of Donsker's theorem for the two-dimensional empirical process. It requires atomless limiting margins, as otherwise $\FF$ has no continuous version in general.
	
	\begin{theorem}[Functional limit theorems]
		\label{thm: functional_limits}
		Assume \ref{assumption: bounded_cost} and \ref{assumption: margins_weak_convergence}.
		\begin{enumerate}[label=(\roman*)]
			\item A distributional version $\hat\Xi$ of $\Xi$ exists and is almost surely unique, and $\Phi_k \Rightarrow \hat\Xi$ in $H^{-\infty}$ with its strong dual topology.
			\item If, in addition, $\rho_{\r_\infty}$ and $\rho_{\c_\infty}$ have no atoms, then the Schr\"odinger-bridge pillow $\FF$ has a continuous version, still denoted by $\FF$, and $\FF_k \Rightarrow \FF$ in $C([0,1]^2)$ with the uniform topology.
		\end{enumerate}
	\end{theorem}
	
	As for the classical empirical process, Theorem~\ref{thm: functional_limits} yields limit laws for goodness-of-fit statistics, here comparing the conditioned table with its static Schr\"odinger bridge. 
	
	\begin{corollary}[Kolmogorov--Smirnov and Cram\'er--von Mises statistics]
		\label{cor: KS}
		Under Assumptions~\ref{assumption: bounded_cost} and \ref{assumption: margins_weak_convergence}, $\sup_B |\pi_{\mathbf{X}_k}(B) - \pi_{\mathbf{Z}_k}(B)| = O_{\P}(N_k^{-1/2})$, where the supremum is over axis-parallel rectangles $B \subset [0,1]^2$. If, in addition, $\rho_{\r_\infty}$ and $\rho_{\c_\infty}$ have no atoms, then 
		\begin{align}
			\sqrt{N_k}\sup_{B}\big|\pi_{\mathbf{X}_k}(B) - \pi_{\mathbf{Z}_k}(B)\big| \Longrightarrow \sup_{B}|\FF(B)|, \qquad \int_{[0,1]^2}\FF_k(u,v)^2\,du\,dv \Longrightarrow \int_{[0,1]^2}\FF(u,v)^2\,du\,dv,
		\end{align}
		where $\int\FF_k^2 = N_k\int_{[0,1]^2}(\pi_{\mathbf{X}_k} - \pi_{\mathbf{Z}_k})([0,u]\times[0,v])^2\,du\,dv$ is the Cram\'er--von Mises statistic.
	\end{corollary}
	
	Finally, all limiting objects are explicit when the limiting Schr\"odinger bridge is a product measure, as for the Fisher--Yates distribution. This holds, for instance, if $\bLambda_k$ has rank one for every $k$, in which case $\mathbf{X}_k$ has the Fisher--Yates distribution \eqref{eq: fisher_yates_measure} and $\Gibbs_k$ is the uniform distribution on $S_{N_k}$.
	
	\begin{corollary}[Independence case]
		\label{cor: product_case}
		Assume \ref{assumption: bounded_cost} and \ref{assumption: margins_weak_convergence}, and suppose that $\kappa_\infty \in \cA$, i.e., $\kappa_\infty(x,y) = a(x) + b(y)$.  Then $\pi_\infty = \rho_{\r_\infty}\otimes\rho_{\c_\infty}$, $D_\infty = 1$, and $V(g)$ is given by 
		\begin{align}
			\int \Big(g(x,y) - \int g(x,y')\,\rho_{\c_\infty}(dy') - \int g(x',y)\,\rho_{\r_\infty}(dx') + \int g \, d(\rho_{\r_\infty}\otimes\rho_{\c_\infty})\Big)^2 \rho_{\r_\infty}(dx)\,\rho_{\c_\infty}(dy).
		\end{align}
		Moreover, $\FF$ has the finite-dimensional distributions of $\big(\mathbb{B}(R(u), C(v))\big)_{u,v}$, where $\mathbb{B}$ is the standard Brownian pillow, the continuous centered Gaussian process on $[0,1]^2$ with covariance $(s\wedge s' - ss')(t\wedge t'-tt')$, and $R$ and $C$ are the distribution functions of $\rho_{\r_\infty}$ and $\rho_{\c_\infty}$; if $R$ and $C$ are continuous, the continuous version of $\FF$ has the same law.
	\end{corollary}
	
	The results of this section are discussed further in Sections~\ref{sec: bg_fluctuations} and \ref{sec: bg_mukherjee}.

	\subsection{Mallows-type random permutations}
	\label{sec: main_mallows}
	
	For unit margins, by \eqref{eq:coupling_Gibbs}, $\mathbf{X} \sim \lambda_{\mathbf{1},\mathbf{1};\bLambda}$ is the permutation matrix of $\sigma \sim \Gibbs_{\mathbf{1},\mathbf{1};\bLambda}$, and for $\bLambda_{ij} = e^{-f(i/N,j/N)}$, $\Gibbs_{\mathbf{1},\mathbf{1};\bLambda}$ is the Gibbs measure
	\begin{align}
		\label{eq: mallows_measure}
		\Gibbs^f_N(\sigma) := \frac{1}{N!\,L_N(f)}\exp\Big(-\sum_{i=1}^N f\big(\tfrac{i}{N},\tfrac{\sigma(i)}{N}\big)\Big), \qquad L_N(f) := \frac{1}{N!}\sum_{\sigma\in S_N}\exp\Big(-\sum_{i=1}^N f\big(\tfrac{i}{N},\tfrac{\sigma(i)}{N}\big)\Big),
	\end{align}
	which belongs to the Mallows models of random permutations \cite{mallows1957non} and includes, for instance, the models based on Spearman's footrule and Spearman's rank correlation \cite{mukherjee2016estimation, pal2024limitingconjecture}. We write $\Gamma^{\mathrm{u}} := \Gamma_{\Leb,\Leb}$ and $f_N(x,y) := f\big(\lceil Nx\rceil/N, \lceil Ny \rceil/N\big)$.
	
	Let $f \in C([0,1]^2)$, let $W := d\pi^f_{\Leb,\Leb}/du$ be the continuous density of the Schr\"odinger bridge with cost $f$ and uniform marginals, and let $\cS$ be the integral operator on $L^2[0,1]$ with kernel $W - 1$.  The following are consequences of Theorems~\ref{thm: permuton_partition_det_limit} and \ref{thm: CLT} specialized to unit margins. 
	
	\begin{corollary}[Mallows-type permutations]
		\label{cor: mallows}
		Let $f$, $W$ and $\cS$ be as above. Then Assumptions~\ref{assumption: bounded_cost} and \ref{assumption: margins_weak_convergence} hold for the unit margins and the prior means $\bLambda_N := (e^{-f(i/N,j/N)})_{i,j\in[N]}$, with $\pi_\infty = \pi^f_{\Leb,\Leb}$ and $\pi_{\mathbf{Z}_N} = \pi^{f_N}_{\Leb,\Leb}$. Moreover:
		\begin{enumerate}[label=(\roman*)]
			\item \emph{(Partition function.)}
			\begin{align}
				\label{eq: mallows_partition}
				\lim_{N\to\infty} e^{N\Gamma^{\mathrm{u}}(f_N)}L_N(f) = \detF(I - \cS^*\cS)^{-1/2}.
			\end{align}
			For $f \in C^1([0,1]^2)$ and Schr\"odinger potentials $(a,b)$ of $\pi^f_{\Leb,\Leb}$, defined on $[0,1]$ by \eqref{eq: schrodinger_system},
			\begin{align}
				\label{eq: mallows_pal}
				\lim_{N\to\infty} e^{N\Gamma^{\mathrm{u}}(f)}L_N(f) = e^{-\frac12[a(1) - a(0) + b(1) - b(0)]}\,\detF(I - \cS^*\cS)^{-1/2},
			\end{align}
			and
			\begin{align}
				\label{eq: mallows_tripathi}
				\lim_{N\to\infty}\frac{1}{N!}\sum_{\sigma\in S_N}\prod_{i=1}^N W\big(\tfrac{i}{N},\tfrac{\sigma(i)}{N}\big) = \detF(I - \cS^*\cS)^{-1/2}.
			\end{align}
			\item \emph{(Linear statistics.)} Let $\sigma_N \sim \Gibbs^f_N$. If $\mathbf{g}=(g_{1},\dots,g_{L}):[0,1]^{2}\rightarrow \R^{L}$ is continuous and $\Upsilon_N(\mathbf{g}) := \sum_{i=1}^N \mathbf{g}\big(\tfrac{i}{N},\tfrac{\sigma_N(i)}{N}\big)$, then
			\begin{align}
				\label{eq: mallows_CLT_continuous}
				N^{-1/2} \left( \Upsilon_N(\mathbf{g}) - \E\,\Upsilon_N(\mathbf{g}) \right) \Longrightarrow \mathcal{N}\big(\mathbf{0}, V(\mathbf{g})\big) \qquad\text{and}\qquad N^{-1} \Cov\big(\Upsilon_N(\mathbf{g})\big) \longrightarrow V(\mathbf{g}),
			\end{align}
			where $V(\mathbf{g})$ is the $L\times L$ matrix with entries $V(\mathbf{g})_{\ell\ell'}=\big\langle (I-P)g_{\ell}, (I-P)g_{\ell'}\big\rangle_{L^2(\pi_\infty)}$.
			If $f$ and $\mathbf{g}$ are Lipschitz, then
			\begin{align}
				\label{eq: mallows_CLT}
				\sqrt{N}\Big(\frac{\Upsilon_N(\mathbf{g})}{N} - \int W \mathbf{g}\,du\Big) \Longrightarrow \mathcal{N}\big(\mathbf{0}, V(\mathbf{g})\big).
			\end{align}
		\end{enumerate}
	\end{corollary}
	
	If $f \in C^1([0,1]^2)$ is symmetric and $f(1-x,1-y) = f(x,y)$, then the potentials can be chosen with $b = a$ and $a(0) = a(1)$, so that \eqref{eq: mallows_pal} becomes $\lim_{N} e^{N\Gamma^{\mathrm{u}}(f)}L_N(f) = \detF(I - \cS^2)^{-1/2}$. This is Pal's conjecture \cite{pal2024limitingconjecture}, which was proved by Tripathi \cite{tripathi2026partitionfunctionclassmallows} for $C^2$ costs in the form \eqref{eq: mallows_tripathi}.
	
	Part~(ii) proves Mukherjee's conjecture for the exponential family of Gibbs permutations $\Gibbs^{-\boldsymbol{\theta}^{\top} \mathbf{f}}_N$, $\boldsymbol{\theta}\in\R^L$, where $\mathbf{f} = (f_1,\dots,f_L)\colon[0,1]^2\to\R^L$ is continuous \cite[Section~1.5]{mukherjee2016estimation}, \cite[Remark~3]{mukherjee2023inference}. Applying part~(ii) with the cost $-\boldsymbol{\theta}^{\top} \mathbf{f}$ and $\mathbf{g} = \mathbf{f}$ shows that its sufficient statistic $\Upsilon_N := \Upsilon_N(\mathbf{f})$ satisfies, for every $\boldsymbol{\theta} \in \R^{L}$, 
	\begin{align}
		\label{eq: mukherjee_CLT}
		N^{-1/2}(\Upsilon_N - \E_{\boldsymbol{\theta}} \Upsilon_N )\Longrightarrow \mathcal{N}\big(\mathbf{0}, \Sigma^{\boldsymbol{\theta}}\big) \qquad\text{and}\qquad N^{-1}\Cov_{\boldsymbol{\theta}}(\Upsilon_N) \longrightarrow \Sigma^{\boldsymbol{\theta}},
	\end{align}
	where $\Sigma^{\boldsymbol{\theta}}_{ij} := \langle (I - P_{\mu_{\boldsymbol{\theta}}}) f_i, (I - P_{\mu_{\boldsymbol{\theta}}}) f_j \rangle_{L^2(\mu_{\boldsymbol{\theta}})}$ with  $\mu_{\boldsymbol{\theta}} := \pi^{-\boldsymbol{\theta}^{\top} \mathbf{f}}_{\Leb,\Leb}$. 
	
	For $\boldsymbol{\theta}= \mathbf{0}$, the projection $P_{\mu_{\boldsymbol{\theta}}}$ can be dropped from the limiting covariance matrix and we recover \cite[Theorem~1.7(b)]{mukherjee2023inference}, which is a consequence of Hoeffding's combinatorial central limit theorem \cite{hoeffding1951combinatorial}. Indeed, 
	as in \cite{mukherjee2016estimation} and \cite{mukherjee2023inference}, cost functions, without loss of generality, can be restricted to the class $\cC$ of functions $f$ with $\int_0^1 f(\cdot, y)\,dy = \int_0^1 f(x, \cdot)\,dx = 0$, 
	since the Gibbs measure is invariant under perturbations of the cost by additively separable functions in $\mathcal{A}$. Since $\cC$ is the orthogonal complement of $\cA$ in $L^2(u)$ and $\mu_{\mathbf{0}} = u$, for $\mathbf{f} \in \cC^{L}$ we have $\Sigma^{\mathbf{0}}_{ij} = \int f_i f_j\,du$. 
	However, for $\boldsymbol{\theta} \ne \mathbf{0}$, centering $f$ in $\mathcal{C}$ does not remove the projection, since $\mathcal{C}$ is not neccesarily orthogonal to $\mathcal{A}$ in $L^2(\mu_{\boldsymbol{\theta}})$.
	See Sections~\ref{sec: bg_pal} and \ref{sec: bg_mukherjee} for further discussion.

	\subsection{Limiting spectral distribution}
	\label{sec: main_spectral}

	We now describe the spectrum of the conditioned table. Following \cite{lyu2024large,duan2026scaling}, we consider the Gram (sample covariance) matrix of the centered and normalized table. Fix margins $(\r,\c)$ and a prior mean $\bLambda$, let $\mathbf{Z} = \mathbf{Z}^{\r,\c}$, $\mathbf{X} \sim \lambda_{\r,\c;\bLambda}$ and $\mathbf{Y} \sim \Poisson(\mathbf{Z})$, and let $\bar s := N/(mn)$ be the average entry of $\mathbf{Z}$. The \emph{fluctuation matrices} are
	\begin{align}
		\label{eq: def_fluctuation_matrices}
		\mathbf{A}_{\mathbf{X}} := \frac{\mathbf{X} - \mathbf{Z}}{\sqrt{(m+n)\bar s}}, \qquad \mathbf{A}_{\mathbf{Y}} := \frac{\mathbf{Y} - \mathbf{Z}}{\sqrt{(m+n)\bar s}}.
	\end{align}
	For an $m\times m$ symmetric matrix $\mathbf{H}$, let $\mu_{\mathbf{H}} := \frac1m\sum_{i=1}^m\delta_{\lambda_i(\mathbf{H})}$ be its empirical spectral distribution. The entries of $\mathbf{A}_{\mathbf{Y}}$ are independent and centered, with variance profile
	\begin{align}
		\label{eq: variance_profile}
		\mathbf{S}_{ij} := \Var\big((\mathbf{A}_{\mathbf{Y}})_{ij}\big) = \frac{\mathbf{Z}_{ij}}{(m+n)\bar s} = \frac{w_{ij}}{m+n}, \qquad w_{ij} := \frac{\mathbf{Z}_{ij}}{\bar s},
	\end{align}
	where $w_{ij}$ is the value on $A_{ij}$ of the density $\varphi_{\mathbf{Z}}$ of $\pi_{\mathbf{Z}}$, the discrete Schr\"odinger bridge. The deterministic approximation for $\mu_{\mathbf{A}_{\mathbf{Y}}\mathbf{A}_{\mathbf{Y}}^\top}$ is the measure $\nu_{\mathbf{S}}$ on $[0,\infty)$ given by the Dyson equation \cite{alt2017local}
	\begin{align}
		\label{eq: dyson_discrete}
		-\frac{1}{G_i(z)} = z - \sum_{l=1}^n \frac{\mathbf{S}_{il}}{1 + \sum_{j=1}^m \mathbf{S}_{jl}\,G_j(z)}, \qquad i \in [m],\ z \in \mathbb{H},
	\end{align}
	where $\mathbb{H}$ is the complex upper half-plane. This equation has a unique holomorphic solution $(G_i)_{i\in[m]} : \mathbb{H}\to\mathbb{H}^m$, and $\frac1m\sum_{i}G_i(z)$ is the Stieltjes transform of $\nu_{\mathbf{S}}$ \cite{alt2017local}.

	In the scaling limit, the variance profile converges to the density of the limiting Schr\"odinger bridge. For an aspect-ratio parameter $\theta \in (0,1)$, which plays the role of the fraction $m/(m+n)$ of rows (and equals $\gamma/(1+\gamma)$ in the limit), a measurable function $W : [0,1]^2 \to [0,\infty)$ and $z \in \mathbb{H}$, we consider the continuum Dyson equation
	\begin{align}
		\label{eq: dyson_continuum}
		-\frac{1}{G(x,z)} = z - (1-\theta)\int_0^1 \frac{W(x,y)\,dy}{1 + \theta\int_0^1 W(x',y)G(x',z)\,dx'}, \qquad x \in [0,1].
	\end{align}
	If $W$ is grid-step with values $w_{ij}$ and $\theta = m/(m+n)$, then \eqref{eq: dyson_continuum} is exactly \eqref{eq: dyson_discrete} with $\mathbf{S}_{ij} = w_{ij}/(m+n)$, in the sense that $G(x,z) = G_{\lceil mx\rceil}(z)$. If $W$ is bounded, then for $\operatorname{Im}z \ge 2\norm{W}_\infty + 1$ the equation \eqref{eq: dyson_continuum} has a unique solution $G(\cdot,z)$ with $\operatorname{Im}G \ge 0$ and $|G| \le (\operatorname{Im}z)^{-1}$ (Lemma~\ref{lem: dyson_continuity}).

	\begin{figure}[H]
		\centering
		\includegraphics[width=1\textwidth]{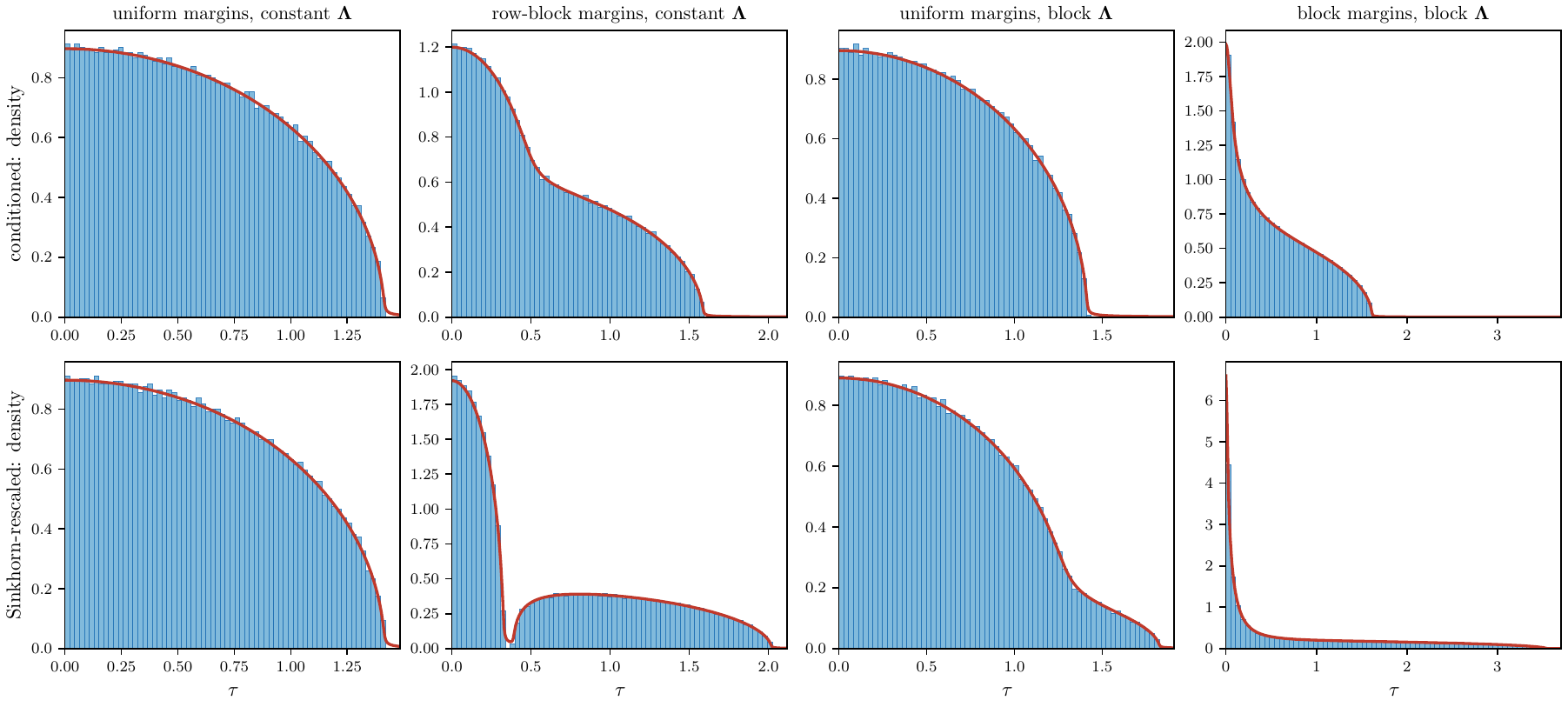}
		\vspace{-0.7cm}
		\caption{Singular values of the fluctuation matrix (histograms) and the density predicted by the Dyson equation (red), for $m = n = 5000$, $N = 0.3n^2$ and $\norm{\bLambda}_1 = N$. Columns: (1)~uniform margins and constant $\bLambda$; (2)~row sums $0.1n$ and $0.5n$ and constant $\bLambda$; (3)~uniform margins and a $2\times 2$ block prior with blocks of sizes $n/4$ and $3n/4$ and entry ratio $5$; (4)~row and column sums $0.15n$ and $0.45n$ and an aligned block prior with entry ratio $5$. Top: the conditioned table, sampled exactly via Proposition~\ref{prop: permutation_CT_coupling}, with the profile $\mathbf{S}_{ij} = \mathbf{Z}_{ij}/((m+n)\bar s)$ in \eqref{eq: dyson_discrete}. Bottom: $\Poisson(\bLambda)$ rescaled to the same margins by Sinkhorn's algorithm, with the profile $e^{\balpha(i)+\bbeta(j)}\mathbf{Z}_{ij}/((m+n)\bar s)$ (Section~\ref{sec: bg_sinkhorn}).}
		\label{fig: ESD_conditioned}
	\end{figure}
	
	Theorem~\ref{thm: ESD_limit} below identifies the limiting spectral distribution of the conditioned table in the dense regime; its proof transfers the global law of the tilted model by a concentration inequality for linear spectral statistics (Proposition~\ref{prop: spectral_concentration}; see Section~\ref{sec: bg_spectral}).
	
	\begin{theorem}[Limiting spectral distribution]
		\label{thm: ESD_limit}
		Assume \ref{assumption: bounded_cost} and \ref{assumption: margins_weak_convergence}, suppose that $\delta \le \rho_{\r_k}, \rho_{\c_k} \le \delta^{-1}$ for some $\delta \in (0,1]$ and all $k$, that $\rho_{\r_k} \to \rho_{\r_\infty}$ and $\rho_{\c_k} \to \rho_{\c_\infty}$ in $L^1[0,1]$, and that there is $c_0 > 0$ such that $\bar s_k \ge c_0$ for all $k$ and $\log\bar s_k = o(m_k)$. Let $W_\infty := d\pi_\infty/du$ and $\theta := \gamma/(1+\gamma)$. Then there is a unique probability measure $\nu_\infty$ on $[0,\infty)$ whose Stieltjes transform coincides with $\int_0^1 G(x,z)\,dx$ for all $z \in \mathbb{H}$ with $\operatorname{Im}z \ge 2\norm{W_\infty}_\infty + 1$, where $G(\cdot,z)$ is the solution of \eqref{eq: dyson_continuum} with $W = W_\infty$. Moreover,
		\begin{align}
			\mu_{\mathbf{A}_{\mathbf{X}_k}\mathbf{A}_{\mathbf{X}_k}^\top} \Longrightarrow \nu_\infty \qquad\text{and}\qquad \mu_{\mathbf{A}_{\mathbf{Y}_k}\mathbf{A}_{\mathbf{Y}_k}^\top} \Longrightarrow \nu_\infty \qquad \text{weakly in probability}.
		\end{align}
	\end{theorem}

	Figure~\ref{fig: ESD_conditioned} illustrates Theorem~\ref{thm: ESD_limit} for four choices of margins and prior means. In each case the spectrum of the conditioned table follows the Dyson equation, and with uniform margins it is the same Marchenko--Pastur law regardless of the prior (first and third columns). By contrast, the Sinkhorn-rescaled Poisson matrix in the bottom row has the same margins and, to first order, the same mean, but a different variance profile, and its spectrum depends on the prior (third column).

	\begin{remark}[Spectra of Gibbs permutations]
		\label{rem: ESD_permutations}
		By \eqref{eq:coupling_Gibbs}, Theorem~\ref{thm: ESD_limit} is equivalently a statement about Gibbs random permutations. Namely, with the block counts $\mathbf{X}(\sigma)$ of \eqref{eq:construction_Gibbs_X}, the conclusion of Theorem~\ref{thm: ESD_limit} holds for $\mathbf{A}_{\mathbf{X}(\sigma_k)}$ with $\sigma_k \sim \Gibbs_k$. Thus, at the level of the global spectrum, the block counts of a Gibbs permutation behave like independent Poisson variables with means $\mathbf{Z}_{k,ij}$, and the variance profile $W_\infty$ is the density of $\pi_\infty = (T_\infty)_{\#}\mu_\infty$, where $\mu_\infty$ is the limiting optimal permuton. If the limiting margins are uniform, then $\nu_\infty$ is the same rescaled Marchenko--Pastur law for every limiting cost (Section~\ref{sec: bg_spectral}). For instance, let $\bLambda_{ij} = e^{-f(i/m,j/m)}$ for a continuous function $f$ on $[0,1]^2$ and let all margins be equal to $d$, so that $\sigma$ is a Gibbs permutation of $[md]$ whose kernel is constant on $d \times d$ blocks. If $m \to \infty$, $d/m$ is bounded below and $\log d = o(m)$, then the $m \times m$ matrix of block counts has this Marchenko--Pastur limit for every $f$; see Figure~\ref{fig: perm_ESD}. The theorem does not apply to unit margins, for which $\bar s = 1/N$.
	\end{remark}

	\begin{figure}[H]
		\centering
		\includegraphics[width=\textwidth]{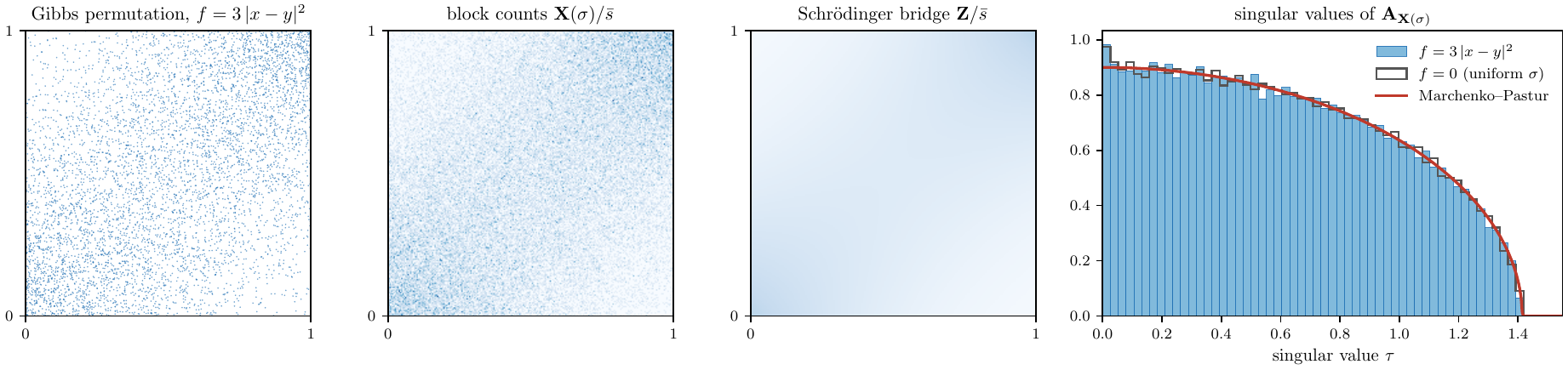}
		\vspace{-0.7cm}
		\caption{Block counts of Gibbs permutations with equal block sizes (Remark~\ref{rem: ESD_permutations}): $m = n = 500$ blocks of size $500$, so that $N = 250{,}000$ and $\bar s = 1$, with $\bLambda_{ij} = e^{-f(i/m,j/m)}$. From left: a sample of $\sigma \sim \Gibbs_{\r,\c;\bLambda}$ for the quadratic cost $f(x,y) = 3|x-y|^2$, shown as the points $(k/N,\sigma(k)/N)$; its block counts $\mathbf{X}(\sigma)/\bar s$; the Schr\"odinger bridge $\mathbf{Z}/\bar s$; and the singular values of $\mathbf{A}_{\mathbf{X}(\sigma)}$ for this $f$ (filled) and $f = 0$, i.e., uniform $\sigma$ (outlined), with the Marchenko--Pastur density $\frac{2}{\pi}\sqrt{2-\tau^2}$ (red). The histograms pool $24$ samples from two Metropolis chains of random transpositions.}
		\label{fig: perm_ESD}
	\end{figure}

	\FloatBarrier
	\section{Background and discussion}
	\label{sec: background}
	
	\subsection{Random contingency tables and the Poisson model}
	\label{sec: bg_tables}
	\label{sec: bg_poisson}
	
	Random contingency tables have been studied extensively, motivated in part by the problem of counting $\CT(\r,\c)$ \cite{barvinok2008enumerating, barvinok2010approximation, barvinok2010number, branden2023lower, chatterjee2010properties}. Barvinok \cite{barvinok2010does} showed that a uniformly random contingency table concentrates around its \emph{typical table}, the mean of the maximum entropy distribution of independent geometric entries with expected margins $(\r,\c)$. Lyu and Mukherjee \cite{lyu2024large} extended this to random matrices with i.i.d.\ entries from a base measure $\mu$ conditioned on their margins: the typical table is the mean of the solution of \eqref{eq:KL_min} with $\cR = \mu^{\otimes(m\times n)}$, a product of exponential tilts whose parameters $\balpha(i)+\bbeta(j)$ solve a dual maximum likelihood problem, and a transference principle shows that the conditioned matrix concentrates around it.
	
	For the Poisson model with prior mean $\bLambda$, this framework is explicit. Since the Kullback--Leibler divergence of $\Poisson(x)$ from $\Poisson(\lambda)$ is $x\log(x/\lambda) + \lambda - x$ and $\sum_{i,j}\x_{ij} = N$ on $\Tpoly(\r,\c)$, the minimizer for the reference measure $\cR = \Poisson(\bLambda)$ is the law of $\mathbf{Y} \sim \Poisson(\mathbf{Z})$ \cite{csiszar1975divergence}, whose mean solves the matrix scaling problem \eqref{eq: def_matrix_scaling}; thus $\mathbf{Z}/N$ is a static Schr\"odinger bridge, and a general cost requires the inhomogeneous prior $\bLambda$ (i.i.d.\ $\Poisson(1)$ entries give an additively separable cost, hence the independent coupling). In the rank-one tilting model $\mathbf{Y}_{ij} \sim \Poisson(\bLambda_{ij}e^{\balpha(i)+\bbeta(j)})$, the log-likelihood of a table with margins $(\r,\c)$ is $\langle\balpha,\r\rangle + \langle\bbeta,\c\rangle - \langle\exp(\balpha\oplus\bbeta),\bLambda\rangle$ plus a term free of $(\balpha,\bbeta)$, so the maximum likelihood estimator solves \eqref{eq: kantorovich_dual}, i.e., it is the pair of Schr\"odinger potentials, and Sinkhorn's algorithm is alternating maximization of the likelihood. Moreover, for $\x \in \CT(\r,\c)$, $\P(\mathbf{Y} = \x) = e^{\langle\balpha,\r\rangle+\langle\bbeta,\c\rangle - \norm{\mathbf{Z}}_1}\prod_{i,j}\bLambda_{ij}^{\x_{ij}}/\x_{ij}!$, which proves \eqref{eq: same_conditional_law}.
	
	Two further features distinguish the Poisson model. First, it is uniformly ``tame'': EOT potentials are bounded whenever the cost is (Lemma~\ref{lem: density_bounds}), so $\mathbf{Z}_{ij}$ is of order $\r(i)\c(j)/N$ for all margins. For uniform tables, the potentials of Barvinok's typical table may be unbounded, with a sharp phase transition in the ratio of the largest to the smallest margin \cite{barvinok2010does, dittmer2020phase}; tameness (cf.\ \cite[Definition~1.8]{lyu2024large}) is therefore assumed in the Barvinok--Hartigan formula \cite{barvinok2012asymptotic} and in the transference principles of \cite{lyu2024large}, whose factors are $\exp(C(m+n)\log(m+n))$ or larger with $C$ depending on the tameness \cite[Theorems~1.13 and 1.14]{lyu2024large}. Second, the combinatorial counterpart of uniform tameness is Barvinok's identity \eqref{eq: barvinok_per_identity}, equivalently our coupling \eqref{eq:coupling_Gibbs}. It gives the transference factor $(eN)^{m+n}m^{-m}n^{-n}$ of Lemma~\ref{thm: transference} for every prior mean and all margins, and the enumeration of Corollary~\ref{cor: enumeration} without tameness. Beyond concentration, the coupling identifies the table with the block counts of a Gibbs permutation, which gives access to fluctuations and sharp partition function asymptotics.
	
	\vspace{-0.2cm}
	\subsection{Entropic optimal transport and fluctuations}
	\label{sec: bg_EOT}
	\label{sec: bg_fluctuations}
	
	For EOT and the Schr\"odinger problem we refer to \cite{leonard2013survey, nutz2021introduction, peyre2019computational}. The convergence in Theorem~\ref{thm: conditioned_model_weak_as_limit_no_rate}(ii) is equivalent to $d_{\mathrm{BL}}(\pi_{\mathbf{X}_k},\pi_\infty) \to 0$ in probability, where $d_{\mathrm{BL}}$ is the bounded Lipschitz metric; in the dense regime, it cannot be strengthened to total variation, since the entries of $\mathbf{X}_k$ are then of order one. The deterministic term in Theorem~\ref{thm: conditioned_model_weak_as_limit} comes from the stability of Schr\"odinger bridges \cite{ghosal2022stability, nutz2023stability, eckstein2022quantitative, deligiannidis2024quantitative, carlier2020differential}; under density bounds, \cite[Theorem~2.15]{duan2026scaling} improves its dependence on $d_H(\varphi_{\bLambda_k},\varphi_{\bLambda_\infty})$ to linear order.
	
	The fluctuation results need neither smooth test functions, nor dense tables, nor convergence of the margins beyond weak convergence: the proofs pass to permuton coordinates, where the optimal permutons have densities in $[e^{-2\osc(\kappa)}, e^{2\osc(\kappa)}]$ for arbitrary margins (Lemma~\ref{lem: permuton_coordinates}). Some continuity of $g$ is needed when the limiting margins have atoms, since, for an atom $x_0$ of $\rho_{\r_\infty}$, the indicator of $\{x_0\}\times[0,v]$ has vanishing fluctuations but in general positive limiting variance. The variance $V(g)$ vanishes exactly for additively separable $g$, whose fluctuations are frozen by the margin constraints. In the Fisher--Yates case with uniform margins, $\FF$ is the standard Brownian pillow of the Blum--Kiefer--Rosenblatt test \cite{blum1961distribution}, and for Lipschitz $g$ the central limit theorem also follows from Hoeffding's combinatorial central limit theorem \cite{hoeffding1951combinatorial} if $N_k = o(n_k^2)$. Unlike the Gaussian free field, whose precision operator is the Laplacian, the precision structure of $\Xi$ has no spatial locality, since cells interact only through the margins; accordingly, $\Xi$ has the regularity of white noise (Remark~\ref{rem: Hs_convergence}).
	
	\vspace{-0.3cm}
	\subsection{Gibbs permutations and Pal's conjecture}
	\label{sec: bg_pal}
	
	Mukherjee \cite{mukherjee2016estimation} showed that the empirical measure of $\sigma$ drawn with probability proportional to $\exp(\theta\sum_i f(i/N,\sigma(i)/N))$ converges to the solution of an EOT problem over permutons, and that $\frac1N\log L_N(f) \to -\Gamma^{\mathrm{u}}(f)$; see also \cite{mukherjee2023inference, trashorras2008large, kenyon2020permutations, borga2024large}. Our Gibbs measures \eqref{eq: gibbs_intro} have kernels that are constant on the blocks of $\widetilde{\bLambda}$. Pal \cite{pal2024limitingconjecture} conjectured that, for symmetric $C^2$ costs with $f(x,x) = 0$, $f(1-x,1-y) = f(x,y)$ and a spectral gap, $e^{N\Gamma^{\mathrm{u}}(f)}L_N(f) \to \detF(I - \cS^2)^{-1/2}$; since $N!L_N(f)$ is a permanent, this is also called Pal's permanent conjecture \cite{ottolini2026permanent}. Pal proved the limit for truncated eigenfunction expansions; Tripathi \cite{tripathi2026partitionfunctionclassmallows} proved the equivalent form \eqref{eq: mallows_tripathi} for $C^2$ kernels with a spectral gap, by an approximate doubly stochastic rescaling and McCullagh's expansion \cite{McCullagh2014asymptotic}; and Ottolini and Starr \cite{ottolini2026permanent} proved it for block-uniform matrices.
	
	Corollary~\ref{cor: mallows} gives the limit for every continuous cost, without symmetry, smoothness or spectral assumptions, with $\Gamma^{\mathrm{u}}(f_N)$ in place of $\Gamma^{\mathrm{u}}(f)$; for $C^1$ costs, replacing $\Gamma^{\mathrm{u}}(f_N)$ by $\Gamma^{\mathrm{u}}(f)$ introduces the boundary factor $e^{-\frac12[a(1)-a(0)+b(1)-b(0)]}$, which equals one under Pal's assumptions and cannot be dropped in general (for $f(x,y) = x$ it equals $e^{-1/2}$). It thus contains Pal's conjecture and Tripathi's theorem for $C^1$ costs and extends them to non-symmetric kernels, with $\cS^*\cS$ in place of $\cS^2$. The difference from \cite{tripathi2026partitionfunctionclassmallows} is that for grid-step costs the doubly stochastic rescaling is exact, given by the discrete Schr\"odinger potentials (Lemma~\ref{lem: grid_step_identity}), so no discretization errors arise. In particular, $\cD_k(f) \ge 1$ by the van der Waerden bound \cite{Falikman1981vanderwaerden}, which also underlies Lemma~\ref{thm: transference}, and the limit $D_\infty = \prod_i(1-s_i^2)^{-1/2}$ is given by the canonical correlations $s_i$ of the limiting bridge. Since $\mathbf{B}\otimes\mathbf{1}_n\mathbf{1}_n^\top$ is Barvinok's matrix of tables with all margins $n$ and prior $\mathbf{B}$, Lemma~\ref{lem: partition_fredholm} recovers \cite[Theorem~2.1]{ottolini2026permanent} and allows unequal and growing blocks.
	
	\subsection{Random margins and Mukherjee's conjecture}
	\label{sec: bg_random_margins}
	\label{sec: bg_mukherjee}
	
	Harchaoui, Liu and Pal \cite{pal2024chaos} consider i.i.d.\ samples $X_i \sim \rho_{\r}$, $Y_i \sim \rho_{\c}$ and the Gibbs average $\hat\mu_N$ of the empirical couplings $\frac1N\sum_i\delta_{(X_i,Y_{\sigma(i)})}$, which converges to $\pi := \pi^\kappa_{\rho_{\r},\rho_{\c}}$. For $\int g\,d\pi = 0$ and $Pg = a\oplus b$ with $\int a\,d\rho_{\r} = 0$, the first-order chaos of $\int g\,d\hat\mu_N$ is $\frac1N\sum_i(a(X_i) + b(Y_i))$, since the system determining it is the system of normal equations \eqref{eq: normal_equations}; the limiting variance is $\norm{a}^2_{L^2(\rho_{\r})} + \norm{b}^2_{L^2(\rho_{\c})}$ \cite[Theorem~4 and Corollary~5]{pal2024chaos}, and the same holds for the empirical EOT plan \cite{gonzalezsanz2024weaklimits}. The remainder $(I-P)g$ contributes only at order $N^{-1}$, through a second-order Gaussian chaos \cite[Theorem~6 and Remark~3]{pal2024chaos}, whereas with fixed margins $\Phi_k(Pg) = 0$ and the fluctuations come from $(I-P)g$. Both pictures come from the linearization of the Schr\"odinger system: the derivative of the potentials in the direction of a cost perturbation $g$ is $Pg$, and the second derivative of the EOT value is $-\norm{(I-P)g}^2_{L^2(\pi)}$ (Proposition~\ref{thm: smooth_cost_perturbations} and its proof).
	
	Mukherjee \cite{mukherjee2016estimation} asked for the limiting distribution of the sufficient statistic $\Upsilon_N$ of the exponential family $(\Gibbs^{-\boldsymbol{\theta}^{\top}\mathbf{f}}_N)_{\boldsymbol{\theta}}$. Hoeffding's theorem covers $\boldsymbol{\theta} = \mathbf{0}$, which Mukherjee and Tagami \cite{mukherjee2023inference} used for the maximum likelihood estimator at $\boldsymbol{\theta} = \mathbf{0}$; they suggested that the techniques of \cite{pal2024chaos} 
	may be used to cover the case of $\boldsymbol{\theta} \ne \mathbf{0}$ \cite[Remark~3]{mukherjee2023inference}. 
	We instead derive the central limit theorem from the partition function asymptotics, via the exact Laplace identity (Lemma~\ref{lem: laplace_identity}), so that Pal's and Mukherjee's conjectures have a common source. Moreover, $N^{-1}\Cov(\Upsilon_N) \to \Sigma^{\boldsymbol{\theta}}$, which, by the multivariate analogue of Proposition~\ref{thm: smooth_cost_perturbations}, is the Hessian of the limit $-\Gamma^{\mathrm{u}}(-\boldsymbol{\theta}^{\top} \mathbf{f})$ of $\frac1N\log L_N(-\boldsymbol{\theta}^{\top}\mathbf{f})$. 
	We expect these results to give asymptotic normality of the maximum likelihood estimator at every $\boldsymbol{\theta}$, at rate $N^{-1/2}$ with limiting covariance $(\Sigma^{\boldsymbol{\theta}})^{-1}$.
	
	\subsection{Spectral properties and Sinkhorn rescaling}
	\label{sec: bg_spectral}
	\label{sec: bg_sinkhorn}
	
	The limiting spectral distribution in Theorem~\ref{thm: ESD_limit} depends only on the density $W_\infty$ of the limiting bridge and on $\gamma$. With uniform limiting margins, $W_\infty$ has uniform marginals, the Dyson equation has a constant solution, and $\nu_\infty$ is a rescaled Marchenko--Pastur law \cite{marchenko1967distribution} for every cost, also for block counts of Gibbs permutations (Remark~\ref{rem: ESD_permutations}). The $L^1$ convergence of the margins cannot be dropped: rows with alternating sums $n$ and $2n$ give a generalized Marchenko--Pastur law with two population variances, whereas their weak limit predicts the Marchenko--Pastur law. Lyu and Mukherjee \cite[Theorem~1.33]{lyu2024large} proved spectral limits for conditioned matrices with identical base measures, using transference under tameness and sub-Gaussian concentration \cite{guionnet2000concentration, klochkov2020uniform}; for Poisson entries we use the modified logarithmic Sobolev inequality \cite{bobkov1998modified} and Lemma~\ref{thm: transference}, with no tameness. The local law \cite{alt2017local} gives bulk rigidity for $\mathbf{Y}$ only with probability $1 - O(m^{-D})$, which does not survive transference; this is why we work with the global law. For bounded average margins, as for Mallows-type permutations, transference is no longer effective, and the permutation representation is a natural substitute: in the Fisher--Yates case, $\mathbf{X}(\sigma)$ is the biadjacency matrix of a bipartite configuration model, linear spectral statistics concentrate by \cite{maurey1979construction}, and the resolvent method for locally tree-like graphs \cite{bordenave2010resolvent} should identify the limit.
	
	Rescaling $\mathbf{X}^{\circ} \sim \Poisson(\bLambda)$ to the margins $(\r,\c)$ by Sinkhorn's algorithm gives a different way to impose margins \cite{landa2022scaling, duan2026scaling}. Rescaling and conditioning agree to first order, but the rescaled fluctuations have variance profile $e^{\balpha(i)+\bbeta(j)}\mathbf{Z}_{ij}$ instead of $\mathbf{Z}_{ij}$, so their spectra differ unless the potentials are constant (Figure~\ref{fig: ESD_conditioned}). Central limit theorems for the rescaled matrix seem harder: its linear statistics depend nonlinearly on $m+n$ random potentials, and the quadratic remainder of the delta method is of order $\max(m,n)/N$, comparable to $N^{-1/2}$ in the dense regime. For the conditioned table, the exponential family structure makes the Laplace transform an exact ratio of partition functions, and the second-order effects of the potentials appear as the determinant $D_\infty$, which cancels in this ratio.
	
	\section{The permutation representation and transference}
	\label{sec: preliminaries}
	
	We first prove the permutation representation \eqref{eq:coupling_Gibbs}, which also yields Barvinok's identity, and then deduce the transference principle.
	
	\subsection{The permutation representation}
	\label{sec: proof_coupling}
	
	For $\bLambda \in \R^{m\times n}_{>0}$, Barvinok \cite{barvinok2008enumerating} proved that
	\begin{align}\label{eq: barvinok_per_identity}
		\sum_{\x \in \CT(\r, \c)} \prod_{i,j} \frac{\bLambda_{ij}^{\x_{ij}}}{\x_{ij}!} = \frac{\per(\widetilde{\bLambda})}{\prod_{i=1}^{m} \r(i)!  \prod_{j = 1}^n \c(j)!}.
	\end{align}
	Here $\widetilde{\bLambda}$ is the block matrix of Section~\ref{sec: intro_gibbs}; see Figure~\ref{fig: tilde_matrix}. We write $\widetilde{\mathbf{A}}$ for the same construction applied to any $m \times n$ matrix $\mathbf{A}$. 
	
	The following proposition makes the coupling \eqref{eq:coupling_Gibbs} precise in permuton coordinates; recall $Q_{\r}, Q_{\c}$ from \eqref{eq: index_quantile} and $T = (q_{\r}, q_{\c})$ from Section~\ref{sec: permutons_prelim}.
	
	\begin{prop}
		\label{prop: permutation_CT_coupling}
		Let $\sigma \sim \Gibbs$ and let $\mathbf{X} = \mathbf{X}(\sigma)$ be its block counts \eqref{eq:construction_Gibbs_X}. Then the following hold.
		\begin{enumerate}[label=(\roman*)]
			\item $\mathbf{X} \in \CT(\r,\c)$ and $\mathbf{X} \sim \lambda_{\r, \c; \bLambda}$, that is, \eqref{eq:coupling_Gibbs} holds. Moreover, Barvinok's identity \eqref{eq: barvinok_per_identity} holds.
			\item For every $g \in L^1([0,1]^2)$, $\displaystyle\int_{[0,1]^2} g \, d\pi_{\mathbf{X}} = \int_{[0,1]^2} (\Ek g)\circ T \, d\mu_{\sigma}$.
			\item The measures $\pi_{\mathbf{X}}$ and $T_{\#}\mu_\sigma$ assign the same mass $\mathbf{X}_{ij}/N = \mu_\sigma\big(T^{-1}(A_{ij})\big)$ to each cell $A_{ij}$. Consequently, for every $g \in C([0,1]^2)$ with modulus of continuity $\omega_g$,
			\begin{align}
				\left| \int_{[0, 1]^2} g \, d\pi_{\mathbf{X}} - \int_{[0, 1]^2} g \circ T \, d\mu_{\sigma} \right| \le \omega_g\big(\sqrt{m^{-2} + n^{-2}}\big).
			\end{align}
		\end{enumerate}
	\end{prop}
	
	\noindent Figure~\ref{fig: coupling} illustrates the two measures in Proposition~\ref{prop: permutation_CT_coupling} on a single cell. 
	
	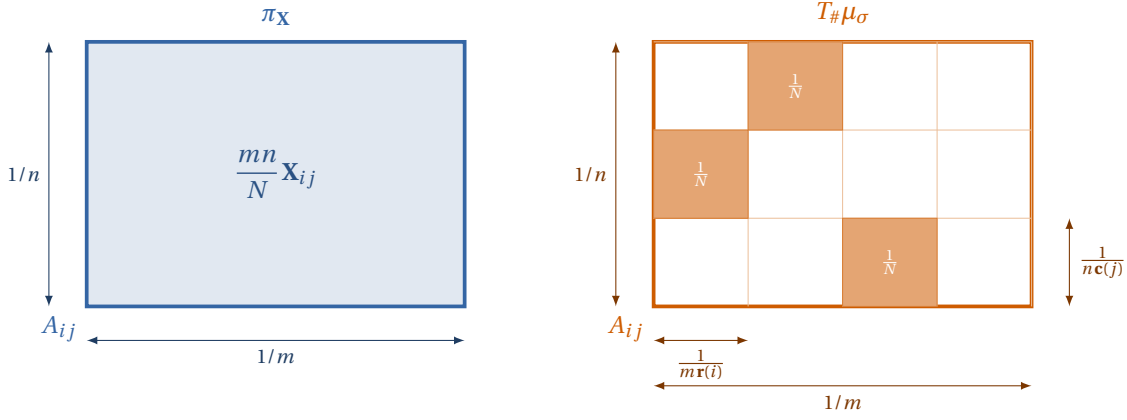
\begin{figure}[t]
		\centering
		\begin{tikzpicture}[font=\small, >=latex]
			\def\cellW{5.0}\def\cellH{3.5}\def\nsubx{4}\def\nsuby{3}
			\pgfmathsetmacro{\dw}{\cellW/\nsubx}
			\pgfmathsetmacro{\dh}{\cellH/\nsuby}
			\definecolor{picolor}{RGB}{52,101,164}
			\definecolor{Tcolor}{RGB}{206,92,0}
			\fill[picolor!15] (0,0) rectangle (\cellW,\cellH);
			\draw[picolor, line width=1.4pt] (0,0) rectangle (\cellW,\cellH);
			\node[picolor!80!black] at (\cellW/2,\cellH/2) {$\displaystyle\frac{mn}{N}\,\mathbf{X}_{ij}$};
			\node[below left, picolor] at (0,0) {$A_{ij}$};
			\node[above, picolor] at (\cellW/2,\cellH+0.1) {$\pi_{\mathbf{X}}$};
			\draw[<->, picolor!60!black, thin] (0,-0.45) -- (\cellW,-0.45) node[midway, below, font=\scriptsize] {$1/m$};
			\draw[<->, picolor!60!black, thin] (-0.5,0) -- (-0.5,\cellH) node[midway, left, font=\scriptsize] {$1/n$};
			\def\ox{7.5}
			\draw[Tcolor, line width=1.4pt] (\ox,0) rectangle (\ox+\cellW,\cellH);
			\foreach \kk in {1,...,\nsubx}{
				\pgfmathsetmacro{\xpos}{\ox+\kk*\dw}
				\draw[Tcolor!40, very thin] (\xpos,0) -- (\xpos,\cellH);
			}
			\foreach \kk in {1,...,\nsuby}{
				\pgfmathsetmacro{\ypos}{\kk*\dh}
				\draw[Tcolor!40, very thin] (\ox,\ypos) -- (\ox+\cellW,\ypos);
			}
			\foreach \col/\row in {0/1, 1/2, 2/0}{
				\pgfmathsetmacro{\xl}{\ox+\col*\dw}
				\pgfmathsetmacro{\xr}{\ox+\col*\dw+\dw}
				\pgfmathsetmacro{\yl}{\row*\dh}
				\pgfmathsetmacro{\yt}{\row*\dh+\dh}
				\fill[Tcolor!55] (\xl,\yl) rectangle (\xr,\yt);
				\draw[Tcolor!80, very thin] (\xl,\yl) rectangle (\xr,\yt);
				\node[white, font=\tiny] at ({(\xl+\xr)/2},{(\yl+\yt)/2}) {$\tfrac{1}{N}$};
			}
			\node[above, Tcolor] at (\ox+\cellW/2,\cellH+0.1) {$T_{\#}\mu_{\sigma}$};
			\node[below left, Tcolor] at (\ox,0) {$A_{ij}$};
			\draw[<->, Tcolor!60!black, thin] (\ox,-0.45) -- (\ox+\dw,-0.45) node[midway, below, font=\scriptsize] {$\tfrac{1}{m\,\r(i)}$};
			\draw[<->, Tcolor!60!black, thin] (\ox+\cellW+0.5,0) -- (\ox+\cellW+0.5,\dh) node[midway, right, font=\scriptsize] {$\tfrac{1}{n\,\c(j)}$};
			\draw[<->, Tcolor!60!black, thin] (\ox,-1.05) -- (\ox+\cellW,-1.05) node[midway, below, font=\scriptsize] {$1/m$};
			\draw[<->, Tcolor!60!black, thin] (\ox-0.5,0) -- (\ox-0.5,\cellH) node[midway, left, font=\scriptsize] {$1/n$};
		\end{tikzpicture}
		\caption{The two measures in Proposition~\ref{prop: permutation_CT_coupling} on a single cell $A_{ij}$, with $\r(i) = 4$, $\c(j) = 3$ and $\mathbf{X}_{ij} = 3$. Left: $\pi_{\mathbf{X}}$ has constant density $\frac{mn}{N}\mathbf{X}_{ij}$ on $A_{ij}$. Right: $T$ maps each square $B_{k\ell}$ with $Q_{\r}(k) = i$ and $Q_{\c}(\ell) = j$ onto one of the $\r(i)\c(j)$ sub-rectangles of $A_{ij}$, and $T_{\#}\mu_\sigma$ puts mass $1/N$ on the $\mathbf{X}_{ij}$ sub-rectangles hit by $\sigma$, at most one in each row and each column of the sub-grid. Both measures give mass $\mathbf{X}_{ij}/N$ to $A_{ij}$.}
		\label{fig: coupling}
	\end{figure}
	
	Throughout this subsection we fix positive integer margins $(\r, \c) \in \N^m\times\N^n$ with $\sum_{i} \r(i) = \sum_{j}\c(j) = N$ and a positive matrix $\bLambda \in \R^{m\times n}_{>0}$. Let $\bar\r(i) := \sum_{\tau \le i}\r(\tau)$ and $\bar\c(j) := \sum_{\tau \le j}\c(\tau)$ denote the cumulative sums, with $\bar\r(0) = \bar\c(0) = 0$. We first establish the basic properties of $T = (q_{\r}, q_{\c})$ with respect to the two grids \eqref{eq: def_partitions}.
	
	\begin{lemma}
		\label{lem: quantile_properties}
		For all $i \in [m]$ and $j \in [n]$,
		\begin{align}
			\label{eq: preimage_of_cell}
			T^{-1}(A_{ij}) = \Big( \tfrac{\bar\r(i-1)}{N}, \tfrac{\bar\r(i)}{N}\Big] \times \Big( \tfrac{\bar\c(j-1)}{N}, \tfrac{\bar\c(j)}{N}\Big] = \bigcup_{k,\ell\,:\, Q_{\r}(k) = i,\ Q_{\c}(\ell) = j} B_{k\ell}.
		\end{align}
		In particular, $T(B_{k\ell}) \subset A_{Q_{\r}(k), Q_{\c}(\ell)}$ for all $k, \ell \in [N]$, and for every $\sigma \in S_N$,
		\begin{align}
			\label{eq: permuton_mass_of_cell}
			\mu_{\sigma}\big(T^{-1}(A_{ij})\big) = \frac{1}{N} \,\#\big\{k \in [N] : Q_{\r}(k) = i,\ Q_{\c}(\sigma(k)) = j \big\}.
		\end{align}
	\end{lemma}
	
	\begin{proof}
		Since $\rho_{\r} = m\r(\tau)/N$ on $\big(\frac{\tau-1}{m}, \frac{\tau}{m}\big]$, we have $F_{\r}(i/m) = \bar\r(i)/N$ for all $i$. As $F_{\r}$ is a continuous increasing bijection of $[0,1]$ with inverse $q_{\r}$, it follows that $q_{\r}(s) \in \big(\frac{i-1}{m}, \frac{i}{m}\big]$ if and only if $s \in \big(\frac{\bar\r(i-1)}{N}, \frac{\bar\r(i)}{N}\big]$. The same holds for $q_{\c}$, which proves the first equality in \eqref{eq: preimage_of_cell}. The second equality holds because $Q_{\r}(k) = i$ if and only if $\bar\r(i-1) < k \le \bar\r(i)$, by the definition \eqref{eq: index_quantile}. Finally, $\mu_\sigma$ assigns mass $1/N$ to each square $B_{k\sigma(k)}$ and no mass to the squares $B_{k\ell}$ with $\ell \ne \sigma(k)$, which gives \eqref{eq: permuton_mass_of_cell}.
	\end{proof}

	\begin{proof}[\textbf{Proof of Proposition~\ref{prop: permutation_CT_coupling}}]
		(i) Let $R_i := Q_{\r}^{-1}(i)$ and $C_j := Q_{\c}^{-1}(j)$, so that $|R_i| = \r(i)$ and $|C_j| = \c(j)$. The $i$th row sum of $\mathbf{X}(\sigma)$ is $|R_i|$ and, since $\sigma$ is a bijection, its $j$th column sum is $|C_j|$, so $\mathbf{X}(\sigma) \in \CT(\r,\c)$. Since $\widetilde{\bLambda}$ equals $\bLambda_{ij}$ on $R_i \times C_j$, the weight of $\sigma$ in \eqref{eq: gibbs_intro} depends on $\sigma$ only through its block counts:
		\begin{align}
			\label{eq: block_weight}
			\prod_{k=1}^{N}\widetilde{\bLambda}_{k\sigma(k)} = \prod_{i,j}\bLambda_{ij}^{\mathbf{X}_{ij}(\sigma)}.
		\end{align}
		
		Fix $\x \in \CT(\r,\c)$. The permutations $\sigma$ with $\mathbf{X}(\sigma) = \x$ correspond bijectively to the following choices: for each $i$, a partition of $R_i$ into labeled sets $S_{i1}, \ldots, S_{in}$ of sizes $\x_{i1}, \ldots, \x_{in}$ (the elements mapped into $C_1, \ldots, C_n$), for each $j$, a partition of $C_j$ into labeled sets $S'_{1j}, \ldots, S'_{mj}$ of sizes $\x_{1j}, \ldots, \x_{mj}$ (their images), and bijections $S_{ij} \to S'_{ij}$. Therefore
		\begin{align}
			\label{eq: permutation_count}
			\#\{\sigma \in S_N : \mathbf{X}(\sigma) = \x\} = \prod_{i=1}^m\frac{\r(i)!}{\prod_{j}\x_{ij}!}\prod_{j=1}^n\frac{\c(j)!}{\prod_{i}\x_{ij}!}\prod_{i,j}\x_{ij}! = \frac{\prod_{i}\r(i)!\prod_{j}\c(j)!}{\prod_{i,j}\x_{ij}!}.
		\end{align}
		Combining \eqref{eq: block_weight} and \eqref{eq: permutation_count}, for every $\x \in \CT(\r,\c)$,
		\begin{align}
			\label{eq: block_count_law}
			\P\big(\mathbf{X}(\sigma) = \x\big) = \frac{1}{\per(\widetilde{\bLambda})}\sum_{\sigma' :\, \mathbf{X}(\sigma') = \x}\ \prod_{i,j}\bLambda_{ij}^{\x_{ij}} = \frac{\prod_{i}\r(i)!\prod_{j}\c(j)!}{\per(\widetilde{\bLambda})}\prod_{i,j}\frac{\bLambda_{ij}^{\x_{ij}}}{\x_{ij}!}.
		\end{align}
		Thus $\P(\mathbf{X}(\sigma) = \x) \propto \prod_{i,j}\bLambda_{ij}^{\x_{ij}}/\x_{ij}!$ on $\CT(\r,\c)$, which is \eqref{eq: conditioned_model}. Moreover, since the probabilities in \eqref{eq: block_count_law} sum to one over $\x \in \CT(\r,\c)$, we obtain Barvinok's identity \eqref{eq: barvinok_per_identity}.
		
		(ii) By \eqref{eq: permuton_mass_of_cell} and \eqref{eq: cell_average_invariance},
		\begin{align}
			\int (\Ek g) \circ T\,d\mu_\sigma = \sum_{i,j} \mathbf{G}_{ij}\,\mu_\sigma\big(T^{-1}(A_{ij})\big) = \frac{1}{N}\sum_{i,j}\mathbf{G}_{ij}\mathbf{X}_{ij} = \int g\,d\pi_{\mathbf{X}}.
		\end{align}
		
		(iii) By \eqref{eq: permuton_mass_of_cell}, $(T_{\#}\mu_\sigma)(A_{ij}) = \mathbf{X}_{ij}/N = \pi_{\mathbf{X}}(A_{ij})$. If two measures have the same mass $w$ on a set $A$, then the difference of the integrals of $g$ over $A$ is at most $w \cdot \sup_{p, p' \in A}|g(p) - g(p')|$. Summing over the cells, whose diameter is $\sqrt{m^{-2} + n^{-2}}$, proves the first bound.
	\end{proof}
	
	\subsection{Transference}
	\label{sec: proof_transference}
	
	The following transference principle bounds probabilities under the conditioned model by those under the tilted model.
	
	\begin{lemma}[Transference]
		\label{thm: transference}
		Fix positive integer margins $(\r, \c) \in \N^{m} \times \N^{n}$ with $\norm{\r}_{1} = \norm{\c}_1 = N$ and a prior mean $\bLambda \in \R_{>0}^{m \times n}$. Let $\mathbf{X} \sim \lambda_{\r, \c; \bLambda}$ and $\mathbf{Y} \sim \Poisson(\mathbf{Z}^{\r,\c})$. Then for measurable $h : \R^{m \times n} \to [0,\infty)$,
		\begin{align}
			\E[h(\mathbf{X})] \leq \frac{(eN)^{m+n}}{m^m n^n} \E[h(\mathbf{Y})].
		\end{align}
		In particular, $\P(\mathbf{X} \in B) \leq (eN)^{m+n} m^{-m} n^{-n}\, \P(\mathbf{Y} \in B)$ for every Borel set $B \subset \R^{m \times n}$.
	\end{lemma}
	
	\begin{proof}
		By \eqref{eq: same_conditional_law} (see Section~\ref{sec: bg_poisson}), $\mathbf{X}$ has the conditional law of $\mathbf{Y}$ given $\{\mathbf{Y} \in \CT(\r,\c)\}$. Since $h \ge 0$,
		\begin{align}
			\E[h(\mathbf{X})] = \E\big[h(\mathbf{Y}) \,\big|\, \mathbf{Y} \in \CT(\r, \c)\big] \leq \frac{\E[h(\mathbf{Y})]}{\P(\mathbf{Y} \in \CT(\r, \c))},
		\end{align}
		so it suffices to show that $\P(\mathbf{Y} \in \CT(\r, \c)) \geq (eN)^{-(m+n)}m^{m}n^{n}$. Since $\sum_{i,j} \mathbf{Z}_{ij} = N$, the definition of $\mathbf{Y}$ and Barvinok's identity \eqref{eq: barvinok_per_identity} give
		\begin{align}
			\label{eq: ind_model_lower_bd_pf_1}
			\P\big(\mathbf{Y} \in \CT(\r, \c)\big) = \sum_{\x \in \CT(\r, \c)} \prod_{i, j}e^{-\mathbf{Z}_{ij}} \frac{\mathbf{Z}_{ij}^{\x_{ij}}}{\x_{ij}!} = e^{-N} \sum_{\x \in \CT(\r, \c)} \prod_{i,j} \frac{\mathbf{Z}_{ij}^{\x_{ij}}}{\x_{ij}!} = e^{-N}\frac{\per(\widetilde{\mathbf{Z}})}{\prod_{i = 1}^{m}\r(i)!  \prod_{j = 1}^n \c(j)!}.
		\end{align}
		Let $\mathbf{B}_{ij} := \mathbf{Z}_{ij}/(\r(i)\c(j))$. Since $\mathbf{Z} \in \Tpoly(\r,\c)$,  $\widetilde{\mathbf{B}}$ is doubly stochastic, and the van der Waerden lower bound for permanents of doubly stochastic matrices \cite{Falikman1981vanderwaerden} gives $\per(\widetilde{\mathbf{B}}) \geq N!/N^{N}$. Since $\widetilde{\mathbf{Z}}$ is obtained from $\widetilde{\mathbf{B}}$ by multiplying the rows of the $i$th block row by $\r(i)$ and the columns of the $j$th block column by $\c(j)$,
		\begin{align}
			\per(\widetilde{\mathbf{Z}}) = \prod_{i = 1}^{m}\r(i)^{\r(i)} \prod_{j = 1}^{n} \c(j)^{\c(j)}\,\per(\widetilde{\mathbf{B}}) \geq \frac{N!}{N^{N}}\prod_{i = 1}^{m}\r(i)^{\r(i)} \prod_{j = 1}^{n} \c(j)^{\c(j)}.
		\end{align}
		Combining this with \eqref{eq: ind_model_lower_bd_pf_1} and using $\sum_i \r(i) + \sum_j\c(j) = 2N$ and $N! \ge (N/e)^N$, we obtain
		\begin{align}
			\P\big(\mathbf{Y} \in \CT(\r, \c)\big) &\geq e^{-N}\frac{N!}{N^{N}} \prod_{i = 1}^{m} \frac{\r(i)^{\r(i)}}{\r(i)!}\prod_{j = 1}^{n} \frac{\c(j)^{\c(j)}}{\c(j)!}
			= \frac{e^{N}N!}{N^{N}} \prod_{i = 1}^{m}\frac{e^{-\r(i)}\r(i)^{\r(i)}}{\r(i)!}\prod_{j = 1}^{n} \frac{e^{-\c(j)}\c(j)^{\c(j)}}{\c(j)!} \\
			& \geq \prod_{i = 1}^{m}\frac{e^{-\r(i)}\r(i)^{\r(i)}}{\r(i)!}\prod_{j = 1}^{n} \frac{e^{-\c(j)}\c(j)^{\c(j)}}{\c(j)!}.
		\end{align}
		By Stirling's bound $k! \le e\,k^{k+1}e^{-k}$, valid for all $k \ge 1$, we have $e^{-k}k^k/k! \geq 1/(ek)$. Hence the last display is at least $e^{-(m+n)}\prod_{i}\r(i)^{-1}\prod_{j}\c(j)^{-1}$. Finally, by the AM--GM inequality, $\prod_{i} \r(i) \le (N/m)^m$ and $\prod_{j} \c(j) \le (N/n)^n$, which yields the required lower bound.
	\end{proof}
	
	\section{Proof of the scaling limit}
	\label{sec: proof_LLN}
	
	We recall the facts on static Schr\"odinger bridges used throughout, relate the discrete Schr\"odinger bridge to permutons, and then prove Theorem~\ref{thm: conditioned_model_weak_as_limit_no_rate}.  In Section~\ref{sec: EOT_prelim}, $\mu$ and $\nu$ are as in Section~\ref{sec: main_histograms}.
	
	\subsection{Static Schr\"odinger bridges}
	\label{sec: EOT_prelim}

	We shall repeatedly use the following elementary facts.
	
	\begin{lemma}[Pythagorean identity]
		\label{lem: pythagorean}
		Let $c \in L^\infty(\mu\otimes\nu)$ and let $\pi^* \in \Pi(\mu,\nu)$ be of the form $d\pi^* = e^{a\oplus b - c}\,d(\mu\otimes\nu)$ for some $a \in L^\infty(\mu)$ and $b \in L^\infty(\nu)$. Then for every $\pi \in \Pi(\mu,\nu)$,
		\begin{align}
			\label{eq: pythagorean}
			\int c\, d\pi + \KL(\pi \Vert \mu \otimes\nu) = \int c\,d\pi^* + \KL(\pi^*\Vert \mu\otimes\nu) + \KL(\pi \Vert \pi^*).
		\end{align}
		In particular, $\pi^* = \pi^c_{\mu,\nu}$ is the unique minimizer of \eqref{eq: EOT_Problem_General}, and $\Gamma_{\mu,\nu}(c) = \int a\,d\mu + \int b\,d\nu$.
	\end{lemma}
	
	\begin{proof}
		Since $\frac{d\pi^*}{d(\mu\otimes\nu)}$ is bounded above and below by positive constants, $\KL(\pi\Vert\mu\otimes\nu) = \infty$ if and only if $\KL(\pi \Vert \pi^*) = \infty$, in which case both sides of \eqref{eq: pythagorean} are infinite. Otherwise $\pi \ll \pi^*$ and $\log\frac{d\pi}{d\pi^*} = \log\frac{d\pi}{d(\mu\otimes\nu)} - a(x) - b(y) + c(x,y)$ holds $\pi$-a.s. Integrating with respect to $\pi$ and using that $\pi$ has marginals $\mu$ and $\nu$, we get
		\begin{align}
			\KL(\pi\Vert \pi^*) = \KL(\pi \Vert \mu\otimes\nu) + \int c\,d\pi - \int a\,d\mu - \int b\,d\nu.
		\end{align}
		Applying this with $\pi = \pi^*$ gives $\int c\,d\pi^* + \KL(\pi^* \Vert\mu\otimes\nu) = \int a\,d\mu + \int b\,d\nu$, and subtracting the two identities yields \eqref{eq: pythagorean}. Since $\KL(\pi\Vert\pi^*) > 0$ for $\pi \ne \pi^*$, the minimizer is unique.
	\end{proof}
	
	\begin{lemma}[Density bounds]
		\label{lem: density_bounds}
		For every $c \in L^\infty(\mu\otimes\nu)$,
		\begin{align}
			e^{-2\osc(c)} \le \frac{d\pi^c_{\mu,\nu}}{d(\mu\otimes\nu)} \le e^{2\osc(c)} \qquad (\mu\otimes\nu)\text{-a.e.}
		\end{align}
	\end{lemma}
	
	\begin{proof}
		Let $\xi := e^{a\oplus b - c}$ with Schr\"odinger potentials $(a,b)$, and set $B := \int e^{b}\,d\nu$. By the first equation in \eqref{eq: schrodinger_system}, $e^{-a(x)} = \int e^{b(y) - c(x,y)}\,\nu(dy)$ lies between $e^{-\operatorname{ess\,sup} c}B$ and $e^{-\operatorname{ess\,inf} c}B$. Hence
		\begin{align}
			e^{-\osc(c)}\, \frac{e^{b(y)}}{B} \le \xi(x,y) = e^{a(x)}e^{b(y)}e^{-c(x,y)} \le e^{\osc(c)}\,\frac{e^{b(y)}}{B}.
		\end{align}
		Integrating in $x$ with respect to $\mu$ and using the second equation in \eqref{eq: schrodinger_system} shows that $e^{-\osc(c)} \le e^{b(y)}/B \le e^{\osc(c)}$ for $\nu$-a.e.\ $y$. Combining this with the display proves the claim.
	\end{proof}

	\begin{definition}[Additively separable functions]
		\label{def: sum_space}
		The space of additively separable functions is
		\begin{align}
			\label{eq: def_sum_space}
			\cA := \big\{ a \oplus b \,:\, a \in L^2(\mu),\; b \in L^2(\nu) \big\}.
		\end{align}
		For $\pi \in \Pi(\mu,\nu)$, we have $\cA \subset L^2(\pi)$, and we let $P_\pi$ denote the orthogonal projection of $L^2(\pi)$ onto the closure of $\cA$ in $L^2(\pi)$.
	\end{definition}
	
	\begin{lemma}[Projection onto additively separable functions]
		\label{lem: sum_space}
		Let $\pi \in \Pi(\mu,\nu)$ satisfy $c_0 \le \frac{d\pi}{d(\mu\otimes\nu)} \le C_0$ for some constants $0 < c_0 \le C_0 < \infty$. Then $\cA$ is a closed subspace of $L^2(\pi)$, so that $P_\pi$ is the orthogonal projection onto $\cA$, and for $h \in L^2(\pi)$, $P_\pi h$ is the unique $w \in \cA$ such that
		\begin{align}
			\label{eq: normal_equations}
			\E_\pi[h - w \mid X] = 0 \quad\text{and}\quad \E_\pi[h - w \mid Y] = 0 \qquad \pi\text{-a.s.},
		\end{align}
		where $(X,Y) \sim \pi$.
	\end{lemma}
	
	\begin{proof}
		The norms of $L^2(\pi)$ and $L^2(\mu\otimes\nu)$ are equivalent, so it suffices to show that $\cA$ is closed in $L^2(\mu\otimes\nu)$. There, $\cA = \R \oplus L^2_0(\mu) \oplus L^2_0(\nu)$ is an orthogonal direct sum of closed subspaces, where $L^2_0$ denotes mean-zero functions of one coordinate; hence $\cA$ is closed. Since $\cA$ is closed, $P_\pi h$ is the unique $w \in \cA$ such that $h - w$ is orthogonal in $L^2(\pi)$ to $a(X)$ and $b(Y)$ for all $a \in L^2(\mu)$ and $b \in L^2(\nu)$. Since $X \sim \mu$ under $\pi$, the functions $a(X)$, $a \in L^2(\mu)$, are the $\sigma(X)$-measurable elements of $L^2(\pi)$, onto which $\E_\pi[\,\cdot\mid X]$ is the orthogonal projection; hence $h - w \perp a(X)$ for all $a$ if and only if $\E_\pi[h - w \mid X] = 0$, and similarly for $Y$.
	\end{proof}

	Finally, the histogram of the solution of \eqref{eq: def_matrix_scaling} is a static Schr\"odinger bridge.
	
	\begin{lemma}[Histogram of the Schr\"odinger bridge]
		\label{lem: typical_table_bridge}
		Let $\mathbf{Z} = \mathbf{Z}^{\r,\c}$ solve \eqref{eq: def_matrix_scaling}. Then $\pi_{\mathbf{Z}} = \pi^{\kappa}_{\rho_{\r},\rho_{\c}}$ and
		\begin{align}
			\label{eq: typical_table_density}
			\frac{d\pi_{\mathbf{Z}}}{d(\rho_{\r}\otimes\rho_{\c})} = \frac{N \mathbf{Z}_{ij}}{\r(i)\c(j)} \quad \text{on } A_{ij}, \qquad \Gamma_{\rho_{\r},\rho_{\c}}(\kappa) = \frac{1}{N}\sum_{i,j} \mathbf{Z}_{ij}\log\frac{\mathbf{Z}_{ij}\norm{\bLambda}_1}{N\bLambda_{ij}}.
		\end{align}
	\end{lemma}
	
	\begin{proof}
		Since $\norm{\mathbf{Z}}_1 = N$, we have $\varphi_{\mathbf{Z}} = mn\mathbf{Z}_{ij}/N$ and $\rho_{\r}(x)\rho_{\c}(y) = mn\,\r(i)\c(j)/N^2$ on $A_{ij}$, which gives the first formula in \eqref{eq: typical_table_density}. By \eqref{eq: typical_table_tilting} and \eqref{eq: def_kappa}, on $A_{ij}$,
		\begin{align}
			\frac{N \mathbf{Z}_{ij}}{\r(i)\c(j)} = \frac{N\bLambda_{ij}}{\r(i)\c(j)}e^{\balpha(i)+\bbeta(j)} = \exp\Big(\balpha(i) + \log\frac{\norm{\bLambda}_1}{N} + \bbeta(j) - \kappa(x,y)\Big),
		\end{align}
		so $\pi_{\mathbf{Z}}$ has the form required in Lemma~\ref{lem: pythagorean} with bounded potentials. Since $\mathbf{Z} \in \Tpoly(\r,\c)$, $\pi_{\mathbf{Z}} \in \Pi(\rho_{\r},\rho_{\c})$, and Lemma~\ref{lem: pythagorean} shows that $\pi_{\mathbf{Z}}$ is the minimizer. The formula for the value follows from \eqref{eq: EOT_is_SB}, because $\pi_{\mathbf{Z}}$ and $\pi_{\bLambda}$ assign masses $\mathbf{Z}_{ij}/N$ and $\bLambda_{ij}/\norm{\bLambda}_1$ to $A_{ij}$ and both have grid-step densities.
	\end{proof}
	
	\subsection{Permuton coordinates}
	
	The following lemma shows that EOT problems between arbitrary marginals on $[0,1]$ can be pulled back to EOT problems between uniform marginals, that is, to optimization problems over permutons. Recall that we use the notation $\cM = \Pi(\textup{Leb}, \textup{Leb})$ to refer to the set of permutons.
	
	\begin{lemma}
		\label{lem: COV_permuton_to_EOT}
		Let $\rho_1, \rho_2$ be probability measures on $[0,1]$, $T = (q_{\rho_1},q_{\rho_2})$, and $c \in L^\infty(\rho_1 \otimes \rho_2)$. Then
		\begin{align}
			\label{eq: pushforward_EOT_equivalence}
			\min_{\mu \in \cM} \left[ \int_{[0, 1]^2} c \circ T \, d\mu + \KL(\mu \Vert u) \right] =  \min_{\pi \in \Pi(\rho_1, \rho_2)} \left[ \int_{[0, 1]^2} c \, d\pi + \KL(\pi \Vert \rho_1 \otimes \rho_2) \right].
		\end{align}
		Moreover, if $\mu^*$ and $\pi^*$ denote the minimizers of the left- and right-hand sides, then $\pi^* = T_{\#}\mu^*$ and $\frac{d\mu^*}{du} = \xi \circ T$, where $\xi := \frac{d\pi^*}{d(\rho_1\otimes\rho_2)}$.
	\end{lemma}

	\begin{proof}
		Since $T_{\#}u = \rho_1\otimes\rho_2$, the function $c \circ T$ is well defined $u$-a.e.\ and bounded. Let $(a,b)$ be Schr\"odinger potentials for $\pi^*$, so that $\xi = e^{a\oplus b - c}$, and let $\hat\mu$ be the measure with density $\xi\circ T = e^{(a \circ q_{\rho_1}) \oplus (b \circ q_{\rho_2}) - c\circ T}$ with respect to $u$. The set of $x$ for which the first equation in \eqref{eq: schrodinger_system} fails is $\rho_1$-null, so its preimage under $q_{\rho_1}$ is Lebesgue-null. Hence, for Lebesgue-a.e.\ $s$,
		\begin{align}
			\int_0^1 \xi\big(q_{\rho_1}(s), q_{\rho_2}(t)\big)\,dt = \int \xi\big(q_{\rho_1}(s), y\big)\,\rho_2(dy) = 1,
		\end{align}
		and similarly with the roles of the coordinates exchanged. Thus $\hat\mu \in \cM$, and $\hat\mu$ has the form required in Lemma~\ref{lem: pythagorean} for the cost $c \circ T$ and the marginals $(\Leb, \Leb)$, with the bounded potentials $a\circ q_{\rho_1}$ and $b \circ q_{\rho_2}$. By Lemma~\ref{lem: pythagorean}, $\hat\mu = \mu^*$ is the minimizer of the left-hand side of \eqref{eq: pushforward_EOT_equivalence}, and the minimum equals $\int a\circ q_{\rho_1}\,d\Leb + \int b\circ q_{\rho_2}\,d\Leb = \int a\,d\rho_1 + \int b\,d\rho_2$, which is the right-hand side by Lemma~\ref{lem: pythagorean}. Finally, $T_{\#}\mu^* = \pi^*$, since $\int \phi\, d(T_{\#}\mu^*) = \int (\phi\,\xi)\circ T\,du = \int \phi\,\xi\,d(\rho_1\otimes\rho_2)$ for bounded measurable $\phi$.
	\end{proof}

	We call the minimizer of the left-hand side of \eqref{eq: pushforward_EOT_equivalence} the \emph{optimal permuton} for the cost $c\circ T$. The next lemma transfers the projection onto $\cA$ to permuton coordinates. If $\rho_1$ or $\rho_2$ has atoms, then $T$ is not injective, and conditional expectations take the place of its inverse.
	
	\begin{lemma}[Permuton coordinates]
		\label{lem: permuton_coordinates}
		Let $\rho_1$, $\rho_2$, $T$ and $c$ be as in Lemma~\ref{lem: COV_permuton_to_EOT}, let $\mu$ be the optimal permuton for the cost $c\circ T$, and let $\pi := \pi^c_{\rho_1,\rho_2}$. Then $\frac{d\mu}{du}$ takes values in $[e^{-2\osc(c)}, e^{2\osc(c)}]$, and for all $g, h \in L^2(\pi)$,
		\begin{align}
			\label{eq: permuton_coordinates}
			\begin{split}
				P_\mu(h\circ T) &= (P_\pi h)\circ T, \\
				\big\langle(I-P_\pi)g, (I-P_\pi)h\big\rangle_{L^2(\pi)} &= \big\langle(I-P_\mu)(g\circ T), (I-P_\mu)(h\circ T)\big\rangle_{L^2(\mu)}.
			\end{split}
		\end{align}
	\end{lemma}
	
	\begin{proof}
		By Lemmas~\ref{lem: density_bounds} and \ref{lem: COV_permuton_to_EOT}, $\pi = T_{\#}\mu$ and $w := \xi\circ T$ is a density of $\mu$ with respect to $u$, where $\xi$ is a Borel version of $\frac{d\pi}{d(\rho_1\otimes\rho_2)}$ with values in $[e^{-2\osc(c)}, e^{2\osc(c)}]$. In particular, $L^2(\mu) = L^2(u)$ with equivalent norms, and, by Lemma~\ref{lem: sum_space}, $P_\pi$ and $P_\mu$ are the orthogonal projections onto the spaces $\cA_\pi \subset L^2(\pi)$ and $\cA_\mu := \{a'\oplus b' : a', b' \in L^2(\Leb)\} \subset L^2(\mu)$ of additively separable functions. Since $T_{\#}\mu = \pi$ and $(q_{\rho_i})_{\#}\Leb = \rho_i$, the map $U h := h\circ T$ is a linear isometry from $L^2(\pi)$ into $L^2(\mu)$ with $U(\cA_\pi) \subset \cA_\mu$.
		
		Let $\sigma(T)$ and $\sigma(q_{\rho_i})$ be the $\sigma$-algebras generated by $T$ and $q_{\rho_i}$, and $\mathcal{B}([0, 1]^2)$ the Borel $\sigma$-algebra of $[0, 1]^2$. Since $T^{-1}(B_1\times B_2) = q_{\rho_1}^{-1}(B_1)\times q_{\rho_2}^{-1}(B_2)$ and the Borel rectangles generate $\mathcal{B}([0, 1]^2)$,
		we have $\sigma(T) = \sigma(q_{\rho_1})\otimes\sigma(q_{\rho_2})$. Let $E_T$ denote the conditional expectation given $\sigma(T)$ in $L^2(\mu)$, which is the orthogonal projection onto the $\sigma(T)$-measurable elements. Since $w$ is $\sigma(T)$-measurable, $E_Tf = \E_u[f\mid\sigma(T)]$ for $f \in L^2(\mu)$: indeed, 
		\begin{align}
			\int_A f\,d\mu = \int f\,\1_Aw\,du = \int \E_u[f\mid\sigma(T)]\,\1_Aw\,du = \int_{A} \E_{u}[f\mid\sigma(T)]d\mu
		\end{align}
		for $A \in \sigma(T)$. Now let $a' \in L^2(\Leb)$ and $\alpha := \E_{\Leb}[a'\mid\sigma(q_{\rho_1})]$. By Fubini's theorem, \begin{align}
			\int_{A_1\times A_2}a'(s)\,u(ds,dt) = \int_{A_1\times A_2}\alpha(s)\,u(ds,dt)
		\end{align} 
		for $A_i \in \sigma(q_{\rho_i})$, and by Dynkin's $\pi$-$\lambda$ theorem, this identity extends to all $A \in \sigma(T)$. Hence $E_T(a'\oplus 0) = \alpha\oplus 0$. By the Doob--Dynkin lemma \cite{kallenberg2002foundations}, $\alpha = a\circ q_{\rho_1}$ for a Borel function $a$ with $\int a^2\,d\rho_1 = \int\alpha^2\,d\Leb < \infty$. Treating the second coordinate in the same way, we obtain $E_T(\cA_\mu) \subset U(\cA_\pi) \subset \cA_\mu$. Since $E_T$ is self-adjoint and maps $\cA_\mu$ into itself, it maps $\cA_\mu^\perp$ into itself, so $E_TP_\mu = P_\mu E_T$.
		
		Let $h \in L^2(\pi)$. Since $Uh$ is $\sigma(T)$-measurable, $P_\mu Uh = P_\mu E_TUh = E_TP_\mu Uh \in U(\cA_\pi)$, say $P_\mu Uh = Ua_0$ with $a_0 \in \cA_\pi$. For every $a \in \cA_\pi$, we have $Ua \in \cA_\mu$ and hence $\langle h - a_0, a\rangle_{L^2(\pi)} = \langle Uh - P_\mu Uh, Ua\rangle_{L^2(\mu)} = 0$. Thus $a_0 = P_\pi h$, that is, $P_\mu Uh = UP_\pi h$, and \eqref{eq: permuton_coordinates} follows because $U$ is an isometry.
	\end{proof}

	\subsection{Large deviations for random permutations}
	
	By \eqref{eq: permutation_measure_continuous}, the Gibbs measure is an exponential tilt of the uniform distribution on $S_N$ by a functional of the permuton $\mu_\sigma$. We therefore derive Theorem~\ref{thm: conditioned_model_weak_as_limit_no_rate} from the following large deviation principle for uniformly random permutations; see \cite[Supplement, Theorem~1.1]{mukherjee2016estimation}, and \cite{trashorras2008large, kenyon2020permutations} for related results.
	
	\begin{theorem}
		\label{thm: permutations_LDP}
		Let $\sigma$ be uniformly distributed on $S_{N}$ and let $\QQ_{N}$ be the law of $\mu_{\sigma}$ on $\cM$. The sequence $(\QQ_{N})_{N\ge 1}$ satisfies a large deviation principle on $\cM$, equipped with the weak topology, with speed $N$ and good rate function $\mu \mapsto \KL(\mu \Vert u)$. 
	\end{theorem}
	
	To handle the dependence of the cost on $k$, we prove and use the following extension of Varadhan's lemma \cite[Theorem~4.3.1]{dembo1998large}, in which the functional $F_k$ may change along the sequence and converges only along convergent sequences of permutons. For $F_k \equiv F$ continuous, it reduces to Varadhan's lemma for the large deviation principle of Theorem~\ref{thm: permutations_LDP}. Here, $(N_k)_{k\ge1}$ is any sequence of positive integers with $N_k \to \infty$, and for functionals $F_k, F : \cM \to [-\infty,\infty)$ we write
	\begin{align}
		\ell_k := \frac{1}{N_k} \log \E_{\QQ_{N_k}} \big[ e^{N_kF_{k}(\mu_{\sigma})} \big], \qquad \ell := \sup_{\mu \in \cM} \big( F(\mu) - \KL(\mu \Vert u) \big).
	\end{align}
	
	\begin{lemma}
		\label{lem: Varadhan_generalization}
		Let $(N_k)_{k \ge 1}$ be positive integers with $N_k \to \infty$, let $F_k, F : \cM \to [-\infty, \infty)$, $k \ge 1$, be measurable, and let $\ell_k$ and $\ell$ be as above.
		\begin{enumerate}[label=(\roman*)]
			\item If $\liminf_{k} F_{k}(\mu_k) \geq F(\mu)$ whenever $\mu_k \to \mu$ in $\cM$ and $\KL(\mu \Vert u) < \infty$, then $\liminf_{k\to\infty} \ell_k \ge \ell$.
			\item If $\limsup_{k} F_{k}(\mu_k) \leq F(\mu)$ whenever $\mu_k \to \mu$ in $\cM$ and $\KL(\mu \Vert u) < \infty$, and if there exists $\eta > 1$ such that
			\begin{align}
				\label{eq: exponential_UI}
				\limsup_{k \to \infty} \frac{1}{N_k} \log \E_{\QQ_{N_k}} \big[ e^{\eta N_kF_{k}(\mu_{\sigma})} \big] < \infty,
			\end{align}
			then $\limsup_{k\to\infty} \ell_k \le \ell$.
		\end{enumerate}
		In particular, if $F_k(\mu_k) \to F(\mu)$ whenever $\mu_k \to \mu$ and $\KL(\mu\Vert u) < \infty$, and \eqref{eq: exponential_UI} holds, then $\lim_{k \to \infty} \ell_k = \ell$.
	\end{lemma}
	
	\begin{proof}
		Both parts use the following consequence of a diagonal argument, which relies on the metrizability of $\cM$: if $\liminf_k F_k(\mu_k) \ge F(\mu)$ for every sequence $\mu_k \to \mu$, then for every $\delta > 0$ there is an open neighborhood $\cU$ of $\mu$ with $\liminf_k \inf_{\nu \in \cU} F_k(\nu) \ge F(\mu) - \delta$, and similarly for upper limits.
		
		(i) Let $s < \ell$, choose $\mu^* \in \cM$ with $\KL(\mu^* \Vert u) < \infty$ and $F(\mu^*) - \KL(\mu^* \Vert u) > s$, and fix $\delta > 0$. By the assumption of (i) and the diagonal argument above, $\mu^*$ has an open neighborhood $\cU^*$ with $\liminf_{k} \inf_{\cU^*} F_{k} \geq F(\mu^*) - \delta$. Since $\ell_k \ge \inf_{\cU^*} F_k + \frac{1}{N_k}\log \QQ_{N_k}(\cU^*)$, Theorem~\ref{thm: permutations_LDP} gives
		\begin{align}
			\liminf_{k \to \infty} \ell_k \geq F(\mu^*) - \delta - \inf_{\mu \in \cU^*} \KL(\mu \Vert u)  \geq  F(\mu^*) - \KL(\mu^* \Vert u) - \delta > s - \delta.
		\end{align}
		Letting $\delta \downarrow 0$ and $s \uparrow \ell$ proves (i).
		
		(ii) Replacing $F_k$ and $F$ by $\max\{F_k,-t\}$ and $\max\{F,-t\}$ preserves the assumptions of (ii), does not decrease $\ell_k$, and replaces $\ell$ by $\max\{\ell,-t\}$, since $\KL(u\Vert u) = 0$; letting $t \to \infty$ at the end, we may therefore assume that $F > -\infty$. For $\omega < \infty$, let $\cC(\omega) := \{ \mu \in \cM : \KL(\mu \Vert u) \leq \omega \}$ be a sublevel set of the rate function; it is compact, since the rate function is good. Fix $\delta > 0$. By the assumption of (ii), the diagonal argument above and the lower semicontinuity of $\mu \mapsto \KL(\mu \Vert u)$, every $\mu \in \cC(\omega)$ has an open neighborhood $\cU_{\mu}$ with
		\begin{align}
			\label{eq: pf_varadhan_generalization_2}
			\limsup_{k \to \infty} \sup_{\nu \in \overline{\cU_{\mu}}} F_{k}(\nu) \leq F(\mu) + \delta \quad\text{and}\quad  \inf_{\nu \in \overline{\cU_{\mu}}} \KL(\nu \Vert u) \geq \KL(\mu \Vert u)  - \delta.
		\end{align}
		By compactness, there are $\mu_{1}, \ldots, \mu_{L} \in \cC(\omega)$ such that $\cU_{\mu_1}, \ldots, \cU_{\mu_L}$ cover $\cC(\omega)$. The set $\cW := \cM \setminus \bigcup_{l = 1}^{L}\cU_{\mu_l}$ is closed and disjoint from $\cC(\omega)$, so $\inf_{\cW} \KL(\cdot\Vert u) \ge \omega$. Let $\ell_\eta$ be the left-hand side of \eqref{eq: exponential_UI} and $\eta' := \eta/(\eta-1)$. By H\"older's inequality,
		\begin{align}
			\E_{\QQ_{N_k}} \big[ e^{N_k F_{k}(\mu_{\sigma})} \big] &\leq \sum_{l = 1}^{L} \exp\Big(N_k\sup_{\overline{\cU_{\mu_l}}} F_k\Big)\,\QQ_{N_k}(\overline{\cU_{\mu_l}}) + \E_{\QQ_{N_k}} \big[ e^{\eta N_k F_{k}(\mu_{\sigma})}\big]^{1/\eta} \QQ_{N_k}(\cW)^{1/\eta'}.
		\end{align}
		By Theorem~\ref{thm: permutations_LDP} and \eqref{eq: pf_varadhan_generalization_2}, the exponential growth rate of the $l$th term in the sum is at most $F(\mu_l) + \delta - \KL(\mu_l\Vert u) + \delta \le \ell + 2\delta$, and that of the last term is at most $\ell_\eta/\eta - \omega/\eta'$. Hence
		\begin{align}
			\limsup_{k \to \infty} \ell_k \leq \max \Big\{ \ell + 2\delta,\ \frac{\ell_\eta}{\eta} - \frac{\omega}{\eta'} \Big\}.
		\end{align}
		Letting $\delta \downarrow 0$ and then $\omega \to \infty$ proves (ii).
	\end{proof}
	
	\subsection{Proof of Theorem~\ref{thm: conditioned_model_weak_as_limit_no_rate}}
	
	The proof rests on the following form of the Gibbs measure in permuton coordinates, stated in \eqref{eq: permutation_measure_continuous}.
	
	\begin{lemma}[Gibbs measure in permuton coordinates]
		\label{lem: gibbs_permuton_form}
		The Gibbs measure \eqref{eq: gibbs_intro} satisfies \eqref{eq: permutation_measure_continuous}.
	\end{lemma}
	
	\begin{proof}
		By Lemma~\ref{lem: quantile_properties}, $\kappa \circ T$ is equal to the constant $\mathbf{K}_{Q_{\r}(k)Q_{\c}(\ell)}$ on $B_{k\ell}$, where $\mathbf{K}_{ij} := \log\frac{\r(i)\c(j)\norm{\bLambda}_1}{N^2\bLambda_{ij}}$ is the value of $\kappa$ on $A_{ij}$. Since $\mu_\sigma$ is uniform on the squares $B_{k\sigma(k)}$, each of mass $1/N$,
		\begin{align}
			N\int \kappa \circ T\, d\mu_\sigma = \sum_{k=1}^{N} \mathbf{K}_{Q_{\r}(k)Q_{\c}(\sigma(k))} = \sum_{i,j} \mathbf{K}_{ij}\mathbf{X}_{ij}(\sigma).
		\end{align}
		Using that $\mathbf{X}(\sigma)$ has row sums $\r$ and column sums $\c$ (Proposition~\ref{prop: permutation_CT_coupling}(i)), we obtain
		\begin{align}
			\label{eq: gibbs_weight_product}
			e^{-N\int \kappa\circ T\,d\mu_\sigma} = \prod_{i,j}\Big(\frac{N^2 \bLambda_{ij}}{\norm{\bLambda}_1\r(i)\c(j)}\Big)^{\mathbf{X}_{ij}(\sigma)} = \frac{(N^2/\norm{\bLambda}_1)^N}{\prod_{i}\r(i)^{\r(i)}\prod_{j}\c(j)^{\c(j)}}\prod_{i,j}\bLambda_{ij}^{\mathbf{X}_{ij}(\sigma)},
		\end{align}
		and $\prod_{i,j}\bLambda_{ij}^{\mathbf{X}_{ij}(\sigma)} = \prod_{k=1}^{N}\widetilde{\bLambda}_{k\sigma(k)}$ by \eqref{eq: block_weight}. Hence the right-hand side of \eqref{eq: permutation_measure_continuous} is proportional to $\prod_{k}\widetilde{\bLambda}_{k\sigma(k)}$. Since both sides of \eqref{eq: permutation_measure_continuous} are probability measures on $S_N$, they coincide.
	\end{proof}
	
	The following lemma will be used to verify the conditions of Lemma~\ref{lem: Varadhan_generalization}.
	
	\begin{lemma}
		\label{lem: convergence_of_pushfowards}
		Let $\rho^{0}_k$ and $\rho^{1}_k$ be probability measures on $[0,1]$ such that $\rho^{0}_k \to \rho^{0}$ and $\rho^{1}_k \to \rho^{1}$ weakly. Let $T_k := (q_{\rho_k^{0}}, q_{\rho_k^{1}})$ and $T := (q_{\rho^0},q_{\rho^1})$. If $\mu_k \in \cM$ and $\mu_k \to \mu$ weakly, then
		\begin{align}
			\lim_{k \to \infty} \int_{[0, 1]^2} g \circ T_k \, d\mu_k  = \int_{[0, 1]^2} g \circ T \, d\mu \qquad \text{for every } g \in C([0, 1]^2).
		\end{align}
		Equivalently, the push-forwards $(T_k)_{\#} \mu_k$ converge weakly to $T_{\#} \mu$.
	\end{lemma}
	
	\begin{proof}
		Note that $\mu \in \cM$, since $\cM$ is weakly closed. By the triangle inequality,
		\begin{align}
			\left|\int g \circ T_k \, d\mu_k  - \int g \circ T \, d\mu \right| \leq \left| \int (g \circ T_k - g \circ T)\,d\mu_k \right| + \left | \int g \circ T \, d\mu_k - \int g \circ T \, d\mu \right|.
		\end{align}
		For the second term, the set of discontinuity points of $g \circ T$ has $\mu$-measure zero, so
		\begin{align}
			\label{eq: pushfwd_weak_convergence_pf_2}
			\int g \circ T \, d\mu_k \to \int g \circ T \, d\mu
		\end{align}
		by the Portmanteau theorem. To prove this claim, note that the discontinuity points of $g\circ T$ are discontinuity points of $T$. Since $q_{\rho^{0}}$ and $q_{\rho^{1}}$ are non-decreasing, each of them has at most countably many discontinuity points, say $\{x_i\}$ and $\{y_i\}$. The discontinuity set of $T$ is then contained in $D := \bigcup_{i} (\{x_i\} \times [0, 1]) \cup \bigcup_{i}  ([0, 1] \times \{y_i\})$, and $\mu(D) = 0$ as $\mu$ has uniform marginals. This proves \eqref{eq: pushfwd_weak_convergence_pf_2}.
		
		For the first term, recall that $q_{\rho_k^{0}} \to q_{\rho^{0}}$ and $q_{\rho_k^{1}} \to q_{\rho^{1}}$ at every continuity point of the limits \cite[Lemma~21.2]{vandervaart1998asymptotic}, hence Lebesgue-almost everywhere, and therefore in Lebesgue measure. Fix $\eps > 0$. By uniform continuity of $g$, there is $\eta > 0$ such that $|g\circ T_k - g \circ T| < \eps$ outside the set
		\begin{align}
			E_{k} := \{ |T_k - T| \geq \eta\} \subset \big(\{|q_{\rho_k^{0}} - q_{\rho^{0}}| \geq \eta/2\} \times[0, 1]\big) \cup \big([0, 1] \times \{|q_{\rho_k^{1}} - q_{\rho^{1}}| \geq \eta/2\}\big).
		\end{align}
		Since $\mu_k \in \cM$,
		\begin{align}
			\mu_k(E_k) \leq \Leb\big(|q_{\rho_k^{0}} - q_{\rho^{0}}| \geq \eta/2\big) +  \Leb\big(|q_{\rho_k^{1}} - q_{\rho^{1}}| \geq \eta/2\big) \to 0.
		\end{align}
		Hence
		\begin{align}
			\left| \int (g \circ T_k - g \circ T)\,d\mu_k \right| \leq \eps + 2\norm{g}_{\infty} \mu_k(E_k) = \eps + o(1).
		\end{align}
		Since $\eps$ was arbitrary, the proof is complete.
	\end{proof}

	\begin{proof}[\textbf{Proof of Theorem~\ref{thm: conditioned_model_weak_as_limit_no_rate}}]
		(i) By \eqref{eq: def_partition_function},
		\begin{align}
			\frac{1}{N_k!}\cZ_k(\kappa_k) = \frac{1}{N_k!} \sum_{\sigma \in S_{N_k}} \exp \left( -N_k \int \kappa_k \circ T_k\, d\mu_{\sigma} \right) = \E_{\QQ_{N_k}} \left[ e^{ -N_k \int \kappa_k \circ T_k\, d\mu_{\sigma}} \right].
		\end{align}
		Let $F_k(\mu) := -\int \kappa_k \circ T_k\,d\mu$ and $F(\mu) := -\int \kappa_\infty\circ T_\infty\,d\mu$. The moment condition \eqref{eq: exponential_UI} holds trivially, since $|F_k| \le \sup_k\norm{\kappa_k}_\infty < \infty$. If $\mu_k \to \mu$ in $\cM$, then
		\begin{align}
			|F_k(\mu_k) - F(\mu)| \le \norm{\kappa_k - \kappa_\infty}_\infty + \Big|\int \kappa_\infty\circ T_k\,d\mu_k - \int \kappa_\infty\circ T_\infty\,d\mu\Big| \to 0
		\end{align}
		by Assumptions~\ref{assumption: bounded_cost} and \ref{assumption: margins_weak_convergence} and Lemma~\ref{lem: convergence_of_pushfowards}. Thus Lemma~\ref{lem: Varadhan_generalization} gives
		\begin{align}
			\label{eq: varadhan_limit}
			\lim_{k \to \infty}\frac{1}{N_k} \log \frac{\cZ_k(\kappa_k)}{N_k!} = \sup_{\mu \in \cM} \left( -\int \kappa_{\infty} \circ T_{\infty}\, d\mu - \KL(\mu \Vert u) \right) = -\Gamma_\infty,
		\end{align}
		where the last equality is Lemma~\ref{lem: COV_permuton_to_EOT}. This proves \eqref{eq: log_parition_limit}.
		
		(ii) Let $\mu^*_\infty$ be the minimizer of $\mu \mapsto \int \kappa_\infty\circ T_\infty\,d\mu + \KL(\mu\Vert u)$ over $\cM$, so that $\pi_\infty = (T_\infty)_{\#}\mu^*_\infty$ by Lemma~\ref{lem: COV_permuton_to_EOT}. We first show that $\mu_\sigma \to \mu^*_\infty$ in probability when $\sigma \sim \Gibbs_k$. Fix an open neighborhood $\cU$ of $\mu^*_\infty$ in $\cM$ and let $\widetilde F_k := F_k$ and $\widetilde F := F$ on $\cM\setminus\cU$, and $\widetilde F_k := \widetilde F := -\infty$ on $\cU$. By Lemma~\ref{lem: gibbs_permuton_form},
		\begin{align}
			\Gibbs_k(\mu_{\sigma} \notin \cU) = \frac{\E_{\QQ_{N_k}} \big[ e^{N_k \widetilde F_k(\mu_{\sigma})} \big]}{\E_{\QQ_{N_k}} \big[ e^{N_k F_k(\mu_{\sigma})} \big]}.
		\end{align}
		If $\mu_k \to \mu \in \cU$, then $\mu_k \in \cU$ for all large $k$ and hence $\widetilde F_k(\mu_k) = -\infty$; if $\mu_k \to \mu \notin \cU$, then $\limsup_k \widetilde F_k(\mu_k) \le \lim_k F_k(\mu_k) = F(\mu)$. Applying Lemma~\ref{lem: Varadhan_generalization}(ii) to the numerator (the moment condition holds as $\widetilde F_k \le F_k$) and \eqref{eq: varadhan_limit} to the denominator, we get
		\begin{align}
			\label{eq: pf_permuton_weak_convergence_2}
			\limsup_{k \to \infty} \frac{1}{N_k} \log \Gibbs_k(\mu_{\sigma} \notin \cU) \leq \sup_{\mu \in \cM\setminus\cU} \big(F(\mu) - \KL(\mu \Vert u)\big) - \sup_{\mu \in \cM} \big( F(\mu) - \KL(\mu \Vert u) \big).
		\end{align}
		The functional $F - \KL(\cdot\Vert u)$ is upper semicontinuous on $\cM$ ($F$ is continuous by Lemma~\ref{lem: convergence_of_pushfowards}), so it attains its supremum on the compact set $\cM\setminus\cU$. Its unique maximizer over $\cM$ is $\mu^*_\infty \in \cU$. Therefore the right-hand side of \eqref{eq: pf_permuton_weak_convergence_2} is strictly negative, and $\Gibbs_k(\mu_\sigma \notin \cU) \to 0$. Since $\cU$ was arbitrary, $\mu_{\sigma} \to \mu^*_\infty$ in probability.
		
		Now fix $g \in C([0,1]^2)$. By Lemma~\ref{lem: convergence_of_pushfowards} and a diagonal argument, for every $\eps > 0$ there are $k_0$ and a neighborhood $\cU$ of $\mu^*_\infty$ such that $|\int g\circ T_k\,d\mu - \int g\circ T_\infty\,d\mu^*_\infty| < \eps$ for all $k \ge k_0$ and $\mu \in \cU$. Hence $\int g\circ T_k\,d\mu_\sigma \to \int g\circ T_\infty\,d\mu^*_\infty = \int g\,d\pi_\infty$ in probability. Finally, by Proposition~\ref{prop: permutation_CT_coupling}(i) and (iii), $\mathbf{X}_k$ can be realized as $\mathbf{X}(\sigma)$ with $\sigma \sim \Gibbs_k$, and then $|\int g\,d\pi_{\mathbf{X}_k} - \int g\circ T_k\,d\mu_\sigma| \le \omega_g\big(\sqrt{m_k^{-2}+n_k^{-2}}\big) \to 0$ by Assumption~\ref{assumption: margins_weak_convergence}. This proves (ii).
	\end{proof}

	\section{Proof of the rate of convergence}
	\label{sec: proof_rates}
	
	In this section we prove Theorem~\ref{thm: conditioned_model_weak_as_limit} on the rate of convergence. The random part of the error is controlled in Section~\ref{sec: rates_concentration} and the deterministic part in Section~\ref{sec: rates_stability}.
	
	\subsection{Concentration around the Schr\"odinger bridge}
	\label{sec: rates_concentration}
	
	Fix a prior mean $\bLambda \in \R^{m \times n}_{>0}$ and positive integer margins $(\r, \c)$ with $\sum_i\r(i) = \sum_j \c(j) = N$. Let $\mathbf{Z} = \mathbf{Z}^{\r,\c}$ be the Schr\"odinger bridge \eqref{eq: def_matrix_scaling}, let $\mathbf{Y} \sim \Poisson(\mathbf{Z})$ be the tilted model, and let $\mathbf{X} \sim \lambda_{\r,\c;\bLambda}$.
	
	\begin{lemma}[Bernstein inequality for Poisson random variables]
		\label{lem: Poisson_Berstein}
		Let $\xi_1, \ldots, \xi_d$ be independent centered Poisson random variables, i.e., $\xi_i + \lambda_i \sim \Poisson(\lambda_i)$, and let $\lambda_{\max} := \max_i \lambda_i$. Then for every $a = (a_1, \ldots, a_d) \in \R^d \setminus\{0\}$ and $t \ge 0$,
		\begin{align}
			\P\left( \Big| \sum_{i = 1}^{d} a_i \xi_i \Big| \geq \norm{a}_2^2\,t \right) \leq 2 \exp \left( - \left( \frac{\norm{a}_2^2t^2}{4\lambda_{\max}} \wedge \frac{3\norm{a}_2^2 t}{4\norm{a}_{\infty}}\right) \right).
		\end{align}
	\end{lemma}
	
	\begin{proof}
		Since $\log \E[e^{s\xi_i}] = \lambda_i(e^s - s - 1) \le \lambda_i s^2/(2(1 - |s|/3))$ for $|s| < 3$, the sum $\sum_i a_i\xi_i$ and its negative are sub-gamma with variance factor $\lambda_{\max}\norm{a}_2^2$ and scale parameter $\norm{a}_\infty/3$. The claim therefore follows from Bernstein's inequality for sub-gamma random variables \cite[Section~2.4]{boucheron2013concentration}, $\P(|\sum_i a_i\xi_i| \ge x) \le 2\exp\big(-x^2/(2(\lambda_{\max}\norm{a}_2^2 + x\norm{a}_\infty/3))\big)$, applied with $x = \norm{a}_2^2\,t$, together with $2(\lambda_{\max} + t\norm{a}_\infty/3) \le 4\max\{\lambda_{\max}, t\norm{a}_\infty/3\}$.
	\end{proof}

	\begin{lemma}
		\label{lem: weak_concentration_random_hist}
		Suppose that $\norm{\bLambda}_{1} \geq c\,mn \norm{\bLambda}_{\infty}$ for some $c \in (0, 1]$. For $g \in L^\infty([0,1]^2)$ and $\mathbf{A} \in \R^{m\times n}$, let $L_g(\mathbf{A}) := \frac{1}{N}\sum_{i,j}\mathbf{G}_{ij}\mathbf{A}_{ij}$ with $\mathbf{G}_{ij} = mn\int_{A_{ij}} g$, so that $L_g(\mathbf{A}) = \int g\,d\pi_{\mathbf{A}}$ whenever $\norm{\mathbf{A}}_1 = N$. Then for every $s \ge 0$,
		\begin{align}
			\label{eq: independence_model_histogram_concentration}
			\P \big( | L_g(\mathbf{Y}) - L_g(\mathbf{Z}) | \geq s \big) \leq 2\exp \left( - \left( \frac{c N s^2}{4 e^{5\norm{\kappa}_\infty}\norm{g}_{L^2}^2} \wedge \frac{3Ns}{ 4\norm{g}_{L^{\infty}}} \right) \right),
		\end{align}
		and
		\begin{align}
			\label{eq: conditioned_model_histogram_concentration}
			\P \left( \left| \int g \, d(\pi_{\mathbf{X}} - \pi_{\mathbf{Z}}) \right| \geq s \right) \leq \frac{(eN)^{m+n}}{m^mn^n}\cdot 2\exp \left(- \left( \frac{c N s^2}{4 e^{5\norm{\kappa}_\infty}\norm{g}_{L^2}^2} \wedge \frac{3Ns}{ 4\norm{g}_{L^{\infty}}} \right) \right),
		\end{align}
		where $\kappa$ is the cost induced by $\pi_{\bLambda}$ as in \eqref{eq: def_kappa}.
	\end{lemma}
	
	\begin{proof}
		We have $L_g(\mathbf{Y}) - L_g(\mathbf{Z}) = \frac{1}{N}\sum_{i,j} \mathbf{G}_{ij}\bar{\mathbf{Y}}_{ij}$, where $\bar {\mathbf{Y}} := \mathbf{Y} - \mathbf{Z}$ has independent centered Poisson entries with parameters $\mathbf{Z}_{ij} \le \lambda := \max_{i,j} \mathbf{Z}_{ij}$. Lemma~\ref{lem: Poisson_Berstein} with $a = \mathbf{G}$ and $t = Ns/\norm{\mathbf{G}}_2^2$ gives
		\begin{align}
			\P\big(|L_g(\mathbf{Y}) - L_g(\mathbf{Z})| \ge s\big) = \P \Big( \Big|\sum_{i,j}\mathbf{G}_{ij} \bar{\mathbf{Y}}_{ij}\Big| \geq \norm{\mathbf{G}}_{2}^2 t \Big) \leq 2 \exp \left( - \left( \frac{N^2s^2}{4\lambda\norm{\mathbf{G}}_2^2} \wedge \frac{3Ns}{4 \norm{\mathbf{G}}_{\infty}}\right) \right).
		\end{align}
		By the Cauchy--Schwarz inequality, $\mathbf{G}_{ij}^2 \le mn\int_{A_{ij}}g^2$, so $\norm{\mathbf{G}}_{2}^2 \leq mn\norm{g}_{L^2}^2$; moreover $\norm{\mathbf{G}}_{\infty} \leq \norm{g}_{L^{\infty}}$. Finally, $\frac{N\mathbf{Z}_{ij}}{\r(i)\c(j)} \le e^{2\osc(\kappa)}$ by Lemmas~\ref{lem: typical_table_bridge} and \ref{lem: density_bounds}, and $\frac{\r(i)\c(j)}{N\bLambda_{ij}} = e^{\kappa_{ij}}\frac{N}{\norm{\bLambda}_1}$ by \eqref{eq: def_kappa}, where $\kappa_{ij}$ is the value of $\kappa$ on $A_{ij}$. Since $2\osc(\kappa) + |\kappa_{ij}| \le 5\norm{\kappa}_\infty$,
		\begin{align}
			\lambda = \max_{i,j}\bLambda_{ij}\,\frac{N\mathbf{Z}_{ij}}{\r(i)\c(j)}\,\frac{\r(i)\c(j)}{N\bLambda_{ij}} \leq \norm{\bLambda}_{\infty}\,e^{5\norm{\kappa}_\infty}\frac{N}{\norm{\bLambda}_1} \le \frac{e^{5\norm{\kappa}_\infty}N}{c\,mn}.
		\end{align}
		Substituting these bounds gives \eqref{eq: independence_model_histogram_concentration}. Since $\norm{\mathbf{X}}_1 = N$, \eqref{eq: conditioned_model_histogram_concentration} follows from \eqref{eq: independence_model_histogram_concentration} and the transference principle (Lemma~\ref{thm: transference}) applied to $h(\mathbf{A}) := \1\big( | L_g(\mathbf{A}) - L_g(\mathbf{Z}) | \geq s \big)$.
	\end{proof}

	\subsection{Stability of Schr\"odinger bridges}
	\label{sec: rates_stability}
	
	The deterministic term in \eqref{eq: rate_bound} comes from the following stability estimate for static Schr\"odinger bridges. The sharper bound of \cite[Theorem~2.15]{duan2026scaling} (Section~\ref{sec: bg_EOT}) requires marginal densities bounded away from zero and infinity, which Theorem~\ref{thm: conditioned_model_weak_as_limit} does not assume, so we give a short proof that uses only bounded density ratios $R/(\mu\otimes\nu)$.
	
	For finite measures $P, Q$ on $[0,1]^2$ with densities $p, q$ with respect to $u$, let $H(P,Q) := \big(\int(\sqrt{p}-\sqrt{q})^2\big)^{1/2}$; thus $H(P,Q) = d_H(p,q)$ for probability measures, and $H$ satisfies the triangle inequality. 
	
	In the following proof, as well as in the proof of Lemma~\ref{lem: quadratic_lower_bound}, we use the following duality. Let $R := e^{-h}(\mu\otimes\nu)$ for a bounded function $h$, and let $\pi$ be the Schr\"odinger bridge, i.e., the minimizer of $\KL(\cdot\Vert R)$ over $\Pi(\mu,\nu)$. Then
	\begin{align}
		\label{eq: SSB_duality}
		\KL(\pi\Vert R) = \max_{\tilde a \in L^\infty(\mu),\, \tilde b \in L^\infty(\nu)} \Big( \int \tilde a\,d\mu + \int\tilde b\,d\nu - \int e^{\tilde a\oplus\tilde b}\,dR + 1 \Big),
	\end{align}
	and the maximum is attained at the Schr\"odinger potentials $(a,b)$, for which $d\pi = e^{a\oplus b}\,dR$ by \eqref{eq: schrodinger_form}. Indeed, for bounded $\tilde a, \tilde b$, Fenchel's inequality $\alpha\log\alpha - \beta\alpha \ge \alpha - e^{\beta}$, applied with $\alpha = d\pi/dR$ and $\beta = \tilde a \oplus \tilde b$ and integrated against $R$, gives the inequality ``$\ge$'' in \eqref{eq: SSB_duality}, and equality holds at $(a,b)$ since $\int e^{a\oplus b}\,dR = \pi([0,1]^2) = 1$ and $\KL(\pi\Vert R) = \int (a\oplus b)\,d\pi = \int a\,d\mu + \int b\,d\nu$; see also \cite[Section~2]{leonard2013survey}.
	
	\begin{prop}[Stability of static Schr\"odinger bridges w.r.t. cost]
		\label{prop: hellinger_stability}
		Let $\mu, \nu, \mu', \nu'$ be a.e.\ positive probability densities on $[0,1]$, and $R, R'$ probability densities on $[0,1]^2$ with, for some $M \ge 1$,
		\begin{align}
			M^{-1} \le \frac{R}{\mu\otimes\nu} \le M \quad\text{and}\quad M^{-1} \le \frac{R'}{\mu'\otimes\nu'} \le M \qquad \text{a.e.}
		\end{align}
		Let $\pi$ and $\pi'$ minimize $\KL(\cdot \Vert R)$ over $\Pi(\mu,\nu)$ and $\KL(\cdot\Vert R')$ over $\Pi(\mu',\nu')$. Then
		\begin{align}
			\label{eq: hellinger_stability}
			H(\pi,\pi')^2 \le 40\log M\,\big(\norm{\mu-\mu'}_{L^1} + \norm{\nu - \nu'}_{L^1}\big) + 11M^5\,d_H(R,R').
		\end{align}
	\end{prop}
	
	\begin{proof}
		Our assumptions imply that all measures are considered absolutely continuous with respect to the Lebesgue measure with strictly positive densities. Therefore, in the proof, all integrals are taken with respect to the Lebesgue measure and other measures are identified with their densities. We identify Radon-Nykodym derivatives with the ratio of densities.
		
		Let $c := -\log\frac{R}{\mu\otimes\nu}$, so that $\KL(\cdot \Vert R) = \int c\,d(\cdot) + \KL(\cdot\Vert\mu\otimes\nu)$ on $\Pi(\mu,\nu)$ and $\osc(c) \le 2\log M$, and similarly for $c'$. By \eqref{eq: schrodinger_form}, there are bounded potentials with $\pi = e^{a\oplus b}R$ and $\pi' = e^{a'\oplus b'}R'$. By Lemma~\ref{lem: density_bounds}, $e^{a\oplus b} = e^{c}\frac{\pi}{\mu\otimes\nu} \in [M^{-5}, M^5]$ a.e. Adding a constant to $a$ and subtracting it from $b$, we may assume that $\operatorname{ess\,sup} a = 0$. Then $b \le 5\log M$ and $b \ge -5\log M - a \ge -5\log M$, and hence $a \ge -5\log M - \operatorname{ess\,sup} b \ge -10\log M$. Thus $\norm{a}_\infty \le 10\log M$ and $\norm{b}_\infty \le 5\log M$, and the same bounds hold for $(a', b')$ after the analogous normalization.
		
		For bounded $(\tilde a, \tilde b)$, define the dual functionals
		\begin{align}
			\cL(\tilde a,\tilde b) := \int\tilde a\,\mu + \int \tilde b\,\nu - \int e^{\tilde a\oplus\tilde b}R + 1, \qquad \cL'(\tilde a,\tilde b) := \int\tilde a\,\mu' + \int \tilde b\,\nu' - \int e^{\tilde a\oplus\tilde b}R' + 1,
		\end{align}
		and let $V := \KL(\pi\Vert R)$ and $V' := \KL(\pi'\Vert R')$.
		By \eqref{eq: SSB_duality}, 
		$\cL(\tilde a, \tilde b) \leq V$, with equality at the Schr\"odinger potentials $(a, b)$, and similarly for $\mathcal{L}'$ and $V'$. 
		
		Let $\psi(x) = x\log x - x + 1$. Fix bounded $(\tilde a,\tilde b)$ and define $Q := e^{\tilde a\oplus\tilde b}R'$. A direct computation using $\pi' \in \Pi(\mu',\nu')$ gives the duality gap identity
		\begin{align}
			V' - \cL'(\tilde a,\tilde b) &= \int \pi' \log \frac{\pi'}{R'} - \int \pi' \log \frac{Q}{R'} + \int e^{\tilde a \oplus \tilde b}R' + 1\\
			&=  \int \pi' \log \frac{\pi'}{Q} + \int e^{\tilde a \oplus \tilde b}R' - \int e^{a' \oplus b'}R'\\
			&= Q \psi\left(\frac{\pi'}{Q}\right) \geq \int \left(\sqrt{\pi'} - \sqrt{Q} \right)^2 = H(\pi', Q)^2
		\end{align}
		where we used $\psi(x) \ge (\sqrt{x}-1)^2$, which follows from $2\sqrt{x}\psi(\sqrt{x}) \geq 0$.
		
		Let $\cE(\tilde a,\tilde b) := \cL'(\tilde a,\tilde b) - \cL(\tilde a,\tilde b)$. Then $V' = \cL'(a',b') = \cL(a',b') + \cE(a',b') \le V + \cE(a',b')$ and $\cL'(a,b) = V + \cE(a,b)$, so that, 
		\begin{align}
			H(\pi',Q)^2 \le V' - \cL'(a,b) \le |\cE(a,b)| + |\cE(a',b')|.
		\end{align}
		By the bounds on the potentials,
		\begin{align}
			|\cE(a,b)| + |\cE(a',b')| \le 20\log M\big(\norm{\mu-\mu'}_{L^1} + \norm{\nu-\nu'}_{L^1}\big) + 2M^5\norm{R - R'}_{L^1}.
		\end{align}
		Since $\pi = e^{a\oplus b}R$, we have $H(Q,\pi) \le M^{5/2}H(R',R)$. By the triangle inequality,
		\begin{align}
			H(\pi,\pi')^2 &\le 2H(\pi',Q)^2 + 2M^5 H(R,R')^2 \\
			&\le 40\log M\big(\norm{\mu-\mu'}_{L^1} + \norm{\nu-\nu'}_{L^1}\big) + 4M^5\norm{R-R'}_{L^1} + 2M^5H(R,R')^2.
		\end{align}
		Finally, $\norm{R - R'}_{L^1} \le H(R,R')\,\norm{\sqrt{R} + \sqrt{R'}}_{L^2} \le 2H(R,R')$, and $H(R,R')^2 \le \sqrt 2\, H(R,R')$ because $H(R,R') \le \sqrt 2$ for probability densities. Since $8 + 2\sqrt 2 \le 11$, this proves \eqref{eq: hellinger_stability}.
	\end{proof}

	\subsection{Proof of Theorem~\ref{thm: conditioned_model_weak_as_limit}}

	\begin{proof}[\textbf{Proof of Theorem~\ref{thm: conditioned_model_weak_as_limit}}]
		We use the triangle inequality
		\begin{align}
			\left| \int g\, d(\pi_{\mathbf{X}_k} - \pi_\infty)\right| \leq \left| \int g\,d(\pi_{\mathbf{X}_k} - \pi_{\mathbf{Z}_k}) \right| + \left| \int g\,d(\pi_{\mathbf{Z}_k} - \pi_\infty) \right|.
		\end{align}
		
		By Lemma~\ref{lem: typical_table_bridge} and \eqref{eq: EOT_is_SB}, $\pi_{\mathbf{Z}_k}$ is the Schr\"odinger bridge between $\rho_{\r_k}$ and $\rho_{\c_k}$ with reference density $R_k := \varphi_{\bLambda_k}$, and $R_k/(\rho_{\r_k}\otimes\rho_{\c_k}) = e^{-\kappa_k} \in [M^{-1},M]$. As discussed before Theorem~\ref{thm: conditioned_model_weak_as_limit}, $\pi_\infty$ is the Schr\"odinger bridge between $\rho_{\r_\infty}$ and $\rho_{\c_\infty}$ with reference density $\varphi_{\bLambda_\infty}$, and $\varphi_{\bLambda_\infty}/(\rho_{\r_\infty}\otimes\rho_{\c_\infty}) = e^{-\kappa_\infty} \in [M^{-1},M]$. By Proposition~\ref{prop: hellinger_stability} and $\norm{\pi_{\mathbf{Z}_k} - \pi_\infty}_{L^1} \le 2H(\pi_{\mathbf{Z}_k},\pi_\infty)$, the second term is bounded by the second term in \eqref{eq: rate_bound}.
		
		For the first term, let $s_k := \sqrt{(m_k+n_k)^{1+\eps}\log N_k/N_k}$. Since $N_k s_k^2 = (m_k+n_k)^{1+\eps}\log N_k$, $N_ks_k \ge (m_k+n_k)^{1+\eps}\log N_k$ (as $s_k \le 1$), $\norm{\kappa_k}_\infty \le \log M$ and $\norm{g}_{L^2} \le \norm{g}_\infty$, Lemma~\ref{lem: weak_concentration_random_hist} gives
		\begin{align}
			\P\left( \left| \int g\,d(\pi_{\mathbf{X}_k} - \pi_{\mathbf{Z}_k}) \right| \ge s_k\right) \le 2\exp\Big((m_k+n_k)\log(eN_k) - C_0(m_k+n_k)^{1+\eps}\log N_k\Big),
		\end{align}
		with $C_0 := \min\big\{c/(4M^5\norm{g}_\infty^2),\ 3/(4\norm{g}_\infty)\big\}$. For large $k$, the right-hand side is at most $2e^{-CN_ks_k^2}$ with $C := C_0/2$, which proves the theorem.
	\end{proof}

	\section{Proofs of the partition function asymptotics}
	\label{sec: partition_function}

	In this section we prove Theorem~\ref{thm: permuton_partition_det_limit} and Corollary~\ref{cor: enumeration}, together with uniform bounds on the normalized partition function (Proposition~\ref{prop: uniform_partition_bounds}) that are used in the proofs of the fluctuation results in Section~\ref{sec: proof_fluctuations}. Theorem~\ref{thm: permuton_partition_det_limit} extends the main result of \cite{tripathi2026partitionfunctionclassmallows}, and much of the proof is adapted from there. The proof has three steps. An exact identity expresses the normalized partition function as a normalized permanent of a doubly stochastic matrix (Section~\ref{sec: pf_exact}). McCullagh's expansion of such permanents turns it into a determinant, which equals the Fredholm determinant of an integral operator on $L^2[0,1]$ (Section~\ref{sec: pf_permanents}). Finally, these operators converge in Hilbert--Schmidt norm (Section~\ref{sec: pf_proof}).
	
	Throughout this section, except in Section~\ref{sec: pf_proof}, we fix positive integer margins $(\r,\c)$ with $\sum_i \r(i) = \sum_j \c(j) = N$, and we write $\Gamma := \Gamma_{\rho_{\r},\rho_{\c}}$ and $\cD(f)$ for the normalized partition function \eqref{eq: def_normalized_partition} with the index $k$ dropped. For a grid-step function $f$, let $\xi_f := d\pi^f_{\rho_{\r},\rho_{\c}}/d(\rho_{\r}\otimes\rho_{\c})$. If $(a,b)$ are Schr\"odinger potentials for $f$, the first equation in \eqref{eq: schrodinger_system} shows that $e^{-a(x)}$ depends on $x \in \big(\frac{i-1}{m},\frac{i}{m}\big]$ only through $i$, and similarly for $b$. We may therefore choose potentials that are constant on the strips $\big(\frac{i-1}{m},\frac{i}{m}\big]$ and $\big(\frac{j-1}{n},\frac{j}{n}\big]$, and then $\xi_f$ is grid-step; we write $\xi_{f,ij}$ for its value on $A_{ij}$. We define the $N \times N$ matrix
	\begin{align}
		\bD_f := \frac{1}{N}\widetilde{\xi_f}, \qquad \text{i.e.,}\qquad (\bD_f)_{k\ell} = \frac{1}{N}\,\xi_{f, Q_{\r}(k)Q_{\c}(\ell)},
	\end{align}
	and for $n \in \N$ we write $\bJ_n$ for the $n\times n$ matrix with all entries equal to $1/n$.
	
	\subsection{An exact identity for grid-step costs}
	\label{sec: pf_exact}
	
	\begin{lemma}
		\label{lem: grid_step_identity}
		Let $f$ be a bounded grid-step function. Then $\bD_f$ is doubly stochastic, $e^{-2\osc(f)} \le N(\bD_f)_{k\ell} \le e^{2\osc(f)}$ for all $k,\ell$, and
		\begin{align}
			\label{eq: grid_step_identity}
			\cD(f) = \frac{N^N}{N!}\per(\bD_f).
		\end{align}
	\end{lemma}
	
	\begin{proof}
		For $k$ with $Q_{\r}(k) = i$, we have $\sum_{\ell}(\bD_f)_{k\ell} = \frac{1}{N}\sum_j \c(j)\,\xi_{f,ij} = \int \xi_f(x,y)\,\rho_{\c}(dy) = 1$ for $x \in \big(\frac{i-1}{m},\frac{i}{m}\big]$, by the Schr\"odinger system; the column sums are handled in the same way. The entry bounds follow from Lemma~\ref{lem: density_bounds}. Let $f_{ij}$ denote the value of $f$ on $A_{ij}$. By Lemma~\ref{lem: quantile_properties}, $f\circ T = f_{Q_{\r}(k)Q_{\c}(\ell)} = \widetilde{f}_{k\ell}$ on $B_{k\ell}$. Since $\mu_\sigma$ is uniform on the squares $B_{k\sigma(k)}$, each of mass $1/N$, we get $N\int f\circ T\,d\mu_\sigma = \sum_k \widetilde f_{k\sigma(k)}$, and hence $\cZ(f) = \per(e^{-\widetilde f})$ by \eqref{eq: def_partition_function}, where the exponential is taken entrywise. Let $(a_i)$ and $(b_j)$ be the values of the potentials, which are constant on strips, and let $\tilde a_k := a_{Q_{\r}(k)}$ and $\tilde b_\ell := b_{Q_{\c}(\ell)}$. Then $e^{-\widetilde f_{k\ell}} = \widetilde{\xi_f}_{k\ell}\,e^{-\tilde a_k - \tilde b_\ell}$, that is,
		\begin{align}
			\label{eq: D_is_scaling}
			\bD_f = \frac1N\diag(e^{\tilde a})\,e^{-\widetilde f}\,\diag(e^{\tilde b}).
		\end{align}
		Pulling the row and column factors out of the permanent, and using $\sum_k \tilde a_k = \sum_i \r(i)a_i = N\int a\,d\rho_{\r}$ and $\sum_\ell \tilde b_\ell = N\int b\,d\rho_{\c}$ together with Lemma~\ref{lem: pythagorean}, we obtain
		\begin{align}
			\cZ(f) = \per\big(e^{-\widetilde f}\big) = N^N e^{-\sum_k\tilde a_k - \sum_\ell \tilde b_\ell}\per(\bD_f) = e^{-N\Gamma(f)}N^N\per(\bD_f),
		\end{align}
		which is \eqref{eq: grid_step_identity} by the definition \eqref{eq: def_normalized_partition} of $\cD(f)$.
	\end{proof}
	
	\subsection{Permanents of doubly stochastic matrices}
	\label{sec: pf_permanents}
	
	We use the following asymptotic expansion of permanents of doubly stochastic matrices, due to McCullagh.
	
	\begin{theorem}[\cite{McCullagh2014asymptotic}]
		\label{thm: mccullogh_permanent_expansion}
		For each $n \ge 1$, let $\mathbf{A}_n$ be a doubly stochastic $n\times n$ matrix. Assume that
		\begin{align}
			\sup_{n} \frac{1}{n^2} \sum_{i,j} \big|n(\mathbf{A}_n - \bJ_n)_{ij}\big|^{p} < \infty \quad \text{for all finite } p \geq 1, \qquad
			\sup_{n} \norm{\mathbf{A}_n - \bJ_n}_{\mathrm{op}} < 1.
		\end{align}
		Then
		\begin{align}
			\frac{n^n}{n!}\per(\mathbf{A}_n) = \big(1 + O(n^{-1})\big)\det\big(\mathbf{I} + \bJ_n - \mathbf{A}_n^{\top}\mathbf{A}_n\big)^{-1/2}.
		\end{align}
	\end{theorem}
	
	For a doubly stochastic matrix $\mathbf{A}$, we have $\mathbf{A}^\top\bJ_n = \bJ_n\mathbf{A} = \bJ_n$ and $\bJ_n^2 = \bJ_n$. Hence
	\begin{align}
		\label{eq: det_rewrite}
		(\mathbf{A} - \bJ_n)^\top(\mathbf{A} - \bJ_n) = \mathbf{A}^\top \mathbf{A} - \mathbf{A}^\top\bJ_n - \bJ_n\mathbf{A} + \bJ_n = \mathbf{A}^\top \mathbf{A} - \bJ_n,
	\end{align}
	and the determinant in Theorem~\ref{thm: mccullogh_permanent_expansion} equals $\det\big(\mathbf{I} - (\mathbf{A}-\bJ_n)^\top(\mathbf{A}-\bJ_n)\big)$. For $K \ge 0$, let $\mathfrak{D}_n(K)$ denote the set of doubly stochastic $n \times n$ matrices $\mathbf{A}$ with $e^{-K} \le n\mathbf{A}_{ij} \le e^{K}$ for all $i,j$. By Lemma~\ref{lem: grid_step_identity}, $\bD_f \in \mathfrak{D}_N(2\osc(f))$. Since we apply Theorem~\ref{thm: mccullogh_permanent_expansion} along arbitrary sequences of dimensions $N_k$, we need the following uniform version.

	\begin{corollary}
		\label{cor: uniform_mccullagh}
		For every $K \ge 0$ and every $\mathbf{A} \in \mathfrak{D}_n(K)$, $\norm{\mathbf{A} - \bJ_n}_{\mathrm{op}} \le 1 - e^{-K}$. Moreover,
		\begin{align}\label{eq:uniform_mccullagh}
			\lim_{n \to \infty}\ \sup_{\mathbf{A} \in \mathfrak{D}_n(K)} \left| \frac{n^n}{n!}\per(\mathbf{A})\,\det\big(\mathbf{I} - (\mathbf{A}-\bJ_n)^\top(\mathbf{A}-\bJ_n)\big)^{1/2} - 1\right| = 0.
		\end{align}
	\end{corollary}
	
	The following lemma gives the first claim of Corollary~\ref{cor: uniform_mccullagh} and the bound on $\norm{\cS}_{\mathrm{op}}$.
	
	\begin{lemma}
		\label{lem: EOT_operator_spectral_radius}
		Let $(\Omega_1,\mu_1)$ and $(\Omega_2,\mu_2)$ be probability spaces and let $K$ be a measurable function on $\Omega_1\times\Omega_2$ with $K \ge \eta$ almost everywhere for some $\eta \in [0,1]$, $\int K(x,y)\,\mu_2(dy) = 1$ for $\mu_1$-a.e.\ $x$ and $\int K(x,y)\,\mu_1(dx) = 1$ for $\mu_2$-a.e.\ $y$. Then the integral operator $\cK : L^2(\mu_2)\to L^2(\mu_1)$ with kernel $K - 1$ satisfies $\norm{\cK}_{\mathrm{op}} \le 1 - \eta$.
	\end{lemma}
	
	\begin{proof}
		If $\eta = 1$, then $K \equiv 1$ almost everywhere and $\cK = 0$. Otherwise, $\bar{K} := (K - \eta)/(1 - \eta)$ is again a nonnegative kernel with both marginal integrals equal to one, and $K - 1 = (1 - \eta)(\bar K - 1)$. For $h \in L^2(\mu_2)$, let $h_0 := h - \int h\,d\mu_2$. Since $\int (\bar K(x,y) - 1)\,\mu_2(dy) = 0$,
		\begin{align}
			\cK h(x) = (1-\eta)\int \big(\bar K(x,y) - 1\big)h(y)\,\mu_2(dy) = (1 - \eta)\int \bar K(x,y)h_0(y)\,\mu_2(dy).
		\end{align}
		By Jensen's inequality, applied to the probability measure $\bar K(x,y)\,\mu_2(dy)$, and Fubini's theorem,
		\begin{align}
			\norm{\cK h}_{L^2(\mu_1)}^2 \le (1-\eta)^2\iint \bar K(x,y)h_0(y)^2\,\mu_2(dy)\,\mu_1(dx) = (1-\eta)^2\norm{h_0}_{L^2(\mu_2)}^2 \le (1-\eta)^2\norm{h}_{L^2(\mu_2)}^2.
		\end{align}
	\end{proof}

	\begin{proof}[\textbf{Proof of Corollary~\ref{cor: uniform_mccullagh}}]
		The kernel $(i,j) \mapsto n\mathbf{A}_{ij}$ is doubly stochastic with respect to the uniform probability measure on $[n]$ and bounded below by $e^{-K}$. The associated operator of Lemma~\ref{lem: EOT_operator_spectral_radius} is represented by the matrix $\mathbf{A} - \bJ_n$, and the operator norm with respect to the uniform probability measure on $[n]$ is the Euclidean one. Hence Lemma~\ref{lem: EOT_operator_spectral_radius} gives $\norm{\mathbf{A} - \bJ_n}_{\mathrm{op}} \le 1 - e^{-K}$. Since also $|n(\mathbf{A} - \bJ_n)_{ij}| \le e^K$, every sequence $\mathbf{A}_n \in \mathfrak{D}_n(K)$ satisfies the hypotheses of Theorem~\ref{thm: mccullogh_permanent_expansion}. If the claim failed, there would be $\eps > 0$, integers $n_1 < n_2 < \cdots$ and matrices $\mathbf{A}^{(l)} \in \mathfrak{D}_{n_l}(K)$ for which the absolute value in \eqref{eq:uniform_mccullagh} is at least $\eps$. Setting $\mathbf{A}_n := \mathbf{A}^{(l)}$ if $n = n_l$ and $\mathbf{A}_n := \bJ_n$ otherwise, we would obtain a sequence satisfying the hypotheses of Theorem~\ref{thm: mccullogh_permanent_expansion} for which its conclusion, rewritten with \eqref{eq: det_rewrite}, fails.
	\end{proof}
	
	For $\mathbf{A} = \bD_f$, the determinant in \eqref{eq:uniform_mccullagh} is a Fredholm determinant on $L^2[0,1]$. For a bounded grid-step function $f$, let $w_f := \xi_f\circ T$, which is the density of the optimal permuton for the cost $f\circ T$ by Lemma~\ref{lem: COV_permuton_to_EOT}, and let $\cS_f$ be the integral operator on $L^2[0,1]$ with kernel $w_f - 1$.
	
	\begin{lemma}
		\label{lem: partition_fredholm}
		Let $f$ be a bounded grid-step function. Then the nonzero singular values of $\cS_f$ coincide, with multiplicities, with those of $\bD_f - \bJ_N$, and are at most $1 - e^{-2\osc(f)}$. For every $K \ge 0$,
		\begin{align}
			\lim_{N\to\infty}\ \sup_{(\r,\c),\,f} \Big|\cD(f)\,\detF\big(I - \cS_f^*\cS_f\big)^{1/2} - 1\Big| = 0,
		\end{align}
		where the supremum is over all positive integer margins $(\r,\c)$ with total $N$ and all grid-step functions $f$ with $\osc(f) \le K$.
	\end{lemma}
	
	\begin{proof}
		By Lemma~\ref{lem: quantile_properties}, $w_f = N(\bD_f)_{k\ell}$ on $B_{k\ell}$. Let $V_N \subset L^2[0,1]$ be the space of functions that are constant on each interval $I^N_k := \big(\frac{k-1}{N},\frac{k}{N}\big]$, with orthonormal basis $\epsilon_k := \sqrt{N}\1_{I^N_k}$, $k \in [N]$. Since the kernel of $\cS_f$ is constant on each square $B_{k\ell} = I^N_k\times I^N_\ell$, $\cS_f$ vanishes on $V_N^\perp$ and maps $L^2[0,1]$ into $V_N$. For $h = \sum_\ell h_\ell\epsilon_\ell$ and $s \in I^N_k$,
		\begin{align}
			(\cS_f h)(s) = \sum_{\ell}\big(N(\bD_f)_{k\ell} - 1\big)\frac{1}{N}\sqrt{N}h_\ell = \sqrt{N}\big((\bD_f - \bJ_N)h\big)_k,
		\end{align}
		so $\langle \cS_f h, \epsilon_k\rangle = ((\bD_f - \bJ_N)h)_k$. Thus $\cS_f$ is represented on $V_N$ by the matrix $\bD_f - \bJ_N$, which proves the first claim; the bound on the singular values follows from $\bD_f \in \mathfrak{D}_N(2\osc(f))$ and Corollary~\ref{cor: uniform_mccullagh}. In particular, $\det\big(\mathbf{I} - (\bD_f - \bJ_N)^\top(\bD_f - \bJ_N)\big) = \detF(I - \cS_f^*\cS_f)$, and the limit follows from Lemma~\ref{lem: grid_step_identity} and Corollary~\ref{cor: uniform_mccullagh} with $2K$ in place of $K$.
	\end{proof}
	
	\begin{prop}
		\label{prop: uniform_partition_bounds}
		For every $K \ge 0$ there exist $C_K < \infty$ and $N_K \in \N$ such that for all positive integer margins $(\r,\c)$ with total $N \ge N_K$ and all grid-step functions $f$ with $\osc(f) \le K$,
		\begin{align}
			1 \le \cD(f) \le C_K.
		\end{align}
	\end{prop}
	
	\begin{proof}
		The lower bound is the van der Waerden bound $\per(\bD_f) \ge N!/N^N$ \cite{Falikman1981vanderwaerden}, by Lemma~\ref{lem: grid_step_identity}. For the upper bound, let $s_1 \ge s_2 \ge \cdots$ be the singular values of $\cS_f$. By Lemma~\ref{lem: partition_fredholm}, $s_1 \le 1 - \eta$ with $\eta := e^{-2K}$, and by Lemma~\ref{lem: density_bounds}, $\sum_i s_i^2 = \norm{\cS_f}^2_{\mathrm{HS}} = \int (\xi_f - 1)^2\,d(\rho_{\r}\otimes\rho_{\c}) \le e^{4K}$. Since $-\log(1-x) \le x/(1-x)$ and $1 - s_i^2 \ge 1 - (1-\eta)^2 \ge \eta$,
		\begin{align}
			\detF\big(I - \cS_f^*\cS_f\big) = \prod_i(1 - s_i^2) \ge \exp\Big(-\frac{1}{\eta}\sum_i s_i^2\Big) \ge \exp\big(-e^{6K}\big).
		\end{align}
		By Lemma~\ref{lem: partition_fredholm}, there is $N_K$ such that $\cD(f) \le 2\detF(I - \cS_f^*\cS_f)^{-1/2} \le 2e^{e^{6K}/2} =: C_K$ for $N \ge N_K$.
	\end{proof}
	
	\subsection{Proof of Theorem~\ref{thm: permuton_partition_det_limit}}
	\label{sec: pf_proof}
	
	It remains to pass to the limit in the Fredholm determinant. We first record a stability estimate for EOT plans in total variation. Here and below, $\norm{\cdot}_{\mathrm{TV}}$ denotes the total variation norm, $\norm{\pi - \pi'}_{\mathrm{TV}} = \int|d\pi - d\pi'|$.
	
	\begin{lemma}
		\label{lem: TV_cost_stability}
		Let $\mu,\nu$ be Borel probability measures on $[0,1]$ and $c, c' \in L^\infty(\mu\otimes\nu)$. Then
		\begin{align}
			\label{eq: TV_cost_stability}
			\norm{\pi^c_{\mu,\nu} - \pi^{c'}_{\mu,\nu}}_{\mathrm{TV}} \le \norm{c - c'}_\infty.
		\end{align}
	\end{lemma}
	
	\begin{proof}
		Write $\pi := \pi^c_{\mu,\nu}$ and $\pi' := \pi^{c'}_{\mu,\nu}$. Applying \eqref{eq: pythagorean} with the cost $c$ to $\pi'$, and with the cost $c'$ to $\pi$, and adding the two identities gives
		\begin{align}
			\KL(\pi\Vert\pi') + \KL(\pi'\Vert\pi) = \int (c' - c)\,d(\pi - \pi') \le \norm{c - c'}_\infty\norm{\pi - \pi'}_{\mathrm{TV}}.
		\end{align}
		By Pinsker's inequality, each term on the left-hand side is at least $\frac12\norm{\pi - \pi'}_{\mathrm{TV}}^2$; see also \cite[Proposition~3.12]{eckstein2022quantitative}.
	\end{proof}
	
	Now let the margins $(\r_k,\c_k)$ vary with $k$, and suppose that $\rho_{\r_k} \to \rho_{\r_\infty}$ and $\rho_{\c_k} \to \rho_{\c_\infty}$ weakly. For bounded measurable functions $f_k$ that converge uniformly to a continuous function $f_\infty$ on $[0,1]^2$, let $w_k$ and $w_\infty$ be the densities with respect to $u$ of the optimal permutons for the costs $f_k\circ T_k$ and $f_\infty\circ T_\infty$. By Lemma~\ref{lem: COV_permuton_to_EOT}, $w_k = \xi_k\circ T_k$ and $w_\infty = \xi_\infty\circ T_\infty$, where $\xi_k := d\pi_k(f_k)/d(\rho_{\r_k}\otimes\rho_{\c_k})$ and $\xi_\infty$ is as in Theorem~\ref{thm: permuton_partition_det_limit}.
	
	\begin{lemma}
		\label{lem: operator_HS_convergence}
		In this setting, $\norm{w_k - w_\infty}_{L^2(u)} \to 0$. In particular, the optimal permutons for the costs $f_k\circ T_k$ converge in total variation to the optimal permuton for the cost $f_\infty\circ T_\infty$.
	\end{lemma}
	
	\begin{proof}
		Let $(a_\infty,b_\infty)$ be Schr\"odinger potentials for the limiting problem. Since $f_\infty$ is continuous, we may define $a_\infty$ and $b_\infty$ on all of $[0,1]$ through the Schr\"odinger system \eqref{eq: schrodinger_system}, and then they are continuous. Hence $\xi_\infty = e^{a_\infty\oplus b_\infty - f_\infty}$ is bounded and continuous on $[0,1]^2$. Let $\bar\pi_k := \pi_k(f_\infty)$ be the solution of the intermediate problem with the limiting cost and the discrete margins, and let $\bar\xi_k := d\bar\pi_k/d(\rho_{\r_k}\otimes\rho_{\c_k})$. By the triangle inequality,
		\begin{align}
			\norm{w_k - w_\infty}_{L^2(u)} \le \norm{(\xi_k - \bar\xi_k)\circ T_k}_{L^2(u)} + \norm{(\bar\xi_k - \xi_\infty)\circ T_k}_{L^2(u)} + \norm{\xi_\infty\circ T_k - \xi_\infty\circ T_\infty}_{L^2(u)}.
		\end{align}
		
		For the last term, the quantile functions of $\rho_{\r_k}$ and $\rho_{\c_k}$ converge to those of $\rho_{\r_\infty}$ and $\rho_{\c_\infty}$ at every continuity point of the latter \cite[Lemma~21.2]{vandervaart1998asymptotic}, hence Lebesgue-a.e., so that $T_k \to T_\infty$ $u$-a.e. Since $\xi_\infty$ is bounded and continuous, the last term tends to zero by dominated convergence.
		
		For the second term, let $(\bar a_k, \bar b_k)$ be Schr\"odinger potentials for $\bar\pi_k$, again defined on all of $[0,1]$ through the Schr\"odinger system. Since $(T_k)_{\#}u = \rho_{\r_k}\otimes\rho_{\c_k}$, a change of variables gives
		\begin{align}
			\norm{(\bar\xi_k - \xi_\infty)\circ T_k}_{L^2(u)} \le \norm{e^{\bar a_k\oplus\bar b_k - f_\infty} - e^{a_\infty\oplus b_\infty - f_\infty}}_\infty \le e^{\norm{f_\infty}_\infty}\norm{e^{\bar a_k\oplus\bar b_k} - e^{a_\infty\oplus b_\infty}}_\infty.
		\end{align}
		Since $f_\infty$ is continuous on the compact set $[0,1]^2$ and $\rho_{\r_k}\otimes\rho_{\c_k} \to \rho_{\r_\infty}\otimes\rho_{\c_\infty}$ weakly, \cite[Proposition~B.1]{nutz2023stability} shows that $\bar a_k\oplus\bar b_k \to a_\infty\oplus b_\infty$ uniformly. These functions are uniformly bounded, so the second term tends to zero.
		
		For the first term, again by a change of variables and by Lemma~\ref{lem: TV_cost_stability},
		\begin{align}
			\norm{(\xi_k - \bar\xi_k)\circ T_k}_{L^2(u)}^2 = \int\big(\xi_k - \bar\xi_k\big)^2\,d(\rho_{\r_k}\otimes\rho_{\c_k}) \le C\norm{\pi_k(f_k) - \bar\pi_k}_{\mathrm{TV}} \le C\norm{f_k - f_\infty}_\infty \to 0,
		\end{align}
		where $C := \sup_k\big(\norm{\xi_k}_\infty + \norm{\bar\xi_k}_\infty\big)$ is finite by Lemma~\ref{lem: density_bounds}, since $\sup_k\osc(f_k) < \infty$. The last claim follows since $\norm{\cdot}_{L^1(u)} \le \norm{\cdot}_{L^2(u)}$.
	\end{proof}
	
	\begin{proof}[\textbf{Proof of Theorem~\ref{thm: permuton_partition_det_limit}}]
		Let $K := \sup_k \osc(f_k) < \infty$, and let $w_k$ and $w_\infty$ be as above.
		
		\emph{Step 1: Reduction to a Fredholm determinant.} By Lemma~\ref{lem: partition_fredholm}, $\cD_k(f_k) = (1 + o(1))\detF(I - \cS_k^*\cS_k)^{-1/2}$, where $\cS_k := \cS_{f_k}$ is the integral operator on $L^2[0,1]$ with kernel $w_k - 1$.
		
		\emph{Step 2: Convergence of the operators.} Let $\widetilde\cS$ be the integral operator on $L^2[0,1]$ with kernel $w_\infty - 1$. The Hilbert--Schmidt norm of an integral operator on $L^2[0,1]$ is the $L^2(u)$ norm of its kernel, so Lemma~\ref{lem: operator_HS_convergence} gives $\norm{\cS_k - \widetilde\cS}_{\mathrm{HS}} = \norm{w_k - w_\infty}_{L^2(u)} \to 0$. Hence
		\begin{align}
			\norm{\cS_k^*\cS_k - \widetilde\cS^*\widetilde\cS}_{\mathrm{tr}} \le \big(\norm{\cS_k}_{\mathrm{HS}} + \norm{\widetilde\cS}_{\mathrm{HS}}\big)\norm{\cS_k - \widetilde\cS}_{\mathrm{HS}} \to 0.
		\end{align}
		Since the Fredholm determinant is continuous with respect to the trace norm \cite[Theorem~3.4]{simon2005trace}, $\detF(I - \cS_k^*\cS_k) \to \detF(I - \widetilde\cS^*\widetilde\cS)$.
		
		\emph{Step 3: Identification of the limit.} Let $U_{\r} : L^2(\rho_{\r_\infty}) \to L^2[0,1]$, $h \mapsto h\circ q_{\rho_{\r_\infty}}$, and define $U_{\c}$ analogously. Since $(q_{\rho_{\r_\infty}})_{\#}\Leb = \rho_{\r_\infty}$, these maps are isometries, and $U_{\c}^*h(y) = \E[h(t) \mid q_{\rho_{\c_\infty}}(t) = y]$ for $t$ uniformly distributed on $[0,1]$. Conditioning on $q_{\rho_{\c_\infty}}(t)$ shows that
		\begin{align}
			(\widetilde\cS h)(s) = \int_0^1\big(\xi_\infty(q_{\rho_{\r_\infty}}(s), q_{\rho_{\c_\infty}}(t)) - 1\big)h(t)\,dt = \big(\cS U_{\c}^*h\big)(q_{\rho_{\r_\infty}}(s)),
		\end{align}
		that is, $\widetilde\cS = U_{\r}\cS U_{\c}^*$. Since $U_{\r}^*U_{\r} = U_{\c}^*U_{\c} = I$, the operator $\widetilde\cS^*\widetilde\cS = U_{\c}\cS^*\cS U_{\c}^*$ has the same nonzero eigenvalues as $\cS^*\cS$. Therefore $\detF(I - \widetilde\cS^*\widetilde\cS) = \detF(I - \cS^*\cS) = \prod_{i}(1 - s_i^2)$. Finally, $\norm{\cS}_{\mathrm{HS}}^2 = \int (\xi_\infty - 1)^2\,d(\rho_{\r_\infty}\otimes\rho_{\c_\infty}) < \infty$, and $\norm{\cS}_{\mathrm{op}} \le 1 - e^{-2\osc(f_\infty)}$ by Lemmas~\ref{lem: density_bounds} and \ref{lem: EOT_operator_spectral_radius}. Thus $s_i \le 1 - e^{-2\osc(f_\infty)}$ and $\sum_i s_i^2 < \infty$, so the product $\prod_i (1 - s_i^2)$ converges to a positive number. Combining the three steps proves \eqref{eq: permuton_partition_limit_statement}.
	\end{proof}
	
	\subsection{Proof of Corollary~\ref{cor: enumeration}}
	
	\begin{proof}[\textbf{Proof of Corollary~\ref{cor: enumeration}}]
		We drop the index $k$. Summing \eqref{eq: gibbs_weight_product} over $\sigma \in S_N$ and using \eqref{eq: block_weight} and Barvinok's identity \eqref{eq: barvinok_per_identity}, we obtain
		\begin{align}
			\cZ(\kappa) = \frac{(N^2/\norm{\bLambda}_1)^N}{\prod_{i}\r(i)^{\r(i)}\prod_{j}\c(j)^{\c(j)}}\,\per(\widetilde{\bLambda}) = \Big(\frac{N^2}{\norm{\bLambda}_1}\Big)^N\prod_{i}\frac{\r(i)!}{\r(i)^{\r(i)}}\prod_{j}\frac{\c(j)!}{\c(j)^{\c(j)}}\,\cZ_{\r,\c;\bLambda}.
		\end{align}
		Since $\cZ(\kappa) = N!\,e^{-N\Gamma(\kappa)}\cD(\kappa)$ by \eqref{eq: def_normalized_partition}, this reads $\cZ_{\r,\c;\bLambda} = \cD(\kappa)\,\widehat{\cZ}_{\r,\c;\bLambda}$ for all margins. Under Assumptions~\ref{assumption: bounded_cost} and \ref{assumption: margins_weak_convergence}, $\cD_k(\kappa_k) \to D_\infty$ by Theorem~\ref{thm: permuton_partition_det_limit} applied with $f_k = \kappa_k$, which proves \eqref{eq: enumeration_identity}. Since $\cD(\kappa) \ge 1$ by the van der Waerden bound (see the proof of Proposition~\ref{prop: uniform_partition_bounds}), we also have $\widehat{\cZ}_{\r,\c;\bLambda} \le \cZ_{\r,\c;\bLambda}$ for all margins.
	\end{proof}
	
	\section{Proofs of the fluctuation results}
	\label{sec: proof_fluctuations}
	
	In this section we prove the central limit theorem (Theorem~\ref{thm: CLT}), its consequence for Mallows-type permutations (Corollary~\ref{cor: mallows}), the properties of the limiting process, and the functional limit theorems (Theorem~\ref{thm: functional_limits} and Corollaries~\ref{cor: KS} and \ref{cor: product_case}).
	
	The proof of Theorem~\ref{thm: CLT} rests on three ingredients. The exact Laplace identity of Section~\ref{sec: laplace_section} expresses the moment generating function of $\Phi_k(g)$ through partition functions and the EOT value with a perturbed cost. The ratio of partition functions is controlled by Theorem~\ref{thm: permuton_partition_det_limit} and the uniform bounds of Proposition~\ref{prop: uniform_partition_bounds}, and the EOT term by Proposition~\ref{thm: smooth_cost_perturbations}.

	Throughout this section, except in Section~\ref{sec: cloned}, we assume Assumptions~\ref{assumption: bounded_cost} and \ref{assumption: margins_weak_convergence}; Lemmas~\ref{lem: laplace_identity}, \ref{lem: quadratic_lower_bound} and \ref{lem: variance_continuity} and Proposition~\ref{thm: smooth_cost_perturbations} do not use them. Proposition~\ref{prop: exp_moments} and Corollary~\ref{cor: moments} 
	require only that the constant
	\begin{align}
		\label{eq: def_K1}
		K_1 := \sup_k\osc(\kappa_k) + 2,
	\end{align}
	is finite, which follows by Assumption~\ref{assumption: bounded_cost} (Section~\ref{sec: assumptions}). Since $\kappa_k \to \kappa_\infty$ uniformly and $\kappa_\infty$ is continuous, also $\osc(\kappa_\infty) = \sup\kappa_\infty - \inf\kappa_\infty \le K_1$. We write $|z|_1 := |s| + |t|$ for $z = (s,t) \in \R^2$.
	
	\subsection{An exact formula for the Laplace transform}
	\label{sec: laplace_section}
	
	The identity below involves the cell averages $\Ek_k g$ of Definition~\ref{def: cell_average}. Note that $\Ek$ is the conditional expectation, with respect to $u$, onto the $\sigma$-algebra generated by the cells $A_{ij}$. In particular, $\Ek$ is the $L^2(u)$-orthogonal projection onto grid-step functions, $\norm{\Ek g}_{L^p} \le \norm{g}_{L^p}$ for all $p \in [1,\infty]$, and $\Ek g \to g$ in $L^p$ for every $p \in [1,\infty)$ as $m, n \to \infty$. For $g \in L^\infty([0,1]^2)$ and $t \in \R$, we set $f := -tN_k^{-1/2}\Ek_k g$.
	
	\begin{lemma}[Exact Laplace identity]
		\label{lem: laplace_identity}
		For every $k$, $g \in L^\infty([0,1]^2)$ and $t \in \R$,
		\begin{align}
			\label{eq: laplace_identity}
			\log\E\big[e^{t\Phi_k(g)}\big] = \log\frac{\cD_k(\kappa_k + f)}{\cD_k(\kappa_k)} - N_k\Big[\Gamma_k(\kappa_k + f) - \Gamma_k(\kappa_k) - \int f\,d\pi_{\mathbf{Z}_k}\Big].
		\end{align}
	\end{lemma}
	
	\begin{proof}
		We drop the index $k$. By \eqref{eq:coupling_Gibbs}, $\mathbf{X} \overset{d}{=} \mathbf{X}(\sigma)$ with $\sigma \sim \Gibbs$, and $t\sqrt{N}\int g\,d\pi_{\mathbf{X}(\sigma)} = t\sqrt{N}\int(\Ek g)\circ T\,d\mu_\sigma = -N\int f\circ T\,d\mu_\sigma$ by Proposition~\ref{prop: permutation_CT_coupling}(ii). Hence
		\begin{align}
			\E\Big[\exp\Big(t\sqrt{N}\int g\,d\pi_{\mathbf{X}}\Big)\Big] = \frac{1}{\cZ(\kappa)}\sum_{\sigma\in S_N}e^{-N\int (\kappa + f)\circ T\,d\mu_\sigma} = \frac{\cZ(\kappa + f)}{\cZ(\kappa)}.
		\end{align}
		By \eqref{eq: cell_average_invariance}, $t\sqrt{N}\int g\,d\pi_{\mathbf{Z}} = t\sqrt N\int \Ek g\,d\pi_{\mathbf{Z}} = -N\int f\,d\pi_{\mathbf{Z}}$. Therefore $\log \E[e^{t\Phi(g)}] = \log\cZ(\kappa+f) - \log\cZ(\kappa) + N\int f\,d\pi_{\mathbf{Z}}$, and \eqref{eq: laplace_identity} follows by substituting $\cZ(h) = N!\,e^{-N\Gamma(h)}\cD(h)$ for $h = \kappa$ and $h = \kappa + f$.
	\end{proof}
	
	\begin{lemma}[Quadratic lower bound]
		\label{lem: quadratic_lower_bound}
		For probability measures $\mu,\nu$ on $[0,1]$ and $c, f \in L^\infty(\mu\otimes\nu)$,
		\begin{align}
			\Gamma_{\mu,\nu}(c + f) - \Gamma_{\mu,\nu}(c) - \int f\,d\pi^{c}_{\mu,\nu} \ge -\frac{e^{\norm{f}_\infty}}{2}\int f^2\,d\pi^c_{\mu,\nu}.
		\end{align}
	\end{lemma}
	
	\begin{proof}
		Since $\int (c+f)\,d\pi + \KL(\pi\Vert\mu\otimes\nu) = \KL(\pi \Vert e^{-c-f}(\mu\otimes\nu))$ for $\pi \in \Pi(\mu,\nu)$, the duality \eqref{eq: SSB_duality} with $h := c + f$ gives
		\begin{align}
			\Gamma_{\mu,\nu}(c+f) \ge \int a\,d\mu + \int b\,d\nu - \int e^{a\oplus b - c - f}\,d(\mu\otimes\nu) + 1
		\end{align}
		for all bounded $a$ and $b$.
		Choosing for $(a,b)$ the Schr\"odinger potentials of $c$, for which $e^{a\oplus b - c}\,d(\mu\otimes\nu) = d\pi^c_{\mu,\nu}$ and $\int a\,d\mu + \int b\,d\nu = \Gamma_{\mu,\nu}(c)$, gives $\Gamma_{\mu,\nu}(c+f) \ge \Gamma_{\mu,\nu}(c) - \int(e^{-f} - 1)\,d\pi^c_{\mu,\nu}$. The claim follows from $e^{-y} - 1 + y \le \frac{y^2}{2}e^{|y|}$.
	\end{proof}
	
	\subsection{Second-order expansion of the EOT value}
	\label{sec: EOT_expansion}
	
	By Lemma~\ref{lem: laplace_identity}, the logarithmic moment generating function of $\Phi_k(g)$ contains the term $N_k\big[\Gamma_k(\kappa_k + f) - \Gamma_k(\kappa_k) - \int f\,d\pi_{\mathbf{Z}_k}\big]$ with a perturbation $f$ of order $N_k^{-1/2}$, and the Gaussian limit arises from its second-order Taylor expansion. The following proposition provides this expansion: it identifies the first two derivatives of the EOT value along a line of costs, the second being minus the squared distance to the additively separable functions, which is the quadratic form of the limiting variance $V(g)$. The same expansion yields the boundary factor in Corollary~\ref{cor: mallows}(i).
	
	In this subsection and in Section~\ref{sec: smooth_cost}, $\mu$ and $\nu$ are arbitrary Borel probability measures on $[0,1]$, all function spaces refer to $\mu$, $\nu$ or $\mu\otimes\nu$, and we abbreviate $\Gamma(c) := \Gamma_{\mu,\nu}(c)$ and $\pi^c := \pi^c_{\mu,\nu}$. For $f, g \in L^\infty(\mu\otimes\nu)$ and $t \in \R$, we write $\pi_t := \pi^{f+tg}$.
	
	\begin{prop}
		\label{thm: smooth_cost_perturbations}
		Let $f, g \in L^{\infty}(\mu \otimes \nu)$. Then $t \mapsto \Gamma(f + tg)$ is $C^{\infty}$ on $\R$, and for every $t \in \R$,
		\begin{align}
			\label{eq: EOT_derivatives}
			\frac{d}{dt}\Gamma(f+tg) = \int g \, d\pi_t, \qquad \frac{d^2}{dt^2}\Gamma(f+tg) = -\norm{(I - P_{\pi_t})g}_{L^2(\pi_t)}^2.
		\end{align}
	\end{prop}
	
	Here $P_{\pi_t}$ is the projection of Definition~\ref{def: sum_space}, which projects onto $\cA$ by Lemmas~\ref{lem: density_bounds} and \ref{lem: sum_space}. The proof is given in Section~\ref{sec: smooth_cost}.
	
	\subsection{Continuity of the variance}
	
	\begin{lemma}[Continuity of the variance functional]
		\label{lem: variance_continuity}
		Let $\pi_k, \pi$ be probability measures on $[0,1]^2$ with densities $W_k, W$ with respect to $u$ such that $c \le W_k, W \le C$ for constants $0 < c \le C$ and $W_k \to W$ in $L^1(u)$. Let $h_k \to h$ in $L^2(u)$, with $h$ bounded. Then
		\begin{align}
			\norm{(I - P_{\pi_k})h_k}^2_{L^2(\pi_k)} \to \norm{(I-P_\pi)h}^2_{L^2(\pi)}.
		\end{align}
	\end{lemma}
	
	\begin{proof}
		The marginal densities of $\pi_k$ and $\pi$ take values in $[c,C]$. Hence the space $\cA$ of additively separable functions is the same closed subspace of $L^2(u)$ for all these measures, and $\norm{(I - P_\pi)h}^2_{L^2(\pi)} = V_\pi(h) := \inf_{w\in\cA}\int (h - w)^2W\,du$ by Lemma~\ref{lem: sum_space}; similarly for $\pi_k$.
		
		\emph{Upper bound.} Given $\eps > 0$, since $L^{\infty}(u) \cap \mathcal{A}$ is dense in $\cA$, we may choose $w = a\oplus b$ with $a, b$ bounded and $\int (h-w)^2W\,du \le V_\pi(h) + \eps$. Then $V_{\pi_k}(h_k) \le \int (h_k - w)^2W_k\,du$ and
		\begin{align}
			\Big|\int (h_k - w)^2W_k\,du - \int (h-w)^2W\,du\Big| &\le C\norm{h_k - h}_{L^2}\norm{h_k + h - 2w}_{L^2} \\
			&\quad + \Big|\int (h-w)^2(W_k - W)\,du\Big| \to 0.
		\end{align}
		The last term is at most $\norm{h-w}_\infty^2\norm{W_k - W}_{L^1(u)}$, since $h$ and $w$ are bounded. Hence $\limsup_k V_{\pi_k}(h_k) \le V_\pi(h) + \eps$ for every $\eps > 0$.
		
		
		
		\emph{Lower bound.} Pass to a subsequence along which $V_{\pi_k}(h_k)$ converges to its $\liminf$  and for which $W_k \to W$ $u$-a.e., and choose $w_k \in \cA$ with $\int(h_k - w_k)^2W_k\,du \le V_{\pi_k}(h_k) + 1/k$. The sequence $v_k := h_k - w_k$ is bounded in $L^2(u)$, so along a further subsequence $v_k$ converges weakly to some $\tilde v$ in $L^2(u)$. Then $w_k = h_k - v_k \rightarrow h - \tilde v$ weakly, and $h - \tilde v \in \cA$, since closed subspaces are weakly closed. Now for $\varphi \in L^2(u)$,
		\begin{align}
			\int\varphi\sqrt{W_k}\,v_k\,du - \int\varphi\sqrt{W}\,\tilde v\,du = \int\varphi\sqrt W\,(v_k - \tilde v)\,du + \int \varphi\big(\sqrt{W_k} - \sqrt{W}\big)v_k\,du \to 0,
		\end{align}
		because $\norm{\varphi(\sqrt{W_k} - \sqrt W)}_{L^2} \to 0$ by dominated convergence. 
		Thus $\sqrt{W_k}\,v_k \rightarrow \sqrt{W}\,\tilde v$ weakly, and by weak lower semicontinuity of the norm, $\liminf_k V_{\pi_k}(h_k) \ge \int \tilde v^2W\,du \ge V_\pi(h)$.
	\end{proof}

	\subsection{Uniform exponential moments}
	The last two claims of Theorem~\ref{thm: CLT} and the functional limit theorems require bounds on $\Phi_k(g)$ that are uniform in $k$ and $g$. These follow from the Laplace identity \eqref{eq: laplace_identity} and the uniform partition function bounds of Proposition~\ref{prop: uniform_partition_bounds}. Let $C_U := C_{K_1}$ and $N_U := N_{K_1}$ be the constants of Proposition~\ref{prop: uniform_partition_bounds} for $K = K_1$.
	
	\begin{prop}
		\label{prop: exp_moments}
		For all $k$ with $N_k \ge N_U$, all $g \in L^\infty([0,1]^2)$, and all $t\in\R$ with $|t|\norm{\Ek_kg}_\infty \le \sqrt{N_k}$,
		\begin{align}
			\E\big[e^{t\Phi_k(g)}\big] \le C_U\exp\Big(\frac{e}{2}\,t^2\norm{\Ek_kg}_{L^2(\pi_{\mathbf{Z}_k})}^2\Big).
		\end{align}
	\end{prop}
	
	\begin{proof}
		In \eqref{eq: laplace_identity} we have $\norm{f}_\infty \le 1$, so $\osc(\kappa_k + f) \le K_1$, and Proposition~\ref{prop: uniform_partition_bounds} gives $\cD_k(\kappa_k + f) \le C_U$ and $\cD_k(\kappa_k) \ge 1$. By Lemma~\ref{lem: quadratic_lower_bound} with $c = \kappa_k$ and Lemma~\ref{lem: typical_table_bridge}, the second term in \eqref{eq: laplace_identity} is at most $\frac{N_k}{2}e\int f^2\,d\pi_{\mathbf{Z}_k} = \frac{e}{2}\,t^2\norm{\Ek_kg}^2_{L^2(\pi_{\mathbf{Z}_k})}$.
	\end{proof}
	
	\begin{corollary}
		\label{cor: moments}
		For every integer $p \ge 1$ there is a constant $C_p$, depending only on $p$ and $K_1$, such that for all $k$ with $N_k \ge N_U$ and all $g \in L^\infty([0,1]^2)$,
		\begin{align}
			\E|\Phi_k(g)|^p \le C_p\Big(\norm{\Ek_kg}_{L^2(\pi_{\mathbf{Z}_k})}^p + N_k^{-p/2}\norm{\Ek_kg}_\infty^p\Big) \le C_p\Big(\norm{g}_{L^2(\pi_{\mathbf{Z}_k})}^p + N_k^{-p/2}\norm{g}_\infty^p\Big).
		\end{align}
	\end{corollary}
	
	\begin{proof}
		Let $\varsigma := \norm{\Ek_kg}_{L^2(\pi_{\mathbf{Z}_k})}$. If $\varsigma = 0$, then $\Ek_k g = 0$, since $\varphi_{\mathbf{Z}_k}$ is positive on every cell by \eqref{eq: typical_table_tilting}, and $\Phi_k(g) = \Phi_k(\Ek_k g) = 0$. Otherwise, let
		\begin{align}
			t := \min\big\{\varsigma^{-1},\ \sqrt{N_k}/\norm{\Ek_kg}_\infty\big\}.
		\end{align}
		Using $|y|^p \le p!\,t^{-p}(e^{ty} + e^{-ty})$ and Proposition~\ref{prop: exp_moments} with $\pm t$, we get $\E|\Phi_k(g)|^p \le 2p!\,C_Ue^{e/2}\,t^{-p}$, and $t^{-p} \le \varsigma^p + N_k^{-p/2}\norm{\Ek_kg}_\infty^p$. For the second inequality, note that $\norm{\Ek_kg}_\infty \le \norm{g}_\infty$ and, by Jensen's inequality and \eqref{eq: cell_average_invariance}, $\norm{\Ek_kg}^2_{L^2(\pi_{\mathbf{Z}_k})} \le \int\Ek_k(g^2)\,d\pi_{\mathbf{Z}_k} = \int g^2\,d\pi_{\mathbf{Z}_k}$.
	\end{proof}
	
	\subsection{Warm-up: cloned margins}
	\label{sec: cloned}
	
	Before proving Theorem~\ref{thm: CLT}, we treat cloned margins \cite[Example~1.9]{lyu2024large}, for which Assumption~\ref{assumption: bounded_cost} fails in general. Fix positive integer margins $(\r_0,\c_0) \in \N^{m_0}\times\N^{n_0}$ with total $N_0$ and a prior mean $\bLambda_0 \in \R^{m_0\times n_0}_{>0}$, with Schr\"odinger bridge $\mathbf{Z}_0$, cost function $\kappa_0$ and $T_0 := (q_{\r_0},q_{\c_0})$. For $k \in \N$, the $k$-cloned problem has the margins $\r_k := k\r_0\otimes\mathbf{1}_k$ and $\c_k := k\c_0\otimes\mathbf{1}_k$ and the prior mean $\bLambda_k := \bLambda_0\otimes\mathbf{1}_k\mathbf{1}_k^\top$, where $\mathbf{1}_k \in \R^k$ is the all-ones vector and $\otimes$ is the Kronecker product. Then $m_k = km_0$, $n_k = kn_0$ and $N_k = k^2N_0$, and Definition~\ref{def: histogram} and \eqref{eq: def_kappa} give $\rho_{\r_k} = \rho_{\r_0}$, $\rho_{\c_k} = \rho_{\c_0}$, $T_k = T_0$ and $\kappa_k = \kappa_0$ for all $k$. Consequently, $\pi_{\mathbf{Z}_k} = \pi^{\kappa_0}_{\rho_{\r_0},\rho_{\c_0}} = \pi_{\mathbf{Z}_0}$ for all $k$ by Lemma~\ref{lem: typical_table_bridge} (in fact, $\mathbf{Z}_k = \mathbf{Z}_0\otimes\mathbf{1}_k\mathbf{1}_k^\top$). Let $V$ and $\Xi$ be as in Section~\ref{sec: main_fluctuations} with $\pi_\infty := \pi_{\mathbf{Z}_0}$.
	
	\begin{prop}[Central limit theorem for cloned margins]
		\label{prop: cloned_CLT}
		For all $g_1, \ldots, g_p \in L^\infty([0,1]^2)$, $(\Phi_k(g_1), \ldots, \Phi_k(g_p)) \Rightarrow (\Xi(g_1), \ldots, \Xi(g_p))$.
	\end{prop}
	
	\begin{proof}
		Let $g \in L^\infty([0,1]^2)$ and $t \in \R$, and let $s_k := -tN_k^{-1/2}$ and $h_k := \Ek_kg$, so that $f = s_kh_k$ in Lemma~\ref{lem: laplace_identity}. For bounded costs $c$, we write $\pi(c) := \pi^c_{\rho_{\r_0},\rho_{\c_0}}$ and $\xi_c := \frac{d\pi(c)}{d(\rho_{\r_0}\otimes\rho_{\c_0})}$. We consider only $k$ so large that $|s_k|\norm{g}_\infty \le 1$, so that $\osc(\kappa_0 + \theta h_k) \le K_1 = \osc(\kappa_0) + 2$ for $|\theta| \le |s_k|$.
		
		\emph{The partition function term.} By Lemma~\ref{lem: partition_fredholm}, $\cD_k(c) = (1 + o(1))\detF(I - \cS_c^*\cS_c)^{-1/2}$ for $c \in \{\kappa_0, \kappa_0 + s_kh_k\}$, where $\cS_c$ is the integral operator on $L^2[0,1]$ with kernel $\xi_c\circ T_0 - 1$. By Lemma~\ref{lem: TV_cost_stability}, $\norm{\pi(\kappa_0 + s_kh_k) - \pi(\kappa_0)}_{\mathrm{TV}} \le |s_k|\norm{g}_\infty$. Since $(T_0)_{\#}u = \rho_{\r_0}\otimes\rho_{\c_0}$ and the densities $\xi_c$ take values in $[e^{-2K_1},e^{2K_1}]$ by Lemma~\ref{lem: density_bounds},
		\begin{align}
			\norm{\cS_{\kappa_0 + s_kh_k} - \cS_{\kappa_0}}_{\mathrm{HS}}^2 = \int\big(\xi_{\kappa_0 + s_kh_k} - \xi_{\kappa_0}\big)^2\,d(\rho_{\r_0}\otimes\rho_{\c_0}) \le e^{2K_1}|s_k|\norm{g}_\infty \to 0.
		\end{align}
		Hence $\cS^*_{\kappa_0 + s_kh_k}\cS_{\kappa_0 + s_kh_k} \to \cS^*_{\kappa_0}\cS_{\kappa_0}$ in the trace norm, as in Step~2 of the proof of Theorem~\ref{thm: permuton_partition_det_limit}. Since the Fredholm determinant is continuous in the trace norm, and the nonzero singular values of $\cS_{\kappa_0}$ are at most $1 - e^{-2K_1}$ by Lemma~\ref{lem: partition_fredholm}, $\cD_k(\kappa_0 + s_kh_k)$ and $\cD_k(\kappa_0)$ both converge to $\detF(I - \cS^*_{\kappa_0}\cS_{\kappa_0})^{-1/2} \in [1,\infty)$, and the first term in \eqref{eq: laplace_identity} tends to zero.
		
		\emph{The entropic term.} By Proposition~\ref{thm: smooth_cost_perturbations}, applied with $\mu = \rho_{\r_0}$, $\nu = \rho_{\c_0}$, the cost $\kappa_0$ and the direction $h_k$, Taylor's theorem gives $\theta_k$ between $0$ and $s_k$ such that, with $\pi'_k := \pi(\kappa_0 + \theta_kh_k)$,
		\begin{align}
			\label{eq: taylor_entropic}
			-N_k\Big[\Gamma_k(\kappa_0 + s_kh_k) - \Gamma_k(\kappa_0) - s_k\int h_k\,d\pi_{\mathbf{Z}_0}\Big] = \frac{t^2}{2}\norm{(I - P_{\pi'_k})h_k}^2_{L^2(\pi'_k)}.
		\end{align}
		The densities of $\rho_{\r_0}$ and $\rho_{\c_0}$ are bounded above and below by positive constants, since the margins are positive, and hence, by Lemma~\ref{lem: density_bounds}, so are the densities of $\pi'_k$ and $\pi_\infty = \pi(\kappa_0)$ with respect to $u$, uniformly in $k$. Moreover, $\norm{\pi'_k - \pi_\infty}_{\mathrm{TV}} \le |s_k|\norm{g}_\infty \to 0$ by Lemma~\ref{lem: TV_cost_stability}, and $h_k \to g$ in $L^2(u)$, because the mesh of the grid tends to zero (Section~\ref{sec: laplace_section}). By Lemma~\ref{lem: variance_continuity}, the right-hand side of \eqref{eq: taylor_entropic} converges to $\frac{t^2}{2}V(g)$.
		
		By \eqref{eq: laplace_identity}, it follows that $\E[e^{t\Phi_k(g)}] \to e^{t^2V(g)/2}$ for every $t \in \R$. Convergence of moment generating functions on $\R$ implies convergence in distribution \cite{curtiss1942note}, so $\Phi_k(g) \Rightarrow \mathcal{N}(0,V(g))$. For the joint convergence, note that $\sum_{i}t_i\Phi_k(g_i) = \Phi_k(\sum_i t_ig_i)$ converges in distribution to $\mathcal{N}(0, V(\sum_it_ig_i))$, which is the law of $\sum_i t_i\Xi(g_i)$. The Cram\'er--Wold device completes the proof.
	\end{proof}
	
	For $m_0 = n_0$, $\r_0 = \c_0 = \mathbf{1}_{m_0}$ and $\bLambda_0 = \mathbf{B}$, the matrix $\widetilde{\bLambda}_k$ of the $k$-cloned problem is the block-uniform matrix $\mathbf{B}\otimes\mathbf{1}_{k^2}\mathbf{1}_{k^2}^\top$ of \cite{ottolini2026permanent}, and Proposition~\ref{prop: cloned_CLT} describes the joint Gaussian fluctuations of linear statistics of the corresponding Gibbs random permutations.
	
	\subsection{Proof of Theorem~\ref{thm: CLT}}
	\label{sec: proof_CLT}
	
	Compared with cloned margins, the general case requires two additional ingredients, since the margins now depend on $k$ and converge only weakly. First, we pass to permuton coordinates, in which the densities of the optimal permutons are bounded above and below for arbitrary margins (Lemma~\ref{lem: permuton_coordinates}). Second, these densities converge in $L^2(u)$ under weak convergence of the margins and uniform convergence of the costs to a continuous limit (Lemma~\ref{lem: operator_HS_convergence}). The test functions must then be transported to permuton coordinates, which is done in the following lemma.
	
	\begin{lemma}[Transported cell averages]
		\label{lem: transported_cell_averages}
		Assume \ref{assumption: margins_weak_convergence}.
		\begin{enumerate}[label=(\alph*)]
			\item If $g$ is a bounded Borel function on $[0,1]^2$ that is continuous at $(\rho_{\r_\infty}\otimes\rho_{\c_\infty})$-almost every point, then $(\Ek_kg)\circ T_k \to g\circ T_\infty$ in $L^2(u)$.
			\item Under hypothesis~(ii) of Theorem~\ref{thm: CLT}, $\rho_{\r_\infty}$ and $\rho_{\c_\infty}$ are absolutely continuous, and $(\Ek_kg)\circ T_k \to g\circ T_\infty$ in $L^2(u)$ for every $g \in L^\infty([0,1]^2)$.
		\end{enumerate}
	\end{lemma}
	
	\begin{proof}
		(a) The set $D_g$ of discontinuity points of $g$ is a Borel set with $(\rho_{\r_\infty}\otimes\rho_{\c_\infty})(D_g) = 0$, so $T_\infty^{-1}(D_g)$ is $u$-null, and $T_k \to T_\infty$ $u$-a.e., as in the proof of Lemma~\ref{lem: operator_HS_convergence}. Let $(s,t) \in (0,1]^2$ be such that $T_k(s,t) \to T_\infty(s,t) \notin D_g$. Since $F_{\r_k}(0) = F_{\c_k}(0) = 0$, the point $T_k(s,t)$ lies in $(0,1]^2$, so $(\Ek_kg)(T_k(s,t))$ is the average of $g$ over the cell of the $m_k\times n_k$ grid that contains $T_k(s,t)$. These cells shrink to the point $T_\infty(s,t)$, at which $g$ is continuous, so $(\Ek_kg)(T_k(s,t)) \to g(T_\infty(s,t))$. Hence $(\Ek_kg)\circ T_k \to g\circ T_\infty$ $u$-a.e., and the claim follows by dominated convergence.
		
		(b) By the Dunford--Pettis theorem, a subsequence of $(\rho_{\r_k})$ converges weakly in $L^1[0,1]$; its limit is a density of $\rho_{\r_\infty}$, so $\rho_{\r_\infty}$, and similarly $\rho_{\c_\infty}$, is absolutely continuous. For $k \in \N\cup\{\infty\}$, let $\varrho_k := \rho_{\r_k}\otimes\rho_{\c_k}$ be the density of $(T_k)_{\#}u$. Since $\{\varrho_k > M^2\} \subset (\{\rho_{\r_k} > M\}\times[0,1])\cup([0,1]\times\{\rho_{\c_k} > M\})$, and a single integrable function is uniformly integrable, $\eps_M := \sup_{k \in \N\cup\{\infty\}}\int_{\{\varrho_k > M^2\}}\varrho_k\,du \to 0$ as $M \to \infty$. Let $g \in L^\infty([0,1]^2)$, represented by a Borel function with $\sup|g| = \norm{g}_\infty$, let $\eta > 0$, and choose $\tilde g \in C([0,1]^2)$ with $\sup|\tilde g| \le \norm{g}_\infty$ and $\norm{g - \tilde g}_{L^2(u)} \le \eta$. For finite $k$, $\varrho_k$ is constant on the cells, so Jensen's inequality gives, with $\varphi := g - \tilde g$,
		\begin{align}
			\norm{(\Ek_k\varphi)\circ T_k}^2_{L^2(u)} = \int(\Ek_k\varphi)^2\varrho_k\,du \le \int \Ek_k(\varphi^2)\,\varrho_k\,du = \int\varphi^2\varrho_k\,du \le M^2\eta^2 + 4\norm{g}_\infty^2\eps_M,
		\end{align}
		and similarly $\norm{\varphi\circ T_\infty}^2_{L^2(u)} \le M^2\eta^2 + 4\norm{g}_\infty^2\eps_M$. Since $(\Ek_k\tilde g)\circ T_k \to \tilde g\circ T_\infty$ in $L^2(u)$ by (a), the $\limsup$ of $\norm{(\Ek_kg)\circ T_k - g\circ T_\infty}_{L^2(u)}$ is at most $2(M^2\eta^2 + 4\norm{g}_\infty^2\eps_M)^{1/2}$. Letting first $\eta \to 0$ and then $M \to \infty$ proves the claim.
	\end{proof}
	
	\begin{proof}[\textbf{Proof of Theorem~\ref{thm: CLT}}]
		By Lemma~\ref{lem: transported_cell_averages}(b), the limiting margins are absolutely continuous in case~(ii), so that every $g \in L^\infty([0,1]^2)$ defines an element of $L^2(\pi_\infty)$. In both cases, the test functions form a vector space; in case~(i), because the discontinuity set of a linear combination is contained in the union of those of its terms. Fix such a test function $g$ and $t \in \R$, let $s_k := -tN_k^{-1/2}$ and $h_k := \Ek_kg$, and take $k$ so large that $|s_k|\norm{g}_\infty \le 1$. We follow the proof of Proposition~\ref{prop: cloned_CLT}.
		
		The functions $\kappa_k$ and $\kappa_k + s_kh_k$ are grid-step and converge uniformly to the continuous function $\kappa_\infty$. By Theorem~\ref{thm: permuton_partition_det_limit}, $\cD_k(\kappa_k + s_kh_k)$ and $\cD_k(\kappa_k)$ both converge to $D_\infty$, and the first term in \eqref{eq: laplace_identity} tends to zero. As in \eqref{eq: taylor_entropic}, with $\kappa_k$ and $\pi_k(\cdot)$ in place of $\kappa_0$ and $\pi(\cdot)$, the second term is equal to $\frac{t^2}{2}\norm{(I - P_{\pi_{k,\theta_k}})h_k}^2_{L^2(\pi_{k,\theta_k})}$ for some $\theta_k$ between $0$ and $s_k$, where $\pi_{k,\theta_k} := \pi_k(\kappa_k + \theta_kh_k)$. Let $\mu_k$ and $\mu_\infty$ be the optimal permutons for the costs $(\kappa_k + \theta_kh_k)\circ T_k$ and $\kappa_\infty\circ T_\infty$. Since $\osc(\kappa_k + \theta_kh_k) \le K_1$ and $\osc(\kappa_\infty) \le K_1$, Lemma~\ref{lem: permuton_coordinates} shows that the densities of $\mu_k$ and $\mu_\infty$ with respect to $u$ take values in $[e^{-2K_1},e^{2K_1}]$ and that
		\begin{align}
			\norm{(I-P_{\pi_{k,\theta_k}})h_k}^2_{L^2(\pi_{k,\theta_k})} = \norm{(I-P_{\mu_k})(h_k\circ T_k)}^2_{L^2(\mu_k)}, \qquad V(g) = \norm{(I-P_{\mu_\infty})(g\circ T_\infty)}^2_{L^2(\mu_\infty)}.
		\end{align}
		Since $\kappa_k + \theta_kh_k \to \kappa_\infty$ uniformly, Lemma~\ref{lem: operator_HS_convergence} shows that the density of $\mu_k$ converges to that of $\mu_\infty$ in $L^2(u)$, hence in $L^1(u)$, and $h_k\circ T_k \to g\circ T_\infty$ in $L^2(u)$ by Lemma~\ref{lem: transported_cell_averages}. By Lemma~\ref{lem: variance_continuity}, applied to $\mu_k$ and $\mu_\infty$, the second term in \eqref{eq: laplace_identity} converges to $\frac{t^2}{2}V(g)$. We conclude as for Proposition~\ref{prop: cloned_CLT}.
		
		It remains to prove the last two claims. By Corollary~\ref{cor: moments}, 
		\begin{align}
			\E|\Phi_k(\tilde g_k)|^p \leq C_{p}(\norm{\tilde g_k}_{L^2(\pi_k)}^{p} + o(1))
		\end{align}
		for every integer $p \ge 1$ whenever $\sup_k\norm{\tilde g_k}_\infty < \infty$. If $\norm{g_{i,k} - g_i}_{\infty} \to 0$, then,  applying this with $p = 2$ shows $\E\,\Phi_k(g_{i,k} - g_i)^2 \to 0$, so that $\Phi_k(g_{i,k}) - \Phi_k(g_i) \to 0$ in $L^2$ and the convergence in distribution persists. 
		Moreover, every linear combination $\sum_i t_i\Phi_k(g_{i,k})$ has moments of all orders bounded uniformly in $k$, hence is uniformly integrable together with all its powers, and its moments converge to those of the Gaussian limit. Joint convergence follows as in the proof of Proposition~\ref{prop: cloned_CLT}.
	\end{proof}

	\subsection{Proof of Corollary~\ref{cor: mallows}}
	\label{sec: proof_mallows}

	\begin{proof}[\textbf{Proof of Corollary~\ref{cor: mallows}}]
		Throughout the proof, all margins are equal to one, that is, $m = n = N$ and $\r = \c = \mathbf{1}$. Then $\rho_{\r} = \rho_{\c} = \Leb$, $T$ is the identity, the cells $A_{ij}$ are the squares $B_{ij}$, and $\Gamma_{\rho_{\r},\rho_{\c}} = \Gamma^{\mathrm{u}}$.
		
		\emph{The assumptions and \eqref{eq: mallows_partition}.} For $\bLambda_N$ as in Corollary~\ref{cor: mallows}, $\lambda_{\mathbf{1},\mathbf{1};\bLambda_N}$ is the law of the permutation matrix $\mathbf{X}_N$ of $\sigma_N \sim \Gibbs^f_N$, and $\pi_{\mathbf{X}_N} = \mu_{\sigma_N}$. By \eqref{eq: def_kappa}, the cost function of $(\mathbf{1},\mathbf{1},\bLambda_N)$ is $\kappa_N = f_N + \log(\norm{\bLambda_N}_1/N^2)$, which converges uniformly to $f + \log\int e^{-f}\,du$. Hence Assumptions~\ref{assumption: bounded_cost} and \ref{assumption: margins_weak_convergence} hold with $\gamma = 1$, $\rho_{\r_\infty} = \rho_{\c_\infty} = \Leb$ and $\pi_\infty = \pi^f_{\Leb,\Leb}$, since adding a constant to the cost does not change the Schr\"odinger bridge. By Lemma~\ref{lem: typical_table_bridge}, $\pi_{\mathbf{Z}_N} = \pi^{\kappa_N}_{\Leb,\Leb} = \pi_N := \pi^{f_N}_{\Leb,\Leb}$. Moreover, $f_N$ is grid-step, converges uniformly to $f$, and is equal to $f(k/N,\ell/N)$ on $B_{k\ell}$. Hence $N\int f_N\,d\mu_\sigma = \sum_{i}f(i/N,\sigma(i)/N)$ and $\cZ(f_N) = N!\,L_N(f)$, so that $\cD(f_N) = e^{N\Gamma^{\mathrm{u}}(f_N)}L_N(f)$. Theorem~\ref{thm: permuton_partition_det_limit}, applied with $f_k := f_N$ and $f_\infty := f$, gives \eqref{eq: mallows_partition}, since then $\xi_\infty = W$.
		
		(i) Let $f \in C^1([0,1]^2)$ and $h := f - f_N$. By Proposition~\ref{thm: smooth_cost_perturbations} with $\mu = \nu = \Leb$, the second derivative of $t\mapsto\Gamma^{\mathrm{u}}(f_N + th)$ lies in $[-\norm{h}_\infty^2, 0]$, since $\norm{(I-P_\pi)h}_{L^2(\pi)} \le \norm{h}_\infty$ for every probability measure $\pi$. By Taylor's formula,
		\begin{align}
			\Big|\Gamma^{\mathrm{u}}(f) - \Gamma^{\mathrm{u}}(f_N) - \int h\,d\pi_N\Big| \le \frac12\norm{h}_\infty^2 = O(N^{-2}),
		\end{align}
		since $\norm{h}_\infty \le (\norm{\partial_x f}_\infty + \norm{\partial_y f}_\infty)/N$. As observed at the beginning of Section~\ref{sec: partition_function}, the density $w_N$ of $\pi_N$ is constant on each square $B_{k\ell}$. For $(x,y) \in B_{k\ell}$, the uniform continuity of $\nabla f$ gives
		\begin{align}
			h(x,y) = \partial_x f\big(\tfrac{k}{N},\tfrac{\ell}{N}\big)\big(x - \tfrac{k}{N}\big) + \partial_y f\big(\tfrac{k}{N},\tfrac{\ell}{N}\big)\big(y - \tfrac{\ell}{N}\big) + o(N^{-1})
		\end{align}
		uniformly in $k$, $\ell$, $x$ and $y$. Since $N^3\int_{B_{k\ell}}(x - \frac{k}{N})\,dx\,dy = N^3\int_{B_{k\ell}}(y - \frac{\ell}{N})\,dx\,dy = -\frac12$ and $\sum_{k,\ell}w_{N,k\ell}/N^2 = 1$, where $w_{N,k\ell}$ is the value of $w_N$ on $B_{k\ell}$,
		\begin{align}
			N\int h\,d\pi_N = N\sum_{k,\ell}w_{N,k\ell}\int_{B_{k\ell}}h\,du = -\frac12\int (\partial_x f + \partial_y f)_N\,d\pi_N + o(1),
		\end{align}
		where $(\partial_x f + \partial_y f)_N$ is the grid-step function equal to $(\partial_x f + \partial_y f)(k/N,\ell/N)$ on $B_{k\ell}$. Lemma~\ref{lem: TV_cost_stability} gives $\norm{\pi_N - \pi^f_{\Leb,\Leb}}_{\mathrm{TV}} \le \norm{f_N - f}_\infty \to 0$. As $(\partial_x f + \partial_y f)_N \to \partial_x f + \partial_y f$ uniformly, we obtain
		\begin{align}
			\lim_{N\to\infty} N\big[\Gamma^{\mathrm{u}}(f) - \Gamma^{\mathrm{u}}(f_N)\big] = -\frac12\int(\partial_x f + \partial_y f)\,d\pi^f_{\Leb,\Leb}.
		\end{align}
		The identity $e^{-a(x)} = \int_0^1 e^{b(y) - f(x,y)}\,dy$ holds for every $x \in [0,1]$ by the choice of $a$. Differentiating it gives $a'(x) = \int_0^1 \partial_x f(x,y)\,W(x,y)\,dy$, and hence $\int \partial_x f\,d\pi^f_{\Leb,\Leb} = a(1) - a(0)$; similarly, $\int\partial_y f\,d\pi^f_{\Leb,\Leb} = b(1) - b(0)$. Together with \eqref{eq: mallows_partition}, this proves \eqref{eq: mallows_pal}.
		
		For \eqref{eq: mallows_tripathi}, let $\tilde f := -\log W = f - a\oplus b$, which is continuously differentiable. Since $W$ has uniform marginals, Lemma~\ref{lem: pythagorean} shows that $\pi^{\tilde f}_{\Leb,\Leb} = \pi^f_{\Leb,\Leb}$, with Schr\"odinger potentials $(0,0)$ and $\Gamma^{\mathrm{u}}(\tilde f) = 0$. Since $\exp(-\sum_i \tilde f(i/N,\sigma(i)/N)) = \prod_i W(i/N,\sigma(i)/N)$, \eqref{eq: mallows_tripathi} is \eqref{eq: mallows_pal} applied to $\tilde f$.

		(ii) Let $\Phi_N$ be the fluctuation process of $\mathbf{X}_N$, so that $\Phi_N(h) = \sqrt{N}\int h\,d(\mu_{\sigma_N} - \pi_N)$ for $h \in L^\infty([0,1]^2)$, and let $g_{\ell,N}(x,y) := g_\ell(\lceil Nx\rceil/N, \lceil Ny\rceil/N)$ for $\ell \in [L]$. Since $g_{\ell,N}$ is grid-step and $\mu_{\sigma_N}$ gives mass $1/N$ to each square $B_{i\sigma_N(i)}$, we have $\Upsilon_N(g_\ell) = N\int g_{\ell,N}\,d\mu_{\sigma_N}$, so that
		\begin{align}
			N^{-1/2}\,\Upsilon_N(g_\ell) - \Phi_N(g_{\ell,N}) = \sqrt{N}\int g_{\ell,N}\,d\pi_N
		\end{align}
		is deterministic. As the $g_\ell$ are uniformly continuous, $\norm{g_{\ell,N} - g_\ell}_\infty \to 0$, so Theorem~\ref{thm: CLT} gives $(\Phi_N(g_{1,N}), \ldots, \Phi_N(g_{L,N})) \Rightarrow (\Xi(g_1), \ldots, \Xi(g_L))$ with convergence of all moments. In particular, $\E\,\Phi_N(g_{\ell,N}) \to 0$ and $\Cov(\Phi_N(g_{\ell,N}), \Phi_N(g_{\ell',N})) \to V(\mathbf{g})_{\ell\ell'}$. Since $N^{-1/2}(\Upsilon_N(\mathbf{g}) - \E\,\Upsilon_N(\mathbf{g}))$ is the centered vector $(\Phi_N(g_{\ell,N}) - \E\,\Phi_N(g_{\ell,N}))_{\ell\in[L]}$, this proves \eqref{eq: mallows_CLT_continuous}.
		
		Finally, if $f$ and $\mathbf{g}$ are Lipschitz, then $|\frac1N\sum_i g_\ell(i/N,\sigma(i)/N) - \int g_\ell\,d\mu_\sigma| \le \sqrt{2}\Lip(g_\ell)/N$ for every $\sigma \in S_N$, since $\mu_\sigma$ is uniform on the squares $B_{i\sigma(i)}$, and $|\int g_\ell\,d\pi_N - \int g_\ell\,d\pi^f_{\Leb,\Leb}| \le \norm{g_\ell}_\infty\norm{f_N - f}_\infty \le \sqrt{2}\norm{g_\ell}_\infty\Lip(f)/N$ by \eqref{eq: TV_cost_stability}. Both errors are $o(N^{-1/2})$, so the left-hand side of \eqref{eq: mallows_CLT} differs from $(\Phi_N(g_\ell))_{\ell\in[L]}$ by $o(1)$, and \eqref{eq: mallows_CLT} follows from Theorem~\ref{thm: CLT}.
	\end{proof}

	\subsection{The limiting process}

	We collect the properties of the limiting objects of Section~\ref{sec: main_fluctuations}. In particular, part~(b) provides the continuous version of the pillow in Theorem~\ref{thm: functional_limits}(ii). Let $R$ and $C$ be the distribution functions of $\rho_{\r_\infty}$ and $\rho_{\c_\infty}$, and $\mathbb{G} := \FF\circ T_\infty$ the pillow in permuton coordinates.
	
	\begin{prop}[The limiting process]
		\label{prop: limiting_field}
		Assume \ref{assumption: bounded_cost} and \ref{assumption: margins_weak_convergence}. 
		\begin{enumerate}[label=(\alph*)]
			\item \emph{(Existence.)} The process $\Xi$ in Definition~\ref{def: SB_white_noise} exists and is almost surely linear.
			\item \emph{(Pillow.)} Suppose that $\rho_{\r_\infty}$ and $\rho_{\c_\infty}$ have no atoms. Then $\mathbb{G}$ has a continuous version $\widetilde{\mathbb{G}}$, and $\widetilde{\mathbb{G}}\circ(R,C)$ is a continuous version of $\FF$ that vanishes on the boundary of $[0,1]^2$. For arbitrary limiting margins, if $\pi_\infty = \rho_{\r_\infty}\otimes\rho_{\c_\infty}$, then
			\begin{align}
				\label{eq: pillow_covariance}
				\Cov\big(\FF(u,v),\FF(u',v')\big) = \big(R(u\wedge u') - R(u)R(u')\big)\big(C(v\wedge v') - C(v)C(v')\big).
			\end{align}
		\end{enumerate}
	\end{prop}
	
	\begin{proof}
		(a) Let $\zeta$ be white noise with intensity $\pi_\infty$, i.e., the isonormal Gaussian process on $L^2(\pi_\infty)$. Since $L^2(\pi_\infty)$ is separable, we may fix an orthonormal basis $(f_n)$ of it and i.i.d.\ standard Gaussian variables $(\gamma_n)$ on a probability space $(\Omega, \mathcal{F}, \P)$. Since $\sum_{n} |\langle h, f_n \rangle_{L^2(\pi_{\infty})}|^2 = \norm{h}_{L^2(\pi_{\infty})}^2 < \infty$ for $h \in L^2(\pi_{\infty})$,  the series $\zeta(h) := \sum_n \langle h, f_n\rangle_{L^2(\pi_\infty)}\gamma_n$ converges in $L^2(\Omega)$, and $\E[\zeta(h)\zeta(h')] = \langle h, h'\rangle_{L^2(\pi_\infty)}$. Its partial sums are linear in $h$, so $\zeta(ag+bh) = a\zeta(g) + b\zeta(h)$ almost surely for all $g,h \in L^2(\pi_\infty)$ and $a,b \in \R$. Hence $g \mapsto \zeta((I-P)g)$ is a centered Gaussian process with the covariance in Definition~\ref{def: SB_white_noise}, so $\Xi$ exists. It is almost surely linear since $\zeta$ and $I - P$ are linear; as this property depends only on the finite-dimensional distributions, it holds for every process with the covariance \eqref{eq:covariance_mx_Xi}. 
		Since $P$ is an orthogonal projection and $\pi_\infty$ is a probability measure,
		\begin{align}
			\label{eq: Xi_variance_bound}
			\E\,\Xi(h)^2 = \norm{(I-P)h}^2_{L^2(\pi_\infty)} \le \norm{h}^2_{L^2(\pi_\infty)} \le \sup|h|^2 \qquad \text{for bounded Borel functions } h.
		\end{align}
		
		(b) Let $\square_{u,v} := [0,u]\times[0,v]$. The indicators of $\square_{u,v}$ and $\square_{u',v'}$ differ only on $\big((u\wedge u', u\vee u']\times[0,1]\big)\cup\big([0,1]\times(v\wedge v', v\vee v']\big)$. Since $\pi_\infty$ has marginals $\rho_{\r_\infty}$ and $\rho_{\c_\infty}$, (a) and \eqref{eq: Xi_variance_bound} give
		\begin{align}
			\label{eq: pillow_increment}
			\E|\FF(u,v) - \FF(u',v')|^2 \le \norm{\1_{\square_{u,v}} - \1_{\square_{u',v'}}}^2_{L^2(\pi_\infty)} \le |R(u) - R(u')| + |C(v) - C(v')|.
		\end{align}
		Let $q_R := q_{\rho_{\r_\infty}}$ and $q_C := q_{\rho_{\c_\infty}}$. Since $R$ and $C$ are continuous, $R(q_R(s)) = s$ and $C(q_C(t)) = t$ (Section~\ref{sec: permutons_prelim}), so \eqref{eq: pillow_increment} gives $\E|\mathbb{G}(z) - \mathbb{G}(z')|^2 \le |z - z'|_1$ for $z, z' \in [0,1]^2$. By Gaussianity, $\E|\mathbb{G}(z) - \mathbb{G}(z')|^{6} \le 15\,|z - z'|_1^3$. Since the exponent $3$ exceeds the dimension $2$, the Kolmogorov--Chentsov criterion on the two-dimensional index set \cite[Theorem~I.2.1]{revuz1999continuous} yields a continuous version $\widetilde{\mathbb{G}}$ of $\mathbb{G}$. Next, $q_R(R(u)) \le u$ and $\rho_{\r_\infty}\big((q_R(R(u)),u]\big) = R(u) - R(q_R(R(u))) = 0$, and similarly for $C$, so that $\1_{\square_{u,v}} = \1_{\square_{q_R(R(u)),q_C(C(v))}}$ holds $\pi_\infty$-a.e. Hence $\FF(u,v) = \mathbb{G}(R(u),C(v)) = \widetilde{\mathbb{G}}(R(u),C(v))$ almost surely, so that $\widetilde{\mathbb{G}}\circ(R,C)$ is a continuous version of $\FF$. Finally, $\1_{\square_{u,1}} = \1_{[0,u]}(x)$ and $\1_{\square_{1,v}} = \1_{[0,v]}(y)$ belong to $\cA$, and $\1_{\square_{u,0}}$ and $\1_{\square_{0,v}}$ vanish $\pi_\infty$-a.e.\ because $R(0) = C(0) = 0$. Thus $\FF$ vanishes almost surely at every point of the boundary of $[0,1]^2$, and by continuity, so does $\widetilde{\mathbb{G}}\circ(R,C)$ on the whole boundary, after a modification on a null event.
		
		If $\pi_\infty = \rho_{\r_\infty}\otimes\rho_{\c_\infty}$, then $Pg = \E[g\mid x] + \E[g \mid y] - \E[g]$, where the expectations refer to $\pi_\infty$. Indeed, the right-hand side belongs to $\cA$ and satisfies the normal equations \eqref{eq: normal_equations}, by the independence of the coordinates. For $g = \1_{[0,u]\times[0,v]}$, this gives $(I-P)g = (\1_{[0,u]}(x) - R(u))(\1_{[0,v]}(y) - C(v))$, and \eqref{eq: pillow_covariance} follows from Fubini's theorem.
	\end{proof}
	
	\begin{remark}[Conditioning representation]
		\label{rem: conditioning_representation}
		The process $\Xi$ is white noise with intensity $\pi_\infty$ conditioned to have vanishing marginals. Indeed, let $\zeta$ be as in the proof of Proposition~\ref{prop: limiting_field}(a) and $\cG := \sigma(\zeta(a) : a \in \cA)$. Since $\zeta$ is a linear isometry from $L^2(\pi_\infty)$ into $L^2(\Omega)$ and $\cA$ is closed (Lemmas~\ref{lem: density_bounds} and \ref{lem: sum_space}), the closed linear span of $\{\zeta(a) : a \in \cA\}$ is $\zeta(\cA)$. For centered jointly Gaussian families, conditional expectation coincides with orthogonal projection in $L^2(\Omega)$ \cite[Theorem~9.1]{janson1997gaussian}, so $\E[\zeta(g)\mid\cG] = \zeta(Pg)$. Hence $\zeta(g) - \E[\zeta(g)\mid\cG] = \zeta((I-P)g)$, which is independent of $\cG$ and has the law of $\Xi$.
	\end{remark}
	
	\subsection{Proofs of Theorem~\ref{thm: functional_limits} and Corollaries~\ref{cor: KS} and \ref{cor: product_case}}
	\label{sec: proof_functional}
	
	For $s\in\R$, let $H^s$ be the Hilbert space of formal series $f = \sum_j\hat f_je_j$ in the cosine basis of Section~\ref{sec: main_fluctuations} with $\norm{f}_{H^s}^2 := \sum_j(1+|j|^2)^s\hat f_j^2 < \infty$. Then $H^\infty = \bigcap_s H^s$, with the topology given by the norms $\norm{\cdot}_{H^s}$, and $H^{-\infty} = \bigcup_s H^{-s}$. If $s > 1$, then $\sum_j(1+|j|^2)^{-s} < \infty$, so by the Cauchy--Schwarz inequality the cosine series of every $g \in H^s$ converges absolutely and uniformly; as for $H^\infty$, we identify $g$ with its continuous version.
	
	For the tightness in Theorem~\ref{thm: functional_limits}(ii), we bound the increments of the cumulative processes $\FF_k$ of Definition~\ref{def: fluctuation_fields} in permuton coordinates, $\mathbb{G}_k := \FF_k\circ T_k$, using Corollary~\ref{cor: moments}.

	\begin{corollary}[Increments of the cumulative process in permuton coordinates]
		\label{cor: pillow_increments}
		For all $k$ with $N_k \ge N_U$ and all $z = (s,t)$ and $z' = (s',t')$ in $[0,1]^2$,
		\begin{align}
			\E\big|\mathbb{G}_k(z) - \mathbb{G}_k(z')\big|^6 \le 2C_6|z - z'|_1^3.
		\end{align}
	\end{corollary}
	
	\begin{proof}
		Let $\delta := |z - z'|_1$ and $\square_z := [0,s]\times[0,t]$. Since $T_k$ is a homeomorphism of $[0,1]^2$ with increasing coordinates, $T_k(\square_z) = [0,q_{\r_k}(s)]\times[0,q_{\c_k}(t)]$, so that $\mathbb{G}_k(z) - \mathbb{G}_k(z') = \Phi_k(g)$ with $g := \1_{T_k(\square_z)} - \1_{T_k(\square_{z'})}$, and $g\circ T_k = \1_{\square_z} - \1_{\square_{z'}}$. By Lemmas~\ref{lem: COV_permuton_to_EOT} and \ref{lem: typical_table_bridge}, $\pi_{\mathbf{Z}_k} = (T_k)_{\#}\mu$ for a permuton $\mu$. Since $\square_z\,\triangle\,\square_{z'}$ is covered by two strips of widths $|s - s'|$ and $|t - t'|$, the proof of Corollary~\ref{cor: moments} gives
		\begin{align}
			\norm{\Ek_kg}^2_{L^2(\pi_{\mathbf{Z}_k})} \le \int g^2\,d\pi_{\mathbf{Z}_k} = \int\big(\1_{\square_z} - \1_{\square_{z'}}\big)^2\,d\mu = \mu\big(\square_z\,\triangle\,\square_{z'}\big) \le \delta.
		\end{align}
		By Lemma~\ref{lem: quantile_properties}, $T_k^{-1}(A_{ij}) = I_i\times J_j$ for intervals $I_i$ and $J_j$ of lengths $\r_k(i)/N_k \ge 1/N_k$ and $\c_k(j)/N_k \ge 1/N_k$, and $T_k$ maps $I_i\times J_j$ onto $A_{ij}$ by an affine map, since $F_{\r_k}$ and $F_{\c_k}$ are linear on the corresponding intervals. Hence the value of $\Ek_kg$ on $A_{ij}$, the average of $g$ over $A_{ij}$, is the average of $\1_{\square_z} - \1_{\square_{z'}}$ over $I_i\times J_j$, that is, $\alpha_i(s)\beta_j(t) - \alpha_i(s')\beta_j(t')$ with $\alpha_i(s) := |[0,s]\cap I_i|/|I_i|$ and $\beta_j(t) := |[0,t]\cap J_j|/|J_j|$. These numbers lie in $[0,1]$, and $\alpha_i$ and $\beta_j$ are $N_k$-Lipschitz. Therefore $\norm{\Ek_kg}_\infty \le \min\{1, N_k\delta\}$ and $N_k^{-3}\norm{\Ek_kg}_\infty^6 \le N_k^{-3}\min\{1,N_k\delta\}^6 \le \delta^3$, and the claim follows from Corollary~\ref{cor: moments} with $p = 6$. 
	\end{proof}

	\begin{proof}[\textbf{Proof of Theorem~\ref{thm: functional_limits}}]
		
		(i) We first construct $\hat\Xi$ in the cosine basis. By \eqref{eq: Xi_variance_bound}, $\E\,\Xi(e_j)^2 \le \norm{e_j}_\infty^2 \le 4$. Fix $s > 1$. Since $\sum_j (1+|j|^2)^{-s}\,\E\,\Xi(e_j)^2 \le 4\sum_j(1+|j|^2)^{-s} < \infty$, the series $\hat\Xi := \sum_j \Xi(e_j)e_j$ converges in $L^2(\Omega; H^{-s})$, so $\hat\Xi \in H^{-s}$ almost surely. Since the spaces $H^{-s}$ increase in $s$ with continuous inclusions, the limit does not depend on $s$. For $g \in H^s$, the series $\langle\hat\Xi, g\rangle = \sum_j \hat g_j\Xi(e_j)$ converges almost surely by the Cauchy--Schwarz inequality, and its partial sums are equal to $\Xi$ applied to the partial sums of the cosine series of $g$. Since $s > 1$, the latter converge to $g$ uniformly, so their images under $\Xi$ converge to $\Xi(g)$ in $L^2(\Omega)$. Hence $\langle\hat\Xi,g\rangle = \Xi(g)$ almost surely, and in particular $\hat\Xi$ is a distributional version of $\Xi$. Conversely, let $\hat\Xi'$ be a distributional version of $\Xi$. Since $\Z_{\ge 0}^2$ is countable, almost surely $\langle\hat\Xi', e_j\rangle = \Xi(e_j) = \langle\hat\Xi, e_j\rangle$ for all $j$. An element of $H^{-\infty}$ is determined by its coefficients $\langle\cdot, e_j\rangle$, so $\hat\Xi' = \hat\Xi$ almost surely.
		
		Next, elements of $H^\infty$ are continuous, so by Theorem~\ref{thm: CLT}, $\E[e^{i\Phi_k(g)}] \to e^{-V(g)/2}$ for every $g \in H^\infty$. The limit is continuous on $H^\infty$, since $V^{1/2}$ is a seminorm with $V(g)^{1/2} \le \sup|g| \le 2\big(\sum_j(1+|j|^2)^{-s}\big)^{1/2}\norm{g}_{H^s}$ for $s > 1$, by the Cauchy--Schwarz inequality and $\norm{e_j}_\infty \le 2$, and it is the characteristic functional of $\hat\Xi$, since $\langle\hat\Xi, g\rangle = \Xi(g)$ is centered Gaussian with variance $V(g)$. Since the embedding $H^{s+s'}\hookrightarrow H^s$, which is diagonal in $\{e_j\}$ with singular values $(1+|j|^2)^{-s'/2}$, is Hilbert--Schmidt for $s'>1$, $H^\infty$ is a nuclear Fr\'echet space, and L\'evy's continuity theorem on its dual \cite{fernique1967processus} (see also \cite{mitoma1983tightness}) yields $\Phi_k \Rightarrow \hat\Xi$ in $H^{-\infty}$.
		
		(ii) By Assumption~\ref{assumption: margins_weak_convergence}, $N_k \ge \max(m_k,n_k) \to \infty$, so Corollary~\ref{cor: pillow_increments} applies for all large $k$. The processes $\mathbb{G}_k = \FF_k\circ T_k$ are continuous and vanish at the origin, since $\FF_k(u,v) = \sqrt{N_k}\int_0^u\int_0^v(\varphi_{\mathbf{X}_k} - \varphi_{\mathbf{Z}_k})$. Since the exponent $3$ in Corollary~\ref{cor: pillow_increments} exceeds the dimension $2$, Kolmogorov's tightness criterion \cite{kallenberg2002foundations}, applied to $\mathbb{G}_k$ composed with the $1$-Lipschitz nearest-point projection of $\R^2$ onto $[0,1]^2$, shows that $(\mathbb{G}_k)$ is tight in $C([0,1]^2)$.
		
		Now suppose that $\rho_{\r_\infty}$ and $\rho_{\c_\infty}$ have no atoms, so that $R$ and $C$ are continuous with $R(0) = C(0) = 0$. By P\'olya's theorem, $\eps_k := \sup_u|F_{\r_k}(u) - R(u)| + \sup_v|F_{\c_k}(v) - C(v)| \to 0$. For $h \in C([0,1]^2)$ and $\delta > 0$, let $\omega(h,\delta) := \sup\{|h(z) - h(z')| : |z - z'|_1 \le \delta\}$, and let $\omega_{R,C}(\delta) := \sup\{|R(u) - R(u')| + |C(v) - C(v')| : |u - u'| + |v - v'| \le \delta\}$. Since $\FF_k = \mathbb{G}_k\circ(F_{\r_k},F_{\c_k})$, we have $\omega(\FF_k,\delta) \le \omega\big(\mathbb{G}_k, 2\eps_k + \omega_{R,C}(\delta)\big)$. By the Arzel\`a--Ascoli theorem, the tightness of $(\mathbb{G}_k)$ gives $\lim_{\delta'\downarrow0}\limsup_k\P(\omega(\mathbb{G}_k,\delta') > \eta) = 0$ for every $\eta > 0$, hence $\lim_{\delta\downarrow0}\limsup_k\P(\omega(\FF_k,\delta) > \eta) = 0$, and $(\FF_k)$ is tight in $C([0,1]^2)$, since $\FF_k(0,0) = 0$. By Theorem~\ref{thm: CLT}, applied to the indicators $\1_{[0,u]\times[0,v]}$, which are continuous outside the $(\rho_{\r_\infty}\otimes\rho_{\c_\infty})$-null set $(\{u\}\times[0,1])\cup([0,1]\times\{v\})$, the finite-dimensional distributions of $\FF_k$ converge to those of $\FF$. Since finite-dimensional distributions determine laws on $C([0,1]^2)$, every subsequential limit of $(\FF_k)$ is the law of the continuous version of $\FF$ from Proposition~\ref{prop: limiting_field}(b), and $\FF_k \Rightarrow \FF$.
	\end{proof}

	\begin{remark}[Regularity of $\hat\Xi$ and convergence in $H^{-s}$]
		\label{rem: Hs_convergence}
		By the proof of Theorem~\ref{thm: functional_limits}(i), for every $s > 1$, $\hat\Xi \in H^{-s}$ almost surely and $\langle\hat\Xi, g\rangle = \Xi(g)$ almost surely for every $g \in H^s$. Moreover, $\Phi_k \Rightarrow \hat\Xi$ in $H^{-s}$. Indeed, fix $s_1 \in (1,s)$. By Corollary~\ref{cor: moments} with $p = 2$ and $\norm{e_j}_\infty \le 2$, for $N_k \ge N_U$,
		\begin{align}
			\E\norm{\Phi_k}_{H^{-s_1}}^2 = \sum_j(1+|j|^2)^{-s_1}\,\E\,\Phi_k(e_j)^2 \le 8C_2\sum_j(1+|j|^2)^{-s_1} < \infty.
		\end{align}
		Balls of $H^{-s_1}$ are compact in $H^{-s}$, so Markov's inequality shows that $(\Phi_k)$ is tight in $H^{-s}$. By Theorem~\ref{thm: CLT}, applied to the continuous functions $e_j$, the finite-dimensional distributions of the coordinates $\langle\Phi_k, e_j\rangle = \Phi_k(e_j)$ converge to those of $\hat\Xi$. These determine laws on the separable Hilbert space $H^{-s}$, so every subsequential limit is the law of $\hat\Xi$.
		
		The range $s > 1$ is sharp: if $\pi_\infty = u$, then $\hat\Xi \notin H^{-1}$ almost surely. Indeed, in this case $P e_j = 0$ for $j_1, j_2 \ge 1$, since $\int_0^1 \mathrm{c}_p = 0$ for $p \ge 1$. Hence $\{\Xi(e_j)\}_{j_1,j_2\ge 1}$ are i.i.d.\ standard Gaussian variables. With $a_j := (1+|j|^2)^{-1}$, the series $\sum_{j_1,j_2\ge 1} a_j(\Xi(e_j)^2 - 1)$ of independent centered terms with variances $2a_j^2$ converges almost surely by Kolmogorov's two-series theorem, since $\sum_j a_j^2 < \infty$, whereas $\sum_{j_1,j_2\ge1} a_j = \infty$. Hence $\sum_j a_j\Xi(e_j)^2 = \infty$ almost surely, that is, $\hat\Xi \notin H^{-1}$.

		This is precisely the regularity of white noise on $[0,1]^2$. A white noise $\mathbb{W}$ with intensity $u$ has i.i.d.\ standard Gaussian coefficients $\langle\mathbb{W}, e_j\rangle$, $j \in \Z_{\ge 0}^2$, so the two arguments above show that $\mathbb{W} \in H^{-s}$ almost surely if and only if $s > 1$, the threshold $d/2$ for white noise in dimension $d = 2$. In fact, if $\pi_\infty = u$, then $\cA$ is the closed span of the $e_j$ with $j_1 = 0$ or $j_2 = 0$, and $\hat\Xi$ has the law of $\mathbb{W}$ with these coefficients set to zero.
	\end{remark}

	\begin{proof}[\textbf{Proof of Corollary~\ref{cor: KS}}]
		For every rectangle $B$, $\sqrt{N_k}(\pi_{\mathbf{X}_k}(B) - \pi_{\mathbf{Z}_k}(B)) = \FF_k(B)$. The maps $F \mapsto \sup_B|F(B)|$ and $F \mapsto \int F^2$ are continuous on $C([0,1]^2)$ (the former is $4$-Lipschitz), so the limit laws follow from Theorem~\ref{thm: functional_limits}(ii) and the continuous mapping theorem. For the first claim, since $T_k$ is a bijection of $[0,1]^2$, $\sqrt{N_k}\sup_B|\pi_{\mathbf{X}_k}(B) - \pi_{\mathbf{Z}_k}(B)| = \sup_B|\FF_k(B)| \le 4\norm{\FF_k}_\infty = 4\norm{\mathbb{G}_k}_\infty$, which is tight under Assumptions~\ref{assumption: bounded_cost} and \ref{assumption: margins_weak_convergence} by the proof of Theorem~\ref{thm: functional_limits}(ii).
	\end{proof}
	
	\begin{proof}[\textbf{Proof of Corollary~\ref{cor: product_case}}]
		If $\kappa_\infty = a\oplus b$, then $a$ and $b$ are bounded, since $\kappa_\infty$ is, and $\rho_{\r_\infty}\otimes\rho_{\c_\infty} = e^{a\oplus b - \kappa_\infty}\rho_{\r_\infty}\otimes\rho_{\c_\infty}$ has the form required in Lemma~\ref{lem: pythagorean}, so $\pi_\infty = \rho_{\r_\infty}\otimes\rho_{\c_\infty}$. Then $\xi_\infty \equiv 1$, so $\cS = 0$ and $D_\infty = 1$. The formula for $V(g)$ follows from the expression for $P$ in the product case obtained in the proof of Proposition~\ref{prop: limiting_field}(b). Since $R$ and $C$ are nondecreasing, $R(u\wedge u') = R(u)\wedge R(u')$, so \eqref{eq: pillow_covariance} coincides with the covariance of $\mathbb{B}(R(u),C(v))$, and these centered Gaussian processes have the same finite-dimensional distributions. If $R$ and $C$ are continuous, then $\mathbb{B}\circ(R,C)$ is continuous, so it has the law of the continuous version of $\FF$.
		
		If $\bLambda_k = \mathbf{u}\mathbf{v}^\top$ has rank one, then by \eqref{eq: def_kappa}, $\kappa_k = a_k\oplus b_k$ is additively separable, and hence $\int \kappa_k\circ T_k\,d\mu_\sigma = \int a_k\,d\rho_{\r_k} + \int b_k\,d\rho_{\c_k}$ does not depend on $\sigma$, so $\Gibbs_k$ is uniform. A function is additively separable if and only if all its rectangle increments $\kappa(x,y) - \kappa(x',y) - \kappa(x,y') + \kappa(x',y')$ vanish, and this property is preserved under uniform limits; hence $\kappa_\infty \in \cA$.
	\end{proof}
	
	\subsection{Proof of Proposition~\ref{thm: smooth_cost_perturbations}}
	\label{sec: smooth_cost}

	We now prove Proposition~\ref{thm: smooth_cost_perturbations}, with the notation of Section~\ref{sec: EOT_expansion}. We show that the normalized Schr\"odinger potentials depend smoothly on the cost by the implicit function theorem, differentiate the Schr\"odinger system, and identify the derivative of the potentials with the projection $P_{\pi_t}g$.
	
	
	Let $L^\infty_0(\nu) := \{ b \in L^\infty(\nu) : \int b\,d\nu = 0\}$ and consider the Banach spaces
	\begin{align}
		\mathfrak{X} := L^{\infty}(\mu) \times L^\infty_0(\nu), \qquad \mathfrak{E} := \Big\{ (\phi, \psi) \in L^{\infty}(\mu) \times L^{\infty}(\nu) : \int \phi\,d\mu = \int \psi \,d\nu \Big\},
	\end{align}
	equipped with the norm $\norm{(a,b)} := \norm{a}_{L^\infty(\mu)} + \norm{b}_{L^\infty(\nu)}$. Define $\Theta : L^{\infty}(\mu \otimes \nu) \times \mathfrak{X} \to \mathfrak{E}$ by
	\begin{align}
		\Theta(c; a, b) := \left(  \int e^{a(\cdot) + b(y)-c(\cdot, y)}\, \nu(dy) - 1 , \ \int e^{a(x) + b(\cdot) - c(x, \cdot)}\,\mu(dx) -1  \right).
	\end{align}
	Integrating the first component against $\mu$ and the second against $\nu$ gives the same number, so $\Theta$ indeed takes values in $\mathfrak{E}$. By \eqref{eq: schrodinger_system}, $\Theta(c; a,b) = 0$ if and only if $(a,b)$ are Schr\"odinger potentials for the cost $c$ normalized by $\int b\,d\nu = 0$. Hence, for every $c \in L^\infty(\mu\otimes\nu)$, the equation $\Theta(c; \cdot) = 0$ has exactly one solution in $\mathfrak{X}$, which we denote by $\Psi(c)$. 
	
	\begin{lemma}
		\label{lem: IPFT_Justification_Phi}
		The map $\Theta$ is $C^\infty$ in the Fr\'echet sense, and its partial derivatives with respect to $c$ and to $(a,b)$ are the bounded linear operators
		\begin{align}
			\label{eq: cost_to_margins_derivative}
			\partial_c\Theta(c; a, b)g &= \Big( -\int e^{a(\cdot) + b(y) - c(\cdot, y)}g(\cdot, y)\,\nu(dy),\ -\int e^{a(x) + b(\cdot) - c(x, \cdot)}g(x, \cdot)\,\mu(dx) \Big), \\
			\label{eq: potentials_to_margins_derivative}
			\partial_{(a,b)}\Theta(c; a, b)(h_0,h_1) &= \Big(\int e^{a(\cdot) + b(y) - c(\cdot, y)}\big(h_{0}(\cdot) + h_{1}(y)\big)\,\nu(dy), \\
			&\qquad\quad \int e^{a(x) + b(\cdot) - c(x, \cdot)}\big(h_{0}(x) + h_{1}(\cdot)\big)\,\mu(dx) \Big).
		\end{align}
		Moreover, if $\Theta(c; a, b) = 0$, then $\partial_{(a,b)}\Theta(c; a, b) : \mathfrak{X} \to \mathfrak{E}$ is a linear isomorphism.
	\end{lemma}
	
	\begin{proof}
		All functions involved are bounded: $c \in L^\infty(\mu\otimes\nu)$ by assumption, and $(a,b) \in \mathfrak{X} \subset L^\infty(\mu)\times L^\infty(\nu)$ by the definition of $\mathfrak{X}$. We write $\Theta = \mathcal{I}\circ\exp\circ\,\cL - (1,1)$, where $\cL(c;a,b) := a\oplus b - c$ is a bounded linear map from $L^\infty(\mu\otimes\nu)\times\mathfrak{X}$ to $L^\infty(\mu\otimes\nu)$, $\exp(w) := e^{w}$ is the exponential on $L^\infty(\mu\otimes\nu)$, and $\mathcal{I}(\zeta) := \big(\int \zeta(\cdot,y)\,\nu(dy), \int\zeta(x,\cdot)\,\mu(dx)\big)$ is a bounded linear map from $L^\infty(\mu\otimes\nu)$ to $L^\infty(\mu)\times L^\infty(\nu)$. Since $L^\infty(\mu\otimes\nu)$ is a Banach algebra, the power series $\exp(w) = \sum_{j\ge0}w^j/j!$ converges absolutely in norm, so $\exp$ is $C^\infty$ in the Fr\'echet sense, with derivative $D\exp(w)h = e^{w}h$; indeed, $\norm{e^{w+h} - e^w - e^wh}_\infty \le \frac12 e^{\norm{w}_\infty + \norm{h}_\infty}\norm{h}_\infty^2$. By the chain rule, $\Theta$ is $C^\infty$, and its partial derivatives are obtained by applying $\mathcal{I}$ to $e^{a\oplus b - c}\,\partial\cL(c;a,b)$, which gives \eqref{eq: cost_to_margins_derivative} and \eqref{eq: potentials_to_margins_derivative}; their operator norms are at most $2e^{\norm{a}_\infty + \norm{b}_\infty + \norm{c}_\infty}$.

		Now suppose that $\Theta(c; a,b) = 0$, and let $K := e^{a\oplus b - c} = \frac{d\pi^c}{d(\mu\otimes\nu)}$. By Lemma~\ref{lem: density_bounds}, $K \ge \eps := e^{-2\osc(c)}$. Let $\cK h(x) := \int K(x,y)h(y)\,\nu(dy)$ and $\cK^\top h(y) := \int K(x,y)h(x)\,\mu(dx)$. Then
		\begin{align}
			\partial_{(a,b)}\Theta(c;a,b)(h_0,h_1) = \big(h_0 + \cK h_1,\ \cK^\top h_0 + h_1\big),
		\end{align}
		and for $(\phi,\psi) \in \mathfrak{E}$ the equation $\partial_{(a,b)}\Theta(c;a,b)(h_0,h_1) = (\phi,\psi)$ is equivalent to
		\begin{align}
			h_0 = \phi - \cK h_1, \qquad (I - \cK^\top\cK)h_1 = \psi - \cK^\top\phi.
		\end{align}
		The operator $\cQ := \cK^\top\cK$ has the kernel $q(y,y') := \int K(x,y)K(x,y')\,\mu(dx) \ge \eps^2$ with respect to $\nu$, and $\int q(y,y')\,\nu(dy') = 1$. In particular, $\cQ$ preserves $L^\infty_0(\nu)$, and for $h \in L^\infty_0(\nu)$ we have $\cQ h(y) = \int (q(y,y') - \eps^2)h(y')\,\nu(dy')$, whence $\norm{\cQ h}_\infty \le (1 - \eps^2)\norm{h}_\infty$. Thus the restriction of $\cQ$ to the Banach space $L^\infty_0(\nu)$ has operator norm at most $1 - \eps^2 < 1$. Hence the Neumann series $\sum_{j \ge 0}\cQ^j$ converges in operator norm on $L^\infty_0(\nu)$, its sum is the inverse of $I - \cQ$, and $\norm{(I-\cQ)^{-1}}_{L^\infty_0(\nu)\to L^\infty_0(\nu)} \le \sum_{j\ge0}(1-\eps^2)^j = \eps^{-2}$. Since $\int(\psi - \cK^\top\phi)\,d\nu = \int\psi\,d\nu - \int \phi\,d\mu = 0$, the second equation has a unique solution $h_1 \in L^\infty_0(\nu)$, and then $h_0$ is determined by the first. The solution depends boundedly on $(\phi,\psi)$, which proves the claim; see also \cite[Proposition~3.1]{carlier2020differential}.
	\end{proof}
	
	By Lemma~\ref{lem: IPFT_Justification_Phi} and the implicit function theorem, for every $c_0 \in L^\infty(\mu\otimes\nu)$ there are a neighborhood $\cU$ of $c_0$ and a $C^\infty$ map $\widehat\Psi : \cU \to \mathfrak{X}$ with $\Theta(c; \widehat\Psi(c)) = 0$ for all $c \in \cU$. By uniqueness of the normalized potentials, $\widehat\Psi = \Psi$ on $\cU$. Hence the map $\Psi : L^\infty(\mu\otimes\nu)\to\mathfrak{X}$ is $C^\infty$, with derivative
	\begin{align}
		\label{eq: Psi_jacobain}
		\partial \Psi(c) = - \partial_{(a,b)} \Theta\big(c; \Psi(c)\big)^{-1} \partial_c\Theta\big(c; \Psi(c)\big).
	\end{align}
	
	
	Now let $(a_t, b_t) := \Psi(f + tg)$. By the above, $t \mapsto (a_t,b_t)$ is a $C^\infty$ map from $\R$ to $\mathfrak{X}$; we write $(\dot a_t, \dot b_t)$ for its derivative, and we let $\xi_t := e^{a_t\oplus b_t - f - tg}$, the density of $\pi_t = \pi^{f+tg}$ by Lemma~\ref{lem: pythagorean}.
	
	\begin{proof}[\textbf{Proof of Proposition~\ref{thm: smooth_cost_perturbations}}]
		Differentiating the identity $\Theta(f+tg; a_t,b_t) = 0$ with respect to $t$ gives $\partial_c\Theta\, g + \partial_{(a,b)}\Theta\,(\dot a_t,\dot b_t) = 0$, where the partial derivatives are evaluated at $(f + tg; a_t, b_t)$. By \eqref{eq: cost_to_margins_derivative}, \eqref{eq: potentials_to_margins_derivative} and the identities $\int \xi_t(x,y)\,\nu(dy) = 1$ for $\mu$-a.e.\ $x$ and $\int \xi_t(x,y)\,\mu(dx) = 1$ for $\nu$-a.e.\ $y$, this is
		\begin{align}
			\label{eq: schrodinger_derivative_integral_system}
			\begin{split}
				\dot a_t(x) + \int \xi_t(x,y)\,\dot b_t(y)\,\nu(dy) &= \int \xi_t(x,y)\,g(x,y)\,\nu(dy) \qquad \text{for } \mu\text{-a.e. } x,\\
				\dot b_t(y) + \int \xi_t(x,y)\,\dot a_t(x)\,\mu(dx) &= \int \xi_t(x,y)\,g(x,y)\,\mu(dx) \qquad \text{for } \nu\text{-a.e. } y.
			\end{split}
		\end{align}
		Since $\xi_t$ is the density of $\pi_t$, this says that $\E_{\pi_t}[g - \dot a_t\oplus\dot b_t \mid X] = 0$ and $\E_{\pi_t}[g - \dot a_t\oplus\dot b_t \mid Y] = 0$ for $(X,Y) \sim \pi_t$. As $\xi_t$ is bounded away from $0$ and $\infty$ and $\dot a_t \oplus \dot b_t \in \cA$, Lemma~\ref{lem: sum_space} gives
		\begin{align}
			\label{eq: potential_derivative_projection}
			\dot a_t \oplus \dot b_t = P_{\pi_t} g.
		\end{align}
		By Lemma~\ref{lem: pythagorean}, $\Gamma(f + tg) = \int a_t\,d\mu + \int b_t\,d\nu$, which is $C^\infty$ in $t$ as the composition of a $C^\infty$ map with a bounded linear functional. Integrating the first equation in \eqref{eq: schrodinger_derivative_integral_system} against $\mu$ and using $\int \xi_t(x,y)\,\mu(dx) = 1$, we obtain
		\begin{align}
			\frac{d}{dt}\Gamma(f+tg) = \int \dot a_t\,d\mu + \int \dot b_t\,d\nu = \iint \xi_t\, g\,d(\mu\otimes\nu) = \int g\,d\pi_t.
		\end{align}
		The map $t \mapsto \xi_t$ is $C^1$ into $L^\infty(\mu\otimes\nu)$ with derivative $\xi_t(\dot a_t\oplus\dot b_t - g)$, and integration against $g$ is a bounded linear functional on $L^\infty(\mu\otimes\nu)$. Hence, by \eqref{eq: potential_derivative_projection} and since $P_{\pi_t}$ is an orthogonal projection,
		\begin{align}
			\frac{d^2}{dt^2}\Gamma(f+tg) = \int g\,\big(\dot a_t\oplus\dot b_t - g\big)\,d\pi_t = \big\langle g, (P_{\pi_t} - I)g\big\rangle_{L^2(\pi_t)} = -\norm{(I - P_{\pi_t})g}^2_{L^2(\pi_t)},
		\end{align}
		which proves \eqref{eq: EOT_derivatives}.
	\end{proof}
	
	\section{Proofs of the spectral results}
	\label{sec: proof_ESD}
	
	In this section we prove Theorem~\ref{thm: ESD_limit}. The strategy is to prove the limit for the tilted model $\mathbf{Y} \sim \Poisson(\mathbf{Z})$, whose entries are independent, and to transfer it to the conditioned table $\mathbf{X}$. Since the transference factor $(eN)^{m+n}m^{-m}n^{-n}$ of Lemma~\ref{thm: transference} is only $e^{O((m+n)\log N)}$, the transfer requires linear spectral statistics of the tilted model to concentrate at a faster rate. Section~\ref{sec: proof_ESD_concentration} establishes such concentration (Proposition~\ref{prop: spectral_concentration}), from a modified logarithmic Sobolev inequality for products of Poisson laws and bounds on the discrete gradients of spectral statistics. Section~\ref{sec: proof_ESD_proof} then proves Theorem~\ref{thm: ESD_limit}: the global law of the tilted model follows from the local law for random Gram matrices with a variance profile \cite{alt2017local}, whose deterministic approximation is given by the Dyson equation \eqref{eq: dyson_discrete}, and since the variance profile converges to the density of the limiting Schr\"odinger bridge, the continuity of the Dyson equation in its profile (Lemma~\ref{lem: dyson_continuity}) identifies the limit $\nu_\infty$.

	\subsection{Concentration of spectral statistics}
	\label{sec: proof_ESD_concentration}
	
	We fix margins $(\r,\c)$ and a prior mean $\bLambda$, with $\mathbf{X}$, $\mathbf{Y}$, $\mathbf{A}_{\mathbf{X}}$ and $\mathbf{A}_{\mathbf{Y}}$ as in Section~\ref{sec: main_spectral}. For $f \in C^2([0,\infty))$ such that $f'$ vanishes on $[R^2,\infty)$ for some $R \ge 1$, we set $L_1 := 2R\norm{f'}_\infty$ and $L_2 := 2\norm{f'}_\infty + 4R^2\norm{f''}_\infty$, and with $C_{\mathbf{Z}} := \max_{i,j}\mathbf{Z}_{ij}/\bar s$ we let
	\begin{align}
		\label{eq: def_psi_spectral}
		\psi(t) := \frac{m\,t\sqrt{(m+n)\bar s}}{4L_1}\,\log\left(1 + \frac{2mL_1t}{C_{\mathbf{Z}}\sqrt{(m+n)\bar s}\,\big(4L_1^2 + L_2^2/\bar s\big)}\right), \qquad t \ge 0.
	\end{align}

	\begin{prop}[Exponential concentration of spectral statistics]
		\label{prop: spectral_concentration}
		Let $f$ and $R$ be as above. Then for all $t \ge 0$,
		\begin{align}
			\P\left(\left|\int f\,d\mu_{\mathbf{A}_{\mathbf{Y}}\mathbf{A}_{\mathbf{Y}}^\top} - \E\int f\,d\mu_{\mathbf{A}_{\mathbf{Y}}\mathbf{A}_{\mathbf{Y}}^\top}\right| \ge t\right) &\le 2e^{-\psi(t)}, \\
			\P\left(\left|\int f\,d\mu_{\mathbf{A}_{\mathbf{X}}\mathbf{A}_{\mathbf{X}}^\top} - \E\int f\,d\mu_{\mathbf{A}_{\mathbf{Y}}\mathbf{A}_{\mathbf{Y}}^\top}\right| \ge t\right) &\le 2\exp\Big(m\log\frac{eN}{m} + n\log\frac{eN}{n} - \psi(t)\Big),
		\end{align}
		with $\psi$ as in \eqref{eq: def_psi_spectral}.
	\end{prop}

	The proof of Proposition~\ref{prop: spectral_concentration} has two ingredients. The first is a modified logarithmic Sobolev inequality for products of Poisson laws. For a probability measure $\nu$ and a positive function $f$, let
	\begin{align}
		\Ent_{\nu}(f) := \int f\log f \, d\nu - \int f \, d\nu\,\log \int f \, d\nu
	\end{align}
	denote the entropy of $f$ with respect to $\nu$. For $g : \Z_{\ge 0} \to \R$, let $Dg(x) := g(x+1) - g(x)$, and for $g : \Z_{\ge 0}^{d} \to \R$ and $1 \le i \le d$, let $D_ig(x) := g(x + e_i) - g(x)$, where $e_i$ is the $i$th standard basis vector. The following lemma is a slight generalization of \cite[Proposition~11]{bobkov1998modified}, in which the coordinates may have different laws.
	
	\begin{lemma}
		\label{lem:LSI_concentration}
		Let $\nu_1, \ldots, \nu_d$ be probability measures on $\Z_{\ge 0}$. Assume that there are $c_1, c_2 > 0$ such that for every $\tau > 0$ and every $g : \Z_{\ge 0}\to\R$ with $\sup_{x}|Dg(x)|\le \tau$,
		\begin{align}
			\Ent_{\nu_i}(e^g)\le c_1e^{c_2\tau}\,\E_{\nu_i}\big[|Dg|^2e^g\big], \qquad i = 1, \ldots, d.
		\end{align}
		Let $\nu := \nu_1\otimes\cdots\otimes\nu_d$. If $f:\Z_{\ge 0}^{d}\to\R$ satisfies, for some $\alpha,\beta > 0$ and all $x\in \Z_{\ge 0}^{d}$,
		\begin{align}
			\sum_{i=1}^{d}|D_if(x)|^2\le \alpha^2 \quad\text{and}\quad \max_{1\le i\le d}|D_if(x)|\le \beta,
		\end{align}
		then for every $t\ge 0$,
		\begin{align}
			\nu\big(f\ge \E_{\nu}[f]+t\big)\le \exp\left(-\frac{t}{2c_2\beta}\log\left(1+\frac{\beta c_2t}{4c_1\alpha^2}\right)\right).
		\end{align}
	\end{lemma}
	
	\begin{proof}
		It suffices to establish the inequality obtained from the tensorization step, namely $\Ent_{\nu}(e^f)\le c_1e^{c_2\tau}\E_{\nu}[|Df|^2e^f]$ with $|Df|^2 := \sum_{i}|D_if|^2$, for every $f$ with $\max_i\sup|D_if| \le \tau$. By the subadditivity of entropy \cite[Proposition~1.4.1]{ane2000inegalites},
		\begin{align}
			\Ent_{\nu}(e^f) \le \sum_{i=1}^{d}\E_{\nu}\big[\Ent_{\nu_i}(e^f)\big] \le \sum_{i=1}^{d}\E_{\nu}\Big[c_1e^{c_2\tau}\,\E_{\nu_i}\big[|D_if|^2e^f\big]\Big] = c_1e^{c_2\tau}\,\E_{\nu}\big[|Df|^2e^f\big],
		\end{align}
		where $\Ent_{\nu_i}$ and $\E_{\nu_i}$ act on the $i$th coordinate. The rest of the argument is identical to the proof of \cite[Proposition~11]{bobkov1998modified}.
	\end{proof}
	
	By the modified logarithmic Sobolev inequality for Poisson distributions \cite{bobkov1998modified}, $\Poisson(\lambda)$ with $\lambda \le \bar\lambda$ satisfies
	\begin{align}
		\label{eq: poisson_MLSI}
		\Ent_{\Poisson(\lambda)}(e^g)\le \bar\lambda\, e^{2\tau}\,\E_{\Poisson(\lambda)}\big[|Dg|^2e^g\big] \qquad\text{whenever } \sup|Dg| \le \tau.
	\end{align}

	The second ingredient is a bound on the discrete gradient of spectral statistics. For $\mathbf{B} \in \R^{m\times n}$, let $\widetilde{\mathbf{B}} := \begin{psmallmatrix} 0 & \mathbf{B}^\top \\ \mathbf{B} & 0\end{psmallmatrix}$ be its Hermitian dilation, and let $\mathbf{E}_{ij}$ denote the $m\times n$ matrix units. For $h \in C^2(\R)$ and $\eta > 0$, let $F(\mathbf{B}) := \tr h(\eta\widetilde{\mathbf{B}})$ for $\mathbf{B} \in \R^{m\times n}$.
	
	\begin{lemma}
		\label{lem: spectral_gradient}
		Let $h \in C^2(\R)$ with $\norm{h'}_\infty \le L_1$ and $\norm{h''}_\infty \le L_2$, and let $\eta > 0$. Then for every $\mathbf{B} \in \R^{m\times n}$,
		\begin{align}
			\max_{i,j}\big|F(\mathbf{B} + \mathbf{E}_{ij}) - F(\mathbf{B})\big| \le 2L_1\eta \qquad\text{and}\qquad \sum_{i,j}\big|F(\mathbf{B} + \mathbf{E}_{ij}) - F(\mathbf{B})\big|^2 \le 4(m+n)L_1^2\eta^2 + 4mnL_2^2\eta^4.
		\end{align}
	\end{lemma}
	
	\begin{proof}
		For symmetric matrices $\mathbf{H}, \mathbf{H}'$ with eigenvalues $\lambda_1 \ge \lambda_2 \ge \cdots$ and $\lambda'_1 \ge \lambda'_2 \ge \cdots$, Lidskii's theorem gives $\sum_\ell|\lambda_\ell - \lambda'_\ell| \le \norm{\mathbf{H} - \mathbf{H}'}_{S_1}$, where $\norm{\cdot}_{S_1}$ is the trace norm \cite[Chapter~III]{bhatia1997matrix}. Hence $|\tr h(\mathbf{H}) - \tr h(\mathbf{H}')| \le L_1\norm{\mathbf{H} - \mathbf{H}'}_{S_1}$. The matrix $\widetilde{\mathbf{E}}_{ij}$ has eigenvalues $\pm 1$ and $0$, so $\norm{\eta\widetilde{\mathbf{E}}_{ij}}_{S_1} = 2\eta$, which proves the first bound.
		
		For the second bound, recall that $\frac{d}{d\eps}\tr h(\mathbf{H} + \eps \mathbf{H}')|_{\eps=0} = \tr(h'(\mathbf{H})\mathbf{H}')$. With $\mathbf{G}(\mathbf{B}) := h'(\eta\widetilde{\mathbf{B}})$, the partial derivatives of $F$ are
		\begin{align}
			\partial_{ij}F(\mathbf{B}) = \eta\tr\big(\mathbf{G}(\mathbf{B})\,\widetilde{\mathbf{E}}_{ij}\big) = 2\eta\,\mathbf{G}(\mathbf{B})_{n+i,\,j}.
		\end{align}
		Since $\mathbf{G}(\mathbf{B})$ is symmetric, each entry $\mathbf{G}(\mathbf{B})_{n+i,j}$ appears twice in $\mathbf{G}(\mathbf{B})$, and therefore
		\begin{align}
			\sum_{i,j}\partial_{ij}F(\mathbf{B})^2 \le 2\eta^2\norm{\mathbf{G}(\mathbf{B})}_F^2 \le 2(m+n)L_1^2\eta^2.
		\end{align}
		Next, for symmetric $\mathbf{H}, \mathbf{H}'$ with spectral decompositions $\mathbf{H} = \sum_\ell\lambda_\ell u_\ell u_\ell^\top$ and $\mathbf{H}' = \sum_{\ell}\lambda'_{\ell}u'_{\ell}u'^\top_{\ell}$, and a Lipschitz function $g$,
		\begin{align}
			\norm{g(\mathbf{H}) - g(\mathbf{H}')}_F^2 &= \sum_{\ell,\ell'}\big(g(\lambda_\ell) - g(\lambda'_{\ell'})\big)^2\langle u_\ell,u'_{\ell'}\rangle^2 \\
			&\le \Lip(g)^2\sum_{\ell,\ell'}(\lambda_\ell - \lambda'_{\ell'})^2\langle u_\ell,u'_{\ell'}\rangle^2 = \Lip(g)^2\norm{\mathbf{H} - \mathbf{H}'}_F^2.
		\end{align}
		Applying this with $g = h'$ and using $\norm{\widetilde{\mathbf{E}}_{ij}}_F = \sqrt{2}$, we obtain $|\partial_{ij}F(\mathbf{B} + \vartheta \mathbf{E}_{ij}) - \partial_{ij}F(\mathbf{B})| \le 2\eta\norm{\mathbf{G}(\mathbf{B}+\vartheta \mathbf{E}_{ij}) - \mathbf{G}(\mathbf{B})}_F \le 2\sqrt{2}\,\vartheta L_2\eta^2$ for $\vartheta \in [0,1]$. Since $F(\mathbf{B} + \mathbf{E}_{ij}) - F(\mathbf{B}) = \int_0^1\partial_{ij}F(\mathbf{B} + \vartheta \mathbf{E}_{ij})\,d\vartheta$, it follows that $|F(\mathbf{B} + \mathbf{E}_{ij}) - F(\mathbf{B}) - \partial_{ij}F(\mathbf{B})| \le \sqrt 2L_2\eta^2$, and hence
		\begin{align}
			\sum_{i,j}|F(\mathbf{B}+\mathbf{E}_{ij}) - F(\mathbf{B})|^2 \le 2\sum_{i,j}\partial_{ij}F(\mathbf{B})^2 + 4mnL_2^2\eta^4 \le 4(m+n)L_1^2\eta^2 + 4mnL_2^2\eta^4.
		\end{align}
	\end{proof}

	\begin{proof}[\textbf{Proof of Proposition~\ref{prop: spectral_concentration}}] 
		Let $\eta := ((m+n)\bar s)^{-1/2}$ and $h(x) := f(x^2)$ for $x \in \R$. Then $h \in C^2(\R)$, $h'(x) = 2xf'(x^2)$ and $h''(x) = 2f'(x^2) + 4x^2f''(x^2)$, and both vanish for $|x| \ge R$; hence $\norm{h'}_\infty \le L_1$ and $\norm{h''}_\infty \le L_2$. For $\mathbf{B} \in \R^{m\times n}$, the eigenvalues of $\eta\widetilde{\mathbf{B}}$ are $\pm\eta\sigma_\ell(\mathbf{B})$, $\ell \le m\wedge n$, together with $|m-n|$ zeros, while the eigenvalues of $\eta^2\mathbf{B}\mathbf{B}^\top$ are $\eta^2\sigma_\ell(\mathbf{B})^2$, $\ell \le m\wedge n$, together with $(m-n)^+$ zeros. Hence, with $F(\mathbf{B}) := \tr h(\eta\widetilde{\mathbf{B}})$,
		\begin{align}
			\label{eq: gram_vs_dilation}
			\sum_{i=1}^m f\big(\lambda_i(\eta^2\mathbf{B}\mathbf{B}^\top)\big) = \frac12F(\mathbf{B}) + \frac12(m-n)f(0).
		\end{align}
		In particular, $m\int f\,d\mu_{\mathbf{A}_{\mathbf{Y}}\mathbf{A}_{\mathbf{Y}}^\top} = \frac12F(\mathbf{Y} - \mathbf{Z}) + \frac12(m-n)f(0)$, and similarly for $\mathbf{X}$.
		
		The entries $\mathbf{Y}_{ij}$ are independent with $\mathbf{Y}_{ij} \sim \Poisson(\mathbf{Z}_{ij})$ and $\mathbf{Z}_{ij} \le C_{\mathbf{Z}}\bar s$, so by \eqref{eq: poisson_MLSI} their laws satisfy the hypothesis of Lemma~\ref{lem:LSI_concentration} with $c_1 := C_{\mathbf{Z}}\bar s$ and $c_2 := 2$. By Lemma~\ref{lem: spectral_gradient}, the functions $\mathbf{B} \mapsto \pm F(\mathbf{B} - \mathbf{Z})$ satisfy the hypotheses of Lemma~\ref{lem:LSI_concentration} with $\beta := 2L_1\eta$ and $\alpha^2 := 4(m+n)L_1^2\eta^2 + 4mnL_2^2\eta^4 \le (4L_1^2 + L_2^2/\bar s)/\bar s$, where we used $mn \le (m+n)^2/4$. Lemma~\ref{lem:LSI_concentration}, applied to both functions with deviation $2mt$, gives
		\begin{align}
			\P\big(|F(\mathbf{Y} - \mathbf{Z}) - \E F(\mathbf{Y}-\mathbf{Z})| \ge 2mt\big) \le 2\exp\left(-\frac{2mt}{4\beta}\log\left(1 + \frac{2mt\,\beta}{2c_1\alpha^2}\right)\right) \le 2e^{-\psi(t)},
		\end{align}
		since $\frac{2mt}{4\beta} = \frac{mt\sqrt{(m+n)\bar s}}{4L_1}$ and $\frac{2mt\,\beta}{2c_1\alpha^2} \ge \frac{2mL_1t}{C_{\mathbf{Z}}\sqrt{(m+n)\bar s}\,(4L_1^2 + L_2^2/\bar s)}$. By \eqref{eq: gram_vs_dilation}, this is the first claim. The second claim follows from the first and Lemma~\ref{thm: transference}, applied to the indicator of the event in the first claim, viewed as a function of $\mathbf{Y}$, since $\log\big((eN)^{m+n}m^{-m}n^{-n}\big) = m\log(eN/m) + n\log(eN/n)$.
	\end{proof}

	\subsection{Proof of Theorem~\ref{thm: ESD_limit}}
	\label{sec: proof_ESD_proof}
	
	The proof uses two lemmas: the stability of the continuum Dyson equation \eqref{eq: dyson_continuum} under perturbations of its profile, and the convergence of the profiles $\varphi_{\mathbf{Z}_k}$. For the first, fix $C \ge 0$ and $z \in \mathbb{H}$ with $\operatorname{Im}z \ge 2C + 1$, and let
	\begin{align}
		\mathfrak{B}_z := \big\{G \in L^\infty([0,1];\mathbb{C}) : \operatorname{Im}G \ge 0 \text{ and } |G| \le (\operatorname{Im}z)^{-1} \text{ a.e.}\big\}.
	\end{align}
	For $\theta \in (0,1)$ and a measurable $W : [0,1]^2 \to [0,C]$, define the map $\mathcal{T}_{W,\theta}$ on $\mathfrak{B}_z$ by
	\begin{align}
		\mathcal{T}_{W,\theta}(G)(x) := -\left(z - (1-\theta)\int_0^1\frac{W(x,y)\,dy}{1 + \theta\int_0^1 W(x',y)G(x')\,dx'}\right)^{-1},
	\end{align}
	so that the solutions of \eqref{eq: dyson_continuum} in $\mathfrak{B}_z$ are the fixed points of $\mathcal{T}_{W,\theta}$.
	
	\begin{lemma}[Stability of the Dyson equation]
		\label{lem: dyson_continuity}
		Let $C$, $z$, $\theta$ and $W$ be as above.
		\begin{enumerate}[label=(\alph*)]
			\item $\mathcal{T}_{W,\theta}$ maps $\mathfrak{B}_z$ into itself, and $\norm{\mathcal{T}_{W,\theta}(G) - \mathcal{T}_{W,\theta}(G')}_{L^\infty} \le \frac14\norm{G - G'}_{L^1}$ for all $G, G' \in \mathfrak{B}_z$. In particular, \eqref{eq: dyson_continuum} has a unique solution $G_{W,\theta}(\cdot,z) \in \mathfrak{B}_z$.
			\item If $W$ is grid-step with respect to an $m\times n$ grid, with values $w_{ij}$, and $\theta = m/(m+n)$, and if $(G_i)_{i \in [m]}$ is the solution of \eqref{eq: dyson_discrete} with $\mathbf{S}_{ij} = w_{ij}/(m+n)$, then $G_{W,\theta}(x,z) = G_{\lceil mx \rceil}(z)$ for a.e.\ $x$. In particular, $\int_0^1 G_{W,\theta}(x,z)\,dx$ is the Stieltjes transform of $\nu_{\mathbf{S}}$ at $z$.
			\item If $\theta_k \in (0,1)$ with $\theta_k \to \theta$, and $W_k : [0,1]^2 \to [0,C]$ are measurable functions with $W_k \to W$ in $L^1([0,1]^2)$, then $G_{W_k,\theta_k}(\cdot,z) \to G_{W,\theta}(\cdot,z)$ in $L^1[0,1]$.
		\end{enumerate}
	\end{lemma}
	
	\begin{proof}
		(a) For $G \in \mathfrak{B}_z$, let $a(y) := \theta\int_0^1 W(x',y)G(x')\,dx'$. Then $\operatorname{Im}a \ge 0$ and $|a| \le \theta C/\operatorname{Im}z \le \frac12$, so $|1 + a| \ge \frac12$ and $\operatorname{Im}(1+a)^{-1} = -\operatorname{Im}a/|1+a|^2 \le 0$. Consequently, $u(x) := (1-\theta)\int_0^1 W(x,y)(1 + a(y))^{-1}\,dy$ satisfies $\operatorname{Im}u \le 0$, so that $\operatorname{Im}(z - u) \ge \operatorname{Im}z$. Hence $\mathcal{T}_{W,\theta}(G) = -(z-u)^{-1}$ satisfies $|\mathcal{T}_{W,\theta}(G)| \le (\operatorname{Im}z)^{-1}$ and $\operatorname{Im}\mathcal{T}_{W,\theta}(G) = \operatorname{Im}(z - u)/|z-u|^2 \ge 0$. If $G' \in \mathfrak{B}_z$ with corresponding $a'$ and $u'$, then $|(1+a)^{-1} - (1+a')^{-1}| \le 4|a - a'| \le 4\theta C\norm{G - G'}_{L^1}$, so $|u - u'| \le 4\theta(1-\theta)C^2\norm{G - G'}_{L^1} \le C^2\norm{G - G'}_{L^1}$ and
		\begin{align}
			\big|\mathcal{T}_{W,\theta}(G) - \mathcal{T}_{W,\theta}(G')\big| = \frac{|u - u'|}{|z - u|\,|z-u'|} \le \frac{C^2}{(\operatorname{Im}z)^2}\norm{G - G'}_{L^1} \le \frac14\norm{G - G'}_{L^1}.
		\end{align}
		Since $\mathfrak{B}_z$ is complete for the $L^1$ norm (it is closed in $L^1$, as its defining constraints are preserved by almost everywhere convergence) and $\norm{\cdot}_{L^1} \le \norm{\cdot}_{L^\infty}$, Banach's fixed point theorem applies.
		
		(b) Let $G(x) := G_{\lceil mx\rceil}(z)$, so that $\operatorname{Im}G > 0$. For $x \in \big(\frac{i-1}{m},\frac{i}{m}\big]$, a direct computation gives
		\begin{align}
			(1-\theta)\int_0^1\frac{W(x,y)\,dy}{1 + \theta\int_0^1 W(x',y)G(x')\,dx'} = \sum_{l=1}^n\frac{w_{il}/(m+n)}{1 + \sum_{j=1}^m \frac{w_{jl}}{m+n}G_j(z)} = \sum_{l=1}^n\frac{\mathbf{S}_{il}}{1 + \sum_{j=1}^m \mathbf{S}_{jl}G_j(z)},
		\end{align}
		so $G$ solves \eqref{eq: dyson_continuum}. By the argument in (a), any solution with $\operatorname{Im}G > 0$ satisfies $|G| \le (\operatorname{Im}z)^{-1}$, since the bound $\operatorname{Im}u \le 0$ only uses $\operatorname{Im}G \ge 0$. Hence $G \in \mathfrak{B}_z$, and $G = G_{W,\theta}$ by uniqueness.
		
		(c) Write $G := G_{W,\theta}$, $G_k := G_{W_k,\theta_k}$, $\mathcal{T} := \mathcal{T}_{W,\theta}$ and $\mathcal{T}_k := \mathcal{T}_{W_k,\theta_k}$. By (a),
		\begin{align}
			\norm{G_k - G}_{L^1} \le \norm{\mathcal{T}_k(G_k) - \mathcal{T}_k(G)}_{L^1} + \norm{\mathcal{T}_k(G) - \mathcal{T}(G)}_{L^1} \le \frac14\norm{G_k - G}_{L^1} + \norm{\mathcal{T}_k(G) - \mathcal{T}(G)}_{L^1},
		\end{align}
		so it suffices to show that $\norm{\mathcal{T}_k(G) - \mathcal{T}(G)}_{L^1} \to 0$. Let $a_k, u_k$ and $a, u$ be defined as in (a), with $(W_k,\theta_k)$ and $(W,\theta)$, respectively, and with the same $G$. Then $|\mathcal{T}_k(G) - \mathcal{T}(G)| \le |u_k - u|/(\operatorname{Im}z)^2$ and
		\begin{align}
			|a_k - a|(y) \le \frac{C|\theta_k - \theta|}{\operatorname{Im}z} + \frac{1}{\operatorname{Im}z}\int_0^1|W_k - W|(x',y)\,dx'.
		\end{align}
		Using $|1 + a_k|, |1+a| \ge \frac12$, we obtain
		\begin{align}
			|u_k - u|(x) \le 2C|\theta_k - \theta| + 2\int_0^1|W_k - W|(x,y)\,dy + 4C\int_0^1|a_k - a|(y)\,dy.
		\end{align}
		Integrating in $x$ shows that $\norm{u_k - u}_{L^1} \to 0$, which completes the proof.
	\end{proof}

	For the rest of this subsection, we assume the hypotheses of Theorem~\ref{thm: ESD_limit}. The limiting margins $\rho_{\r_\infty}$ and $\rho_{\c_\infty}$ then also have densities with values in $[\delta,\delta^{-1}]$, since they are $L^1$ limits of such densities. With $K_1$ as in \eqref{eq: def_K1}, we set
	\begin{align}
		\label{eq: def_cW}
		c_W := \delta^2 e^{-2K_1}, \qquad C_W := \delta^{-2}e^{2K_1}.
	\end{align}
	
	\begin{lemma}[Stability of the bridge densities]
		\label{lem: bridge_density_stability}
		\leavevmode
		\begin{enumerate}[label=(\alph*)]
			\item For every $k$, $c_W \le \varphi_{\mathbf{Z}_k} \le C_W$, and $c_W \le W_\infty \le C_W$.
			\item $\varphi_{\mathbf{Z}_k} \to W_\infty$ in $L^1(u)$.
		\end{enumerate}
	\end{lemma}
	
	\begin{proof}
		(a) By Lemmas~\ref{lem: typical_table_bridge} and \ref{lem: density_bounds}, $\frac{d\pi_{\mathbf{Z}_k}}{d(\rho_{\r_k}\otimes\rho_{\c_k})} \in [e^{-2K_1},e^{2K_1}]$, and the claim follows from $\delta \le \rho_{\r_k},\rho_{\c_k}\le\delta^{-1}$. The same argument applies to $\pi_\infty$, since $\osc(\kappa_\infty) \le K_1$.
		
		(b) As in the proof of Theorem~\ref{thm: conditioned_model_weak_as_limit}, $\pi_{\mathbf{Z}_k}$ is the Schr\"odinger bridge between $\rho_{\r_k}$ and $\rho_{\c_k}$ with reference density $R_k := e^{-\kappa_k}\rho_{\r_k}\otimes\rho_{\c_k}$, and similarly $\pi_\infty$ is the Schr\"odinger bridge between $\rho_{\r_\infty}$ and $\rho_{\c_\infty}$ with reference $R_\infty := e^{-\kappa_\infty}\rho_{\r_\infty}\otimes\rho_{\c_\infty}$, since $\int \kappa_\infty\,d\pi + \KL(\pi\Vert\rho_{\r_\infty}\otimes\rho_{\c_\infty}) = \KL(\pi\Vert R_\infty)$ for $\pi \in \Pi(\rho_{\r_\infty},\rho_{\c_\infty})$. Since $\kappa_k \to \kappa_\infty$ uniformly and the margin densities are bounded and converge in $L^1$, $R_k \to R_\infty$ in $L^1(u)$; in particular, $R_\infty$ is a probability density and $d_H(R_k,R_\infty)^2 \le \norm{R_k - R_\infty}_{L^1(u)} \to 0$. Moreover, $R_k/(\rho_{\r_k}\otimes\rho_{\c_k})$ and $R_\infty/(\rho_{\r_\infty}\otimes\rho_{\c_\infty})$ take values in $[M^{-1},M]$ with $M = \exp(\sup_k\norm{\kappa_k}_\infty)$ as before Theorem~\ref{thm: conditioned_model_weak_as_limit}, and the margin densities are bounded below by $\delta$. Hence $H(\pi_{\mathbf{Z}_k},\pi_\infty) \to 0$ by Proposition~\ref{prop: hellinger_stability}, and $\norm{\varphi_{\mathbf{Z}_k} - W_\infty}_{L^1(u)} \le 2H(\pi_{\mathbf{Z}_k},\pi_\infty)$ by the Cauchy--Schwarz inequality.
	\end{proof}

	\begin{proof}[\textbf{Proof of Theorem~\ref{thm: ESD_limit}}]
		By Lemma~\ref{lem: bridge_density_stability}(a), the values $w_{ij}$ of $\varphi_{\mathbf{Z}_k}$ in \eqref{eq: variance_profile} satisfy $c_W \le w_{ij} \le C_W$, so that $C_{\mathbf{Z}_k} \le C_W$; moreover $\norm{W_\infty}_\infty \le C_W$. Let $\theta_k := m_k/(m_k+n_k) \to \theta$ and write $\nu_k := \nu_{\mathbf{S}_k}$.
		
		\emph{Step 1: concentration.} Fix $f$ as in Proposition~\ref{prop: spectral_concentration} and $t > 0$, and let $\tau_k := m\log(eN/m) + n\log(eN/n) = m\log(e\bar s n) + n\log(e\bar s m)$, where we suppress the index $k$. Since $\bar s \ge c_0$ and $m/n \to \gamma$, we have $\tau_k = O\big((m+n)(1 + |\log\bar s| + \log(m+n))\big)$. Let $a := \sqrt{(m+n)\bar s}$ and $\Lambda_f := 4L_1^2 + L_2^2/c_0$. Since $C_{\mathbf{Z}_k} \le C_W$ and $4L_1^2 + L_2^2/\bar s \le \Lambda_f$, we have $\psi(t) \ge \frac{mta}{4L_1}\log(1 + x)$ with $x := \frac{2mL_1t}{C_Wa\Lambda_f}$. Since $\log(1+x) \ge \frac12\min\{x,1\}$ and $a \ge \sqrt{m\bar s}$, we get $\psi(t) \ge \min\big\{\frac{mt\sqrt{m\bar s}}{8L_1}, \frac{m^2t^2}{4C_W\Lambda_f}\big\}$. Both terms are of larger order than $\tau_k$: the first because $\sqrt{m\bar s} \gg 1 + |\log\bar s| + \log m$ (as $\bar s \ge c_0$), and the second because $\log\bar s = o(m)$. Hence $\tau_k - \psi(t) \to -\infty$, so the right-hand sides in Proposition~\ref{prop: spectral_concentration} tend to zero.
		
		\emph{Step 2: the tilted model.} The matrix $\mathbf{A}_{\mathbf{Y}_k}$ has independent centered entries with variances $\mathbf{S}_{ij} = w_{ij}/(m+n) \in \big[\frac{c_W}{m+n},\frac{C_W}{m+n}\big]$. We verify the assumptions of the local law for random Gram matrices \cite{alt2017local}, as in the proof of \cite[Theorem~2.17]{duan2026scaling}. The upper bound on $\mathbf{S}_{ij}$ is the flatness assumption (A), and the lower bound implies assumption (B) \cite[Remark~2.8]{alt2017local}. The ratio $m/n$ stays bounded away from $0$ and $\infty$, which is assumption (D). For the moment assumption (C), recall that the $2p$th central moment of $\Poisson(\lambda)$ is a polynomial in $\lambda$ of degree $p$ with nonnegative coefficients and no constant term, so $\E(\mathbf{Y}_{ij} - \mathbf{Z}_{ij})^{2p} \le A_p(\mathbf{Z}_{ij} + \mathbf{Z}_{ij}^p)$ for a constant $A_p$ depending only on $p$. Since $\mathbf{Z}_{ij} \le C_W\bar s$ and $\bar s \ge c_0$, this gives $\E\big|(\mathbf{Y}_{ij} - \mathbf{Z}_{ij})/\sqrt{\bar s}\big|^{2p} \le A_p\big(C_Wc_0^{1-p} + C_W^p\big)$ for every $p \ge 1$. Hence, by the local law and rigidity estimates of \cite[Theorems~2.7 and 2.9]{alt2017local}, $d_{\mathrm{BL}}\big(\mu_{\mathbf{A}_{\mathbf{Y}_k}\mathbf{A}_{\mathbf{Y}_k}^\top}, \nu_k\big) \to 0$ in probability. Moreover, by \cite[Theorem~2.1]{alt2017local} and scaling, $\nu_k$ is supported in $[0,4C_W]$.
		
		\emph{Step 3: the limit of $\nu_k$.} Let $z \in \mathbb{H}$ with $\operatorname{Im}z \ge 2C_W + 1$. By Lemma~\ref{lem: dyson_continuity}(b), applied with $W = \varphi_{\mathbf{Z}_k}$ and $C = C_W$, the Stieltjes transform of $\nu_k$ at $z$ equals $\int_0^1 G_{\varphi_{\mathbf{Z}_k},\theta_k}(x,z)\,dx$. By Lemma~\ref{lem: bridge_density_stability}(b), $\varphi_{\mathbf{Z}_k} \to W_\infty$ in $L^1(u)$, and hence, by Lemma~\ref{lem: dyson_continuity}(c), the Stieltjes transform of $\nu_k$ at $z$ converges to $\int_0^1 G_{W_\infty,\theta}(x,z)\,dx$. Since the measures $\nu_k$ are supported in $[0,4C_W]$, every subsequence has a further subsequence converging weakly to a probability measure $\nu'$ on $[0,4C_W]$, and the Stieltjes transform of $\nu'$ equals $\int_0^1 G_{W_\infty,\theta}(x,z)\,dx$ for all $z$ with $\operatorname{Im}z \ge 2C_W + 1$. Since the Stieltjes transform of a finite measure is analytic on $\mathbb{H}$ and determines the measure, $\nu'$ does not depend on the subsequence. Hence $\nu_k \Rightarrow \nu_\infty$, where $\nu_\infty$ is the unique probability measure whose Stieltjes transform coincides with $\int_0^1 G_{W_\infty,\theta}(x,z)\,dx$ for $\operatorname{Im}z \ge 2C_W + 1$. By Lemma~\ref{lem: dyson_continuity}(a) with $C = \norm{W_\infty}_\infty$, $G_{W_\infty,\theta}(\cdot,z)$ is the solution of \eqref{eq: dyson_continuum} described before the theorem for every $z$ with $\operatorname{Im}z \ge 2\norm{W_\infty}_\infty + 1$. It is the limit of the iterates $\mathcal{T}^{j}_{W_\infty,\theta}(0)$, $j \to \infty$, uniformly in $x$ and $z$, and by induction each iterate is analytic in $z$. Hence $z \mapsto \int_0^1 G_{W_\infty,\theta}(x,z)\,dx$ is analytic on $\{\operatorname{Im}z > 2\norm{W_\infty}_\infty + 1\}$, and by the identity theorem and continuity it coincides with the Stieltjes transform of $\nu_\infty$ on $\{\operatorname{Im}z \ge 2\norm{W_\infty}_\infty + 1\}$.
		
		\emph{Step 4: conclusion.} Let $f \in C^2_c([0,\infty))$, so that $f'$ vanishes on $[R^2,\infty)$ for some $R \ge 1$. Then
		\begin{align}
			\Big|\int f\,d\mu_{\mathbf{A}_{\mathbf{X}}\mathbf{A}_{\mathbf{X}}^\top} - \int f\,d\nu_\infty\Big| &\le \Big|\int f\,d\mu_{\mathbf{A}_{\mathbf{X}}\mathbf{A}_{\mathbf{X}}^\top} - \E\int f\,d\mu_{\mathbf{A}_{\mathbf{Y}}\mathbf{A}_{\mathbf{Y}}^\top}\Big| + \E\Big|\int f\,d\mu_{\mathbf{A}_{\mathbf{Y}}\mathbf{A}_{\mathbf{Y}}^\top} - \int f\,d\nu_k\Big| \\
			&\quad + \Big|\int f\,d\nu_k - \int f\,d\nu_\infty\Big|,
		\end{align}
		where $\mathbf{X} = \mathbf{X}_k$ and $\mathbf{Y} = \mathbf{Y}_k$. The first term tends to zero in probability by Step~1. The second term tends to zero by Step~2 and bounded convergence, since $f$ is bounded and Lipschitz. The third tends to zero by Step~3. The same argument without the first term applies to $\mathbf{Y}$. Since $\nu_\infty$ is a probability measure and $C^2_c([0,\infty))$ is dense in $C_c([0,\infty))$, a standard argument shows that $d_{\mathrm{BL}}(\mu_{\mathbf{A}_{\mathbf{X}_k}\mathbf{A}_{\mathbf{X}_k}^\top},\nu_\infty) \to 0$ and $d_{\mathrm{BL}}(\mu_{\mathbf{A}_{\mathbf{Y}_k}\mathbf{A}_{\mathbf{Y}_k}^\top},\nu_\infty) \to 0$ in probability.
	\end{proof}
	
	\section*{Acknowledgments}
	
	The authors thank Sumit Mukherjee and Soumik Pal for helpful discussions and comments. HL was partially supported by NSF CAREER grant DMS-2541878.
	
	{\small
		\bibliographystyle{amsalpha}
		\bibliography{refs}
	}

\end{document}